\documentclass[preprint]{imsart}
\usepackage{a4wide}
\RequirePackage[OT1]{fontenc}
\usepackage{mathrsfs}
\usepackage{bbm, dsfont} 
\usepackage{amstext}
\usepackage{graphicx}
\usepackage{cancel,comment}
\RequirePackage{amsthm,amsmath,amssymb}
\RequirePackage[numbers]{natbib}
\RequirePackage[colorlinks,citecolor=blue,urlcolor=blue]{hyperref}

\startlocaldefs
\numberwithin{equation}{section}
\numberwithin{figure}{section}

\newtheorem{satz}{Satz}[section]
\newtheorem{thm}[satz]{Theorem}
\newtheorem{prop}[satz]{Proposition}
\newtheorem{cor}[satz]{Corollary}
\newtheorem{lem}[satz]{Lemma}
\theoremstyle{definition}

\newtheorem{remark}[satz]{Remark}
\newtheorem{remarks}[satz]{Remarks}
\newtheorem{example}[satz]{Example}
\newtheorem{examples}[satz]{Examples}

\newtheorem{ass}{Assumption}

\renewcommand{\hat}{\widehat}
\renewcommand{\tilde}{\widetilde}
\renewcommand{\bar}{\overline}

\renewcommand{\Re}{\operatorname{Re}}
\renewcommand{\Im}{\operatorname{Im}}
\newcommand{\diag}{\operatorname{diag}}

\endlocaldefs

\begin{document}

\global\long\def\phi{\varphi}
 \global\long\def\epsilon{\varepsilon}
 \global\long\def\eps{\varepsilon}
 \global\long\def\theta{\vartheta}
 \global\long\def\E{\mathbb{E}}
 \global\long\def\Var{\operatorname{Var}}
 \global\long\def\Cov{\operatorname{Cov}}
 \global\long\def\N{\mathbb{N}}
 \global\long\def\Z{\mathbb{Z}}
 \global\long\def\R{\mathbb{R}} 
 \global\long\def\C{\mathbb{C}}
 \global\long\def\F{\mathcal{F}}
 \global\long\def\le{\leqslant}
 \global\long\def\ge{\geqslant}
 \global\long\def\MT{\ensuremath{\clubsuit}}
 \global\long\def\1{\mathbbm1}
 \global\long\def\d{\mathrm{d}}
 \global\long\def\P{\mathbb{P}}
 \global\long\def\subset{\subseteq}
 \global\long\def\supset{\supseteq}
 \global\long\def\argmin{\operatorname*{arg\, min}}

\def\ind{\stackrel{d}{=}}
\global\long\def\rank{\operatorname{rank}}
 \global\long\def\arginf{\operatorname*{arg\, inf}}
 \global\long\def\bull{{\scriptstyle \bullet}}
 \global\long\def\supp{\operatorname{supp}}
 \global\long\def\sgn{\operatorname{sign}}
 \global\long\def\tr{\operatorname{tr}}
\global\long\def\spn{\operatorname{span}}
\global\long\def\img{\operatorname{Im}}
\global\long\def\Id{\operatorname{Id}}
\global\long\def\vec{\operatorname{vec}}

\begin{frontmatter}
\title{Low-rank diffusion matrix estimation for high-dimensional time-changed L\'evy processes} 
\runtitle{Low-rank diffusion matrix estimation for L\'evy processes}

\begin{aug}
\author{\fnms{Denis} \snm{Belomestny}\thanksref{a,b}\ead[label=e1]{denis.belomestny@uni-due.de}}
\and
\author{\fnms{Mathias} \snm{Trabs}\thanksref{c}\ead[label=e2]{mathias.trabs@uni-hamburg.de}}

\runauthor{D. Belomestny and M. Trabs}

\affiliation{Duisburg-Essen University, National Research University Higher School of Economics and University of Hamburg}

\address[a]{
Duisburg-Essen University, 
Faculty of Mathematics\\
Thea-Leymann-Str. 9
D-45127 Essen,
Germany\\
\phantom{E-mail: \ }
\printead{e1}}

\address[b]{
National Research University 
Higher School of Economics 
\\
Shabolovka, 26, 119049 Moscow, Russia
}


\address[c]{Universit\"at Hamburg,
Faculty of Mathematics\\
Bundesstra{\ss}e 55,
20146 Hamburg,
Germany \\
\phantom{E-mail: \ }
\printead{e2}}

\end{aug}

\begin{abstract}
The estimation of the diffusion matrix $\Sigma$ of  a high-dimensional, possibly time-changed L\'evy process is studied, based on discrete observations of the process with a fixed distance. A low-rank condition is imposed on $\Sigma$. Applying a spectral approach, we construct a weighted least-squares estimator with nuclear-norm-penalisation. We prove oracle inequalities and derive convergence rates for the diffusion matrix estimator. The convergence rates show a surprising dependency on the rank of $\Sigma$ and are optimal in the minimax sense for fixed dimensions. Theoretical results are illustrated by a simulation study.
\end{abstract}

\begin{abstract}[language=french]
In French.
\end{abstract}


\end{frontmatter}

\section{Introduction}

Non-parametric statistical inference for L{\'e}vy-type processes has been attracting the attention of researchers for many years initiated by  the works of \citet{RT} as well as \citet{BB}. The popularity of L\'evy processes is related to their simplicity on the one hand and the ability to reproduce many stylised patterns presented in  economic  and financial data or to model stochastic phenomena in biology or physics, on the other hand. While non-parametric inference for one dimensional L\'evy processes is nowadays well understood (see e.g. the lecture notes \cite{belomestny2015estimation}), there are surprisingly few results for multidimensional L\'evy or related jump processes. Possible applications demand however for multi- or even high-dimensional methods. As a first contribution to statistics for jump processes in high dimensions, we investigate the optimal estimation of the diffusion matrix of a $d$-dimensional L\'evy process where $d$ may grow  with the number of observations.

More general, we consider a larger class of time-changed L{\'e}vy process. Let $X$ be a $d$-dimensional L{\'e}vy process and let $\mathcal{T}$  be a non-negative, non-decreasing stochastic process with \( \mathcal{T}(0)=0 \). Then the time-changed L{\'e}vy process is defined via $Y_s = X_{\mathcal{T}(s)}$ for $s\ge0$. For instance, in the finance context, the change of time is motivated by the fact that some economical effects (as nervousness of the market which is indicated by volatility) can be better understood in terms of a  ``business time'' which may run faster than the physical one in some periods (see, e.g. \citet{VW}). In view of large portfolios and a huge number of assets traded at the stock markets we have high-dimensional observations.

\par
\citet{Monroe} has shown that even in the case of the Brownian motion $X$, the resulting class of time-changed L\'evy processes is rather large and basically coincides, at least in the one-dimensional case, with the class of all semi-martingales.  Nevertheless, a practical application of this fact for statistical modelling meets two major problems: first, the change of time $\mathcal{T}$ can be highly intricate - for instance, if $Y$ has discontinuous trajectories, cf. \citet{BNS}; second, the dependence structure between $X$ and $\mathcal{T}$ can be also quite involved. In order to avoid these difficulties and to achieve identifiability, we allow $X$ to be a general L{\'e}vy processes, but assume independence of $X$ and $\mathcal{T}$. 
\par
A first natural question is which parameters of the underlying L\'evy process \(X\) can be identified from discrete observations \(Y_{0},Y_{\Delta},\ldots,Y_{n\Delta}\) for some $\Delta>0$ as \(n\to \infty.\) This question has been recently addressed in the literature (see \citet{belomestny2011} and references therein) and the answer turns out to crucially depend  on the asymptotic behaviour of \(\Delta\) and on the degree of our knowledge about \(\mathcal{T}.\)  We study the so-called low-frequency regime, meaning that the observation distance $\Delta$ is fixed. In this case, we have severe identifiability problems which can be easily seen from the fact that subordinated L\'evy processes are again L\'evy processes. Assuming the time-change is known, all parameters of the underlying L\'evy process \(X\) can be identified as \(n\to \infty.\) Since any L\'evy process can be uniquely parametrised by the so-called L\'evy triplet \((\Sigma,\gamma,\nu)\) with a positive semi-definite diffusion (or volatility) matrix \(\Sigma\in \R^{d\times d}\), a drift vector \(\gamma\in \R^d\) and a  jump measure \(\nu,\) we face a semi-parametric estimation problem. Describing the covariance of the diffusion part of the L\'evy process \(X\), the matrix \(\Sigma\) is of particular interest. Aiming for a high dimensional setting, we study the estimation of $\Sigma$ under the assumption that it has low rank. This low rank condition reflects the idea of a few principal components driving the whole process. We also discuss the case where the diffusion matrix can be decomposed into a low rank matrix and a sparse matrix as has been suggested by \citet{fanEtAl2011,fanEtAl2013}.
\par
For discretely observed L\'evy processes it has been recently shown that estimation methods coming from low-frequency observations attain also in the  high-frequency case, where $\Delta\to0$, optimal convergence results, cf. \citet{jacodReiss2014} for volatility estimation and \citet{nicklEtAl2015} for the estimation of the jump measure. In contrast to estimators which are tailor-made for high-frequency observations, these low-frequency methods thus turn out to be robust with respect to the sampling frequency. This observation is a key motivation for considering a fixed $\Delta$ in the present article. It is very common in the literature on statistics for stochastic processes, that the observations are supposed to be equidistant. In view of the empirical results by \cite{aitSahaliaMykland2003}, this model assumption might however be often violated in financial applications. Our observation scheme can be equivalently interpreted as random observation times $(T(\Delta j))_j$ of the L\'evy process $X$. For one-dimensional L\'evy processes general observation distances as been studied by \citet{kappus2015}.  
\par
The research area on statistical inference for discretely observed stochastic processes is rapidly growing. Let us only mention some closely related contributions while we refer to \cite{belomestny2015estimation} for a recent overview of estimation for L\'evy processes. Non-parametric estimation for time-changed L\'evy models has been studied in \cite{belomestny2011,belomestnyPanov2013,bull2014}. In a two-dimensional setting the jump measure of a L\'evy process has been estimated in \cite{buecherVetter2013}. For inference on the volatility matrix (in small dimensions) for continuous semi-martingales in a high-frequency regime we refer to \cite{jacodPodolskij2013, bibingerEtAl2014} and references therein. Large sparse volatility matrix estimation for continuous It\^o processes has been recently studied in \cite{wangZou2010,taoEtAl2011,taoEtAl2013}. 
\par
Our statistical problem turns out to be closely related to a deconvolution problem for a multidimensional normal distribution with zero mean and covariance matrix \(\Sigma\)  convolved with a nuisance distribution which  is  unknown except for some of its structural properties.  Due to the time-change or random sampling, respectively, the normal vector is additionally multiplied with a non-negative random variable. Hence, we face a covariance estimation problem in a generalised mixture model. Since we have no direct access to a sample of the underlying normal distribution, our situation differs considerably from the inference problems which have been studied in the literature on high-dimensional covariance matrix estimation so far.  

The analysis of  many deconvolution and mixture models becomes more transparent in spectral domain. Our accordingly reformulated problem bears some similarity to the so-called trace regression problem and the related matrix estimation, which recently got a lot of attention in statistical literature (see, e.g. \cite{rohdeTsybakov2011,negahbanWainwright2011,agarwal2012noisy,koltchinskii2011nuclear}). We face a non-linear analogue to the estimation of low-rank matrices based on noisy observations. Adapting some ideas, especially from \citet{koltchinskii2011nuclear}, we construct a weighted least squares estimator with nuclear norm penalisation in the spectral domain. 
\par
We prove oracle inequalities for the estimator of $\Sigma$ implying that the estimator adapts to the low-rank condition on $\Sigma$ and that it also applies in the miss-specified case of only approximated sparsity. The resulting convergence rate fundamentally depends on the time-change while the dimension may grow polynomially in the sample size. Lower bounds verify that the rates are optimal in the minimax sense for fixed dimensions. The influence of the rank of $\Sigma$ on the convergence rate reveals a new phenomenon which was not observed in multi-dimensional non-parametric statistics before. Namely, in a certain regime the convergence rate in $n$ is faster for a larger rank of $\Sigma$. To understand this surprising phase transition from a more abstract perspective, we briefly discuss a related regression model.
\par
The paper is organised as follows. In Section~\ref{main_setup}, we introduce notations and formulate the statistical model.
In Section~\ref{seq: estimator} the estimator for the diffusion matrix is constructed and the oracle inequalities for this estimator are derived. Based on these oracle inequalities, we derive in Section~\ref{seq: mcr} upper and lower bounds for the convergence rate. Section~\ref{seq:ext} is devoted to several extensions, including mixing time-changes, incorporating a sparse component of the diffusion matrix and the related trace-regression model. Some numerical examples can be found in Section~\ref{eq:sim}. In Sections \ref{seq:proof_oracle} and \ref{seq:proof_conv} the proofs of the oracle inequalities and of the minimax convergence rates, respectively, are collected. The appendix contains some (uniform) concentration results for multivariate characteristic functions which are of independent interest.

\vspace*{1em}
\noindent\emph{Acknowledgement.} The authors thank Markus Rei\ss{} and two anonymous referees for helpful comments. D.B. acknowledges the financial support from the Russian Academic Excellence Project ``5-100'' and from the Deutsche Forschungsgemeinschaft (DFG) through the SFB 823 ``Statistical modelling of nonlinear dynamic processes''. M.T. gratefully acknowledges the financial support by the DFG research fellowship TR 1349/1-1. A part of this paper has been written while M.T. was affiliated to the Universit\'e Paris-Dauphine.

\section{Main setup}
\label{main_setup}
Recall that a random time change $\{\mathcal{T}(t):t\ge0\}$ is an increasing right-continuous process with left limits such that
$\mathcal{T}(0)=0$. For each $t$ the random variable $\mathcal{T}(t)$
is a stopping time with respect to the underlying filtration. For a L\'evy process $X=\{X_{t}:t\ge0\}$ the time-changed
L\'evy process is given by $ X_{\mathcal{T}(t)},t\ge0$. We throughout assume that $\mathcal{T}$ is independent of $X$. 
By reparametrising the time change we can assume without loss of generality that $\Delta=1$ and that the observations are thus given by the increments 
\[
Y_{j}:=X_{\mathcal{T}(j)}-X_{\mathcal{T}(j-1)},\quad j=1,\dots,n,
\]
for $n\in\N$. Note that in general the observations are not independent, in contrast to the special case of low-frequently observed L\'evy processes. 
The estimation procedure relies on the following crucial insight: If the sequence 
\[
T_{j}:=\mathcal{T}(j)-\mathcal{T}(j-1),\quad j=1,\dots,n,
\]
is stationary and admits an invariant measure $\pi$, then the independence of $\mathcal{T}$
and $X$ together with the L\'evy-Khintchine formula yield that the
observations $Y_{j}$ have a common characteristic function given by
\begin{align*}
\phi(u) & :=\E\big[e^{i\langle Y_{j},u\rangle}\big]=\int_{0}^{\infty}\E[e^{i\langle X_{t},u\rangle}]\pi(\d t)=\int_{0}^{\infty}e^{t\psi(u)}\pi(\d t)=\mathscr{L}(-\psi(u)),\quad u\in\R^d,
\end{align*}
where $\mathscr{L}$ is the Laplace transform of $\pi$ and where the characteristic exponent is given by 
\begin{align}
\psi(u)&=-\frac{1}{2}\langle u,\Sigma u\rangle+i\langle \gamma,u\rangle+\int_{\R^{d}}\big(e^{i\langle x,u\rangle}-1-i\langle x,u\rangle\1_{\{|x|\le1\}}(x)\big)\d\nu(x)\notag\\
&=:-\frac{1}{2}\langle u,\Sigma u\rangle+\Psi(u),\quad u\in\R^{d},\label{eq:psi}
\end{align}
with the diffusion matrix $\Sigma\in\R^{d\times d}$, for a drift parameter $\gamma\in\R^d$ and a jump measure $\nu$. If $\gamma=0$ and $\nu=0$, we end up with the problem of estimating a covariance
matrix. The function $\Psi$ in \eqref{eq:psi} appears due to the
presence of jumps and can be viewed as a nuisance parameter. Let us give some examples of typical time-changes and their Laplace transforms.
\begin{examples}\label{ex:laplace}\hspace{1em}
    \begin{enumerate}
    \item \emph{Low-frequency observations} of $X_t$ with observation distance $\Delta>0$:
    \[
      T_j\sim\pi=\delta_\Delta,\quad \mathscr L(z)=e^{-\Delta z}.
    \]
    \item \emph{Poisson process} time-change or exponential waiting times with intensity parameter $\Delta>0$:
    \[
      T_j\sim\pi= Exp(\Delta),\quad \mathscr L(z)=\frac{1}{1+\Delta z}.
    \]
    \item \emph{Gamma process} time-change with parameters $\Delta,\theta>0$:
    \[
      T_j\sim\pi=\Gamma(\theta,\Delta),\quad \mathscr L(z)=(1+\Delta z)^{-\theta}.
    \]
    \item \emph{Integrated CIR-process} for some $\eta,\kappa,\xi>0$ such that $2\kappa\eta>\xi^2$ (which has $\alpha$-mixing increments, cf. \cite{masuda2007}):
    \begin{align*}
      \mathcal T(t)&=\int_0^tZ_t\d t\quad\text{with}\quad\d Z_t=\kappa(\eta-Z_t)\d t+\xi \sqrt{Z_t}\d W_t,\\
      \mathscr L(z)&\sim \exp\big(-\frac{\sqrt{2z}}{\xi}(1+\kappa\eta)\big)\text{ as }|z|\to\infty\text{ with }\Re z\ge0.
    \end{align*}
  \end{enumerate}
\end{examples}

If the time change $\mathcal T$ is unknown, one faces severe identifiability problems. Indeed, even in the case of a multivariate time-changed Brownian motion \(\Sigma^{\top}W_{\mathcal{T}(t)}\) we can identify the matrix \(\Sigma\) and \(\mathcal{T}(t)\) only up to a multiplicative factor. More generally, if $X$ is not restricted to be Brownian motion there is much more ambiguity as the following example illustrates.
\begin{example}
  Consider two time changes \( \mathcal{T}(t)\) and \( \widetilde{\mathcal{T}}(t), \) where  \( \mathcal{T}(t)\) is a Gamma subordinator with the unit parameters and \( \widetilde{\mathcal{T}}(t)=t \) is deterministic. We have
\( \mathcal{L}_{t}(z)=1/(1+z)^{t} \) and \( \widetilde{\mathcal{L}}_{t}(z)=e^{-tz}. \) Note that \( \E [\mathcal{T}(t)]=\E [\widetilde{\mathcal{T}}(t)]=t. \) Let \( (X_{t}) \) be a one-dimensional L\'evy process with the characteristic exponent \( \psi(u) \) and let \( (\widetilde{X}_{t}) \) be another L\'evy process with the characteristic exponent \( \widetilde \psi(u)=\log(1/(1-\psi(u))) \) (the existence of such L\'evy process follows from the well-known fact that  \( \exp(\widetilde \psi(u)) \)  is the characteristic function of some infinitely divisible distribution). Then we obviously have \( \mathcal{L}_{t}(-\psi(u))\equiv \widetilde{\mathcal{L}}_{t}(-\widetilde{\psi}(u)) \) for all \( t\geq 0. \) Hence, the corresponding time-changed L\'evy processes \( Y_{t}=X_{\mathcal{T}(t)} \) and \( \widetilde{Y}_{t}=\widetilde{X}_{\widetilde{\mathcal{T}}(t)} \) satisfy 
\begin{align*}
    Y_{k\Delta}-Y_{(k-1)\Delta}\ind\widetilde{Y}_{k\Delta}-\widetilde{Y}_{(k-1)\Delta}, \quad k\in \mathbb{N},
\end{align*}    
for any \( \Delta>0. \) Moreover, the increments \( Y_{k\Delta}-Y_{(k-1)\Delta} \) are  independent in this case.
\end{example}
The above example shows that even under the additional assumption $\E [\mathcal{T}(t)]=t$, we cannot, in general, consistently estimate the parameters of the one-dimensional time-changed L\'evy process $(Y_{t})$ from the low-frequency observations. As shown by \citet{belomestny2011} all parameters of a \(d\)-dimensional  (with \(d>1\))  L\'evy process \(X\) can be identified under additional assumption that all components of \(X\) are independent and \(\E [\mathcal{T}(t)]=t.\) However this assumption would imply that the matrix \(\Sigma\) is diagonal and is thus much too restrictive for our setting. Therefore, we throughout assume that the time-change and consequently its Laplace transform are known, see Section~\ref{sec:lapl_est} for a discussion for unknown $\mathscr  L$. Let us note that similar  identifications problems appear even in high-frequency setting, see \cite{figueroa2009nonparametric}.
\vspace*{1em}

Before we construct the volatility matrix estimator in the next section, we should introduce some notation. Recall the Frobenius or trace inner
product $\langle A,B\rangle:=\tr(A^{\top}B)$ for matrices $A,B\in\R^{d\times d}$. For $p\in(0,\infty],$ the Schatten-$p$ norm of $A$ is given by
\[
  \|A\|_{p}:=\Big(\sum_{i=1}^{d}\sigma_{j}(A)^p\Big)^{1/p}
\]
with $\sigma_{1}(A)\geq\ldots\geq\sigma_{d}(A)$ being the singular values of $A$. In particular, $\|A\|_1$, $\|A\|_2$ and $\|A\|_\infty$ denote the nuclear, the Frobenius and the spectral norm of $A$, respectively. We will frequently apply the trace duality property 
\[
|\tr(AB)|\le\|A\|_{1}\|B\|_{\infty},\quad A,B\in\mathbb{R}^{d\times d}.
\]
Moreover, we may consider the entry-wise norms $|A|_p:=(\sum_{i,j}|a_i,j|^p)^{1/p}$ for $0<p\le\infty$ with the usual modification $|A|_0$ denoting the number of non-zero entries of $A=(a_{i,j})_{i,j}$.

For any matrices $A\in\R^{k\times k},B\in\R^{l\times l},k,l\in\N,$ we write
$$\diag(A,B):=\begin{pmatrix}A & 0\\0 & B\end{pmatrix}\in\R^{(k+l)\times(k+l)}.$$
For $a,b\in\R$ we write $a\lesssim b$ if there is a constant $C$ independent of $n, d$ and the involved parameters such that $a\le Cb$.

\section{The estimator and oracle inequalities}
\label{seq: estimator}

In view of the L\'evy-Khintchine formula, we apply a spectral approach. The natural estimator of $\phi$ is given by the empirical characteristic
function 
\[
\phi_{n}(u):=\frac{1}{n}\sum_{j=1}^{n}e^{i\langle u,Y_{j}\rangle},\quad u\in\R^d,
\]
which is consistent whenever $(Y_{j})_{j\ge0}$ is ergodic. Even more, it concentrates around the true characteristic function with parametric rate uniformly on compact sets, cf. Theorems~\ref{thm:concPhi} and \ref{thm:concentration} in the appendix. A plug-in approach yields the estimator
for the characteristic exponent given by 
\begin{equation}
\hat{\psi}_{n}(u):=-\mathscr{L}^{-1}(\phi_{n}(u)),\label{eq:PsiHat}
\end{equation}
where $\mathscr{L}^{-1}$ denotes a continuously chosen inverse of the map
$\{\Re z>0\}\ni z\mapsto\mathscr{L}(v)\in\C\setminus\{0\}$. Since $\Sigma$ only appears in the real part of the characteristic
exponent, we may use $\Re(\hat{\psi}_{n})$. 
\begin{remark}
  For high-frequency observations, that is $\Delta=\Delta_n\to0$, the characteristic function of the observations $Y_j$ can be written as $\phi_\Delta(u)=\mathscr L(-\Delta \psi(u))$ where $\psi$ is the characteristic exponent as before and where $\mathscr L$ is the Laplace transform of the rescaled time-change increments $\frac{\mathcal T(\Delta j)-\mathcal T(\Delta(j-1))}\Delta$, which may be assumed to be stationary. Under some regularity assumption on $\mathscr L$ , we obtain the asymptotic expansion
  \[
    \phi_\Delta(u)=\mathscr L(-\Delta\psi(u))=1-\mathscr L'(0)\Delta\psi(u)+\mathcal O\big(\Delta^2(|u|^4\vee 1)\big)\quad\text{for}\quad\Delta\to0.  
\]
Instead of the inversion in \eqref{eq:PsiHat} the estimator $\bar\psi_n(u)=\frac{1-\phi_n(u)}{\Delta\mathscr L'(0)}$ can be used in this case, depending only on Laplace transform via $\mathscr L'(0)$. Note that the latter is given by the first moment of rescaled time-change increments. However, this asymptotic identification can only be applied in the high-frequency regime $\Delta\to0$, while the estimator \eqref{eq:PsiHat} is robust with respect to the sampling frequency.
\end{remark}

Applying $u^\top\Sigma u=\langle uu^\top,\Sigma\rangle$, estimating
the low-rank matrix $\Sigma$ can thus be reformulated as the regression problem
\begin{align}
  \frac{\hat\psi_{n}(u)}{|u|^{2}}
  =  -\frac{1}{2}\langle \Theta(u),\Sigma\rangle+\frac{\Psi(u)}{|u|^2}+\frac{\hat \psi_{n}(u)-\psi(u)}{|u|^2}\quad\text{with}\quad\Theta(u):=\frac{uu^\top}{|u|^2}\in\R^{d\times d},\label{eq:regression}
\end{align}
where we have normalised the whole formula by the factor $|u|^{-2}$.
The design matrix $vv^{\top}$ for an arbitrary unit vector $v\in\R^{d}$ is deterministic and degenerated. The second term in \eqref{eq:regression} is a deterministic error which will be small only for large $u$. The last term reflects the stochastic error. Due to the identification via the Laplace transform $\mathscr L$, the estimation problem is non-linear and turns out to be ill-posed: the stochastic error grows for large frequencies. 

Inspired by the  regression formula \eqref{eq:regression}, we define the diffusion matrix estimator as a penalised least squares type estimator with regularisation parameter $\lambda>0$ and a spectral cut-off $U>1$:
\begin{equation}
\hat{\Sigma}_{n,\lambda}:=\argmin_{M\in\mathbb{M}}\left\{ \int_{\mathbb{R}^{d}}\Big(2|u|^{-2}\Re\hat{\psi}_{n}(u)+\langle\Theta(u),M\rangle\Big)^{2}w_{U}(u)\,\d u+\lambda\|M\|_{1}\right\} \label{eq:sigmaHat}
\end{equation}
where $\mathbb{M}$ is a subset of positive semi-definite $d\times d$ matrices and $w_U$ is a weight function. We impose the following standing assumption on the weight function which is chosen by the practitioner.
\begin{ass}\label{ass:weight}
  Let $w\colon\R^d\to\R_+$ be a radial non-negative function which is supported on the annulus $\{1/4<|u|\le1/2\}$. For any $U>1$ let $w_{U}(u)=U^{-d}w(u/U),u\in\R^d$. 
\end{ass}

In the special case of a L\'evy process with a finite jump activity $\alpha:=\nu(\R^d)\in(0,\infty)$, we can write remainder from \eqref{eq:psi} as
\[
  \Psi(u)=i\langle \gamma_0,u\rangle+\F\nu(u)-\alpha,\quad u\in\R^d,
\]
for $\gamma_0:=\gamma-\int_{|u|\le 1} x\nu(\d x)$. Since the Fourier transform $\F\nu(u)$ converges to zero as $|u|\to\infty$ under a mild regularity assumption on the finite measure $\nu$, we can reduce the bias of our estimator by the following modification
\begin{align}
(\tilde{\Sigma}_{n,\lambda},\tilde \alpha_{n,\lambda})
&:=\argmin_{M\in\mathbb{M},a\in I}\left\{ \int_{\mathbb{R}^{d}}\Big(\frac2{|u|^{2}}\Re(\hat{\psi}_{n}(u)+a)+\langle\Theta(u),M\rangle\Big)^{2}w_{U}(u)\,\d u+\lambda\big(\|M\|_{1}+\frac a{U^{2}}\big)\right\} \notag\\
&=\argmin_{M\in\mathbb{M},a\in I}\left\{ \int_{\mathbb{R}^{d}}\Big(\frac2{|u|^{2}}\Re\hat{\psi}_{n}(u)+\langle\tilde{\Theta}(u),\diag(M,\frac a{U^{2}})\rangle\Big)^{2}w_{U}(u)\,\d u+\lambda(\|M\|_{1}+\frac a{U^{2}})\right\}\label{eq:sigmaTilde}
\end{align}
with
\[
  \tilde{\Theta}(u):=\tilde\Theta_U(u)=\diag\Big(\Theta(u),\frac {2U^2}{|u|^2}\Big),\quad u\in\R^d,
\]
and some interval $I\subset\R_+$. The most interesting cases are $I=\{0\}$, where $\hat\Sigma_{n,\lambda}$ and $\tilde\Sigma_{n,\lambda}$ coincide, and $I=[0,\infty)$. The factor $U^{-2}$ in front of $a$ is natural, since the ill-posedness of the estimation problem for the jump activity is two degrees larger than for estimating the volatility, cf. \citet{belomestnyReiss2006}. As a side product, $\tilde \alpha_{n,\lambda}$ is an estimator for the jump intensity. By penalising large $a$, the estimator $\tilde\alpha_{n,\lambda}$ is pushed back to zero if the least squares part cannot profit from a finite $a$. It thus coincides with the convention to set $\alpha=0$ in the infinite intensity case.

\vspace*{1em}

Our estimators $\hat\Sigma_{n,\lambda}$ and $\tilde\Sigma_{n,\lambda}$ are related to the weighted scalar product 
\begin{align*}
\big\langle (A,a),(B,b)\big\rangle_{w}
&:=\int_{\mathbb{R}^{d}}\langle\tilde\Theta(u),\diag(A,a)\rangle\langle\tilde\Theta(u),\diag(B,b)\rangle w_{U}(u)\,\d u\\
&=\int_{\mathbb{R}^{d}}\Big(\langle\Theta(u),A\rangle+\frac{2U^2}{|u|^2}a\Big)\Big(\langle\Theta(u),B\rangle+\frac{2U^2}{|u|^2}b\Big) w_{U}(u)\d u
\end{align*}
with the usual notation $\|(A,a)\|_w^2:=\langle (A,a),(A,a)\rangle_w$. For convenience we abbreviate $\|A\|_w:=\|(A,0)\|_w$. The scalar product $\langle\cdot,\cdot\rangle_w$ is the counterpart to the empirical scalar product in the matrix estimation literature. As the following lemma reveals, the weighted norm $\|(A,a)\|_w$ 
and the Frobenius norm $\|\diag(A,a)\|_2=(\|A\|_2^2+a^2)^{1/2}$ are equivalent. As the isometry is not restricted to any sparse sub-class of matrices, it is especially implies the often imposed restricted isometry property. The proof is postponed to Section~\ref{sec:ProofRip}.
\begin{lem}\label{lem:rip}
  For any positive semi-definite matrix $A\in\R^{d\times d}$ and any $a\ge0,$ it holds
  \begin{equation}
    \underline\varkappa_w\|\diag(A,a)\|_2\le\|(A,a)\|_{w} \le \bar\varkappa_w\|\diag(A,a)\|_2,  \label{eq:ineqNorms}
  \end{equation} 
  where $\underline\varkappa_w^2:=\int_{\R^d}v_{1}^{4}/(|v|^{4})w(v)\d v$ and $\bar\varkappa_w^2:=8\int_{\R^d}|v|^{-4}w(v)\d v$.
\end{lem}
Using well-known calculations from the Lasso literature, we obtain the following elementary oracle inequality, which is proven in Section~\ref{sec:ProofOracle1}.  The condition $\alpha\in I$ is trivially satisfied for $I=\R_+$.
\begin{prop}\label{prop:oracle1}
  Let $\mathbb M\subset\R^{d\times d}$ be an arbitrary subset of matrices and define
  \begin{equation}
  \mathcal{R}_n:=2\int_{\mathbb{R}^{d}}\Big(\frac2{|u|^2}\Re\hat{\psi}_{n}(u)-\langle\tilde\Theta(u),\diag(\Sigma,\frac\alpha{U^2})\rangle\Big)\tilde\Theta(u) w_{U}(u)\,\d u\in\R^{(d+1)\times (d+1)},\label{eq:error}
  \end{equation}
  where we set $\alpha=0$ if $\nu(\R^d)=\infty$. Suppose $\alpha\in I$. On the event $\{\|\mathcal R_{n}\|_\infty\le\lambda\}$ for some $\lambda>0$ we have
  \begin{equation}
    \|(\tilde{\Sigma}_{n,\lambda}-\Sigma,U^{-2}(\tilde\alpha_{n,\lambda}-\alpha))\|_{w}^{2}\le\inf_{M\in\mathbb{M}}\left\{\|M-\Sigma\|_{w}^{2}+2\lambda(\|M\|_{1}+U^{-2}\alpha)\right\}.\label{ineq1}
  \end{equation}
\end{prop} 

If the set $\mathbb{M}$ is convex, then sub-differential calculus can be used to refine the inequality \eqref{ineq1}. The proof is inspired by \citet[Thm. 1]{koltchinskii2011nuclear} and postponed to Section~\ref{sec:proofOracle}. Note that this second oracle inequality improves \eqref{ineq1} with respect to two aspects. Instead of $\lambda$ we have $\lambda^2$ in the second term on the right-hand side and the nuclear-norm of $M$ is replaced by its rank.
\begin{thm}\label{thm:oracle2}
  Suppose that $\mathbb{M}\subset\R^{d\times d}$ is a convex subset of the positive semi-definite matrices and let $\alpha\in I$. On the event $\{\|\mathcal R_n\|_\infty\le \lambda\}$ for some $\lambda>0$ and for $\mathcal R_n$ from \eqref{eq:error} the estimators $\tilde\Sigma_{n,\lambda}$ and $\tilde\alpha_{n,\lambda}$ from \eqref{eq:sigmaTilde} satisfy
  \begin{align*}
  &\big\|\big(\tilde{\Sigma}_{n,\lambda}-\Sigma,U^{-2}(\tilde\alpha_{n,\lambda}-\alpha)\big)\big\|_{w}^{2}
  \le\inf_{M\in\mathbb M}\left\{\|M-\Sigma\|_{w}^{2}+(\tfrac{1+\sqrt 2}{2\underline\varkappa_w})^2\lambda^2\big(\operatorname{rank}(M)+\1_{\alpha\neq0}\big)\right\}
  \end{align*}
  where $\underline\varkappa^2_w:=\int_{\R^d}v_{1}^{4}/(2|v|^{4})w(v)\d v$.
\end{thm} 

This oracle inequality is purely non-asymptotic and sharp in the sense that we have the constant one in front of $\|M-\Sigma\|_w^2$ on the right-hand side. Combining Lemma~\ref{lem:rip} and Theorem~\ref{thm:oracle2} immediately yields an oracle inequality for the error in the Frobenius norm. 
\begin{cor}\label{cor:FrobError}
   Let $\mathbb{M}\subset\R^{d\times d}$ be a convex subset of the positive semi-definite matrices and let $\alpha\in I$. On the event $\{\|\mathcal R_n\|_\infty\le \lambda\}$ for some $\lambda>0$ and for $\mathcal R_n$ from \eqref{eq:error} we have
  \[
  \|\tilde{\Sigma}_{n,\lambda}-\Sigma\|^2_{2}+U^{-4}|\tilde\alpha_{n,\lambda}-\alpha|^2
  \le C_w\inf_{M\in\mathbb M}\Big\{\|M-\Sigma\|_2^{2}+\lambda^2\big(\operatorname{rank}(M)+\1_{\alpha\neq0}\big)\Big\}
  \]
  for a constant $C_w$ depending only on $\underline\varkappa_w$ and $\bar\varkappa_w$. 
\end{cor}
The remaining question is how the parameters $U$ and $\lambda$ should be chosen in order to control the event $\{\|\mathcal R_n\|_\infty\le\lambda\}$. The answer will be given in the next section.

\section{Convergence rates}
\label{seq: mcr}
The above oracle inequalities hold true for any L\'evy-process $X$ and any independent time-change $\mathcal T$ with stationary increments. However, to find some $\lambda$ and $U$ such that probability of the event $\{\|\mathcal R_n\|_\infty\le\lambda\}$ with the error term $\mathcal R_n$ from \eqref{eq:error} is indeed large, we naturally need to impose some assumptions. For simplicity we will concentrate on the well-specified case $\Sigma\in\mathbb M$ noting that, thanks to Corollary~\ref{cor:FrobError}, the results carry over to the miss-specified situation.
\begin{ass}\label{ass:jumps} 
  Let the jump measure $\nu$ of $X$ fulfil: 
  \begin{enumerate} 
    \item for some $s\in(-2,\infty)$ and for some constant $C_{\nu}>0$ let
      \begin{align}
      \alpha:=0\quad\text{and}&\quad\sup_{|h|=1}\int_{\mathbb{R}^{d}}\left|\left\langle x,h\right\rangle \right|^{|s|}\nu(dx)\le C_{\nu},&\quad\text{if }s<0,\label{BG}\\
      \alpha:=\nu(\R^d)\in(0,\infty)\quad\text{and}&\quad |\F\nu(u)|^2\le C_\nu(1+|u|^2)^{-s},u\in\R^d,&\quad\text{if }s\ge0.\notag
      \end{align}
    \item $\int_{\R^{d}}|x|^{p}\nu(\d x)<\infty$ for some $p>0$. 
  \end{enumerate}
\end{ass} 
\begin{ass}\label{ass:time} 
  Let the time-change $\mathcal{T}$ satisfy: 
  \begin{enumerate} 
    \item $\E[\mathcal{T}^{p}(1)]<\infty$ for some $p>0$. 
    \item The sequence $T_{j}=\mathcal{T}(j)-\mathcal{T}(j-1)$,
    $j\in\mathbb{N},$ is mutually independent and identically distributed with some law $\pi$ on $\R_+$.
    \item The derivatives of the Laplace transform $\mathscr{L}(z)=\int_{0}^{\infty}e^{-tz}\pi(\d t)$
    satisfy $|\mathscr{L}''(z)/\mathscr{L}'(z)|\le C_L$ for some $C_L>0$ for all $z\in\C$ with $\Re(z)>0$. 
  \end{enumerate}
  If the Laplace transform decays polynomially, we may impose the stronger assumption
  \begin{enumerate}
    \item[(iv)] $|\mathscr{L}''(z)/\mathscr{L}'(z)|\le C_L(1+|z|)^{-1}$
    for some $C_L>0$ for all $z\in\C$ with $\Re(z)>0$.
  \end{enumerate}
\end{ass} 

\begin{remarks}\hspace*{1em}
  \begin{enumerate}
   \item For $s\in(-2,0)$ Assumption~\ref{ass:jumps}(i) allows for L\'evy processes with infinite jump activity. In that case $|s|$ in \eqref{BG} is an upper bound for the Blumenthal--Getoor
index of the process. A more pronounced singularity of the jump measure $\nu$ at zero corresponds to larger values of $|s|$. The lower bound $s>-2$ is natural, since any jump measure satisfies $\int_{\R^d}(|x|^2\wedge1)\nu(\d x)$. On the contrary, $s=0$ in \eqref{BG} implies that $\nu$ is a finite measure. In that case we could further profit from its regularity which we measure by the decay of its Fourier transform. Altogether, the parameter $s$ will determine the approximation error that is due to the term $|u|^{-2}\Psi(u)$ in \eqref{eq:regression}. 
   \item Assumption~\ref{ass:jumps}(ii) implies that $\E[|X_{t}|^{p}]<\infty$ for all $t>0$ and together with the moment condition Assumption~\ref{ass:time}(i), we conclude that $\E[|Y_{k}|^{p}]<\infty$, cf. Lemma~\ref{lem:moments}. 
Assumption~\ref{ass:time}(ii) implies that the increments $Y_j=X_{\mathcal T(j)}-X_{\mathcal T(j-1)}$ are independent and identically distributed. Note that $(T_{j})$ being stationary with invariant measure $\pi$ is necessary to construct the estimator of $\Sigma$. The independence can, however, be relaxed to an $\alpha$-mixing condition as discussed in Section~\ref{sec:mixing}. Finally, Assumptions~\ref{ass:time}(iii) and (iv) are used to linearise the stochastic error.
  \end{enumerate}
\end{remarks}

In the sequel we denote by $B^d_U:=\{u\in\R^d:|u|\le U\}$ the $d$-dimensional ball of radius $U>0$. We can now state the first main result of this section. 

\begin{thm}\label{thm:conR}
Grant Assumptions~\ref{ass:jumps} and \ref{ass:time}(i)-(iii). If also Assumption~\ref{ass:time}(iv) is fulfilled, we set $q=1$ and otherwise let $q=0$. Then for any $n,d\in\N$ and $U\ge1,\kappa>0$ satisfying
\[
  \sqrt {d\log(d+1)}(\log U)\le\kappa\le\sqrt nU^{-d/2}\|\phi\|_{L^1(B_U^d)}^{1/2}\frac{\inf_{|u|\le U}|\psi(u)|^q|\mathscr L'(-\psi(u))|^2}{\inf_{|u|\le U}|\mathscr L'(-\psi(u))|},
\]
the matrix $\mathcal R_n$ from \eqref{eq:error} satisfies for some constants $c,D>0$ depending only on $w, C_\nu$ and $C_L$
\[
\P\Big(\|\mathcal R_{n}\|_\infty\geq \frac{\kappa \|\phi\|_{L^1(B_U^d)}^{1/2}}{\sqrt nU^{2+d/2}\inf_{|u|\le U}|\mathscr L'(-\psi(u))|}+D U^{-(s+2)}\Big)\le 2(d+1)e^{-c\kappa^2}.
\]
In particular, $\P(\|\mathcal R_n\|_\infty>\lambda)\le 2(d+1)e^{-c\kappa^2}$ if 
\[
  \lambda\ge\frac{\kappa \|\phi\|_{L^1(B_U^d)}^{1/2}}{\sqrt nU^{2+d/2}\inf_{|u|\le U}|\mathscr L'(-\psi(u))|}+D U^{-(s+2)}.
\]
\end{thm} 

In order to prove this theorem, we decompose $\mathcal R_n$ into a stochastic error term and a deterministic approximation error. More precisely, the regression formula~\eqref{eq:regression} and $\|\tilde\Theta(u)\|_\infty=\max\{1,2U^2|u|^{-2}\}\le 1$ for any $u\in\supp w_U$ yield
\begin{align}
  \|\mathcal{R}_{n}\|_\infty
  \le&4\Big\|\int_{\R^{d}}|u|^{-2}\Re\big(\mathscr{L}^{-1}(\phi_{n}(u))-\mathscr{L}^{-1}(\phi(u))\big)\tilde\Theta(u) w_{U}(u)\d u\Big\|_\infty\label{eq:decomp}\\
  &\qquad+4\int_{\R^{d}}\frac{|\Re\Psi(u)+\alpha|}{|u|^2}w_{U}(u)\d u.\notag
\end{align}
The order of the deterministic error is $\mathcal O(U^{-s-2})$, cf. Lemma~\ref{lem:ApproxError}, which decays as $U\to \infty$. The rate deteriorates for higher jump activities. This is reasonable since even for high frequency observations it is very difficult to distinguish between small jumps and fluctuations due to the diffusion component, cf. \citet{jacodReiss2014}.
The stochastic error is dominated by its linearisation 
\[
  \Big\|\int_{\R^{d}}\frac{1}{|u|^2}\Re\Big(\frac{\phi_{n}(u)-\phi(u)}{\mathscr{L}'(-\psi(u))}\Big)\tilde\Theta(u) w_{U}(u)\d u\Big\|_\infty,
\]
which is increasing in $U$ due to the denominator $|u|^2\mathscr{L}'(-\psi(u))\to0$ as $|u|\to\infty$. To obtain a sharp oracle inequality for the spectral norm of the linearised stochastic error, we use the noncommutative Bernstein inequality by \citet{recht2011}. To bound the remainder, we apply a concentration result (Theorem~\ref{thm:concentration}) for the empirical characteristic function around $\phi$, uniformly on $B_U^d$. 

The lower bound on $\kappa$ reflects the typical dependence on the dimension $d$ that also appear is in theory on matrix completion, cf. Corollary~2 by \citet{koltchinskii2011nuclear}. Our upper bound on $\kappa$ ensures that the remainder term in the stochastic error is negligible. 

The choice of $U$ is determined by the trade-off between approximation error and stochastic error. Since neither $U$ nor $\lambda$ depend on the rank of $\Sigma$, the theorem verifies that the estimator $\tilde\Sigma_{n,\lambda}$ is adaptive with respect to $\rank(\Sigma)$. Supposing lower bounds for $\mathscr L'(-\psi(u))$ that depend only on the radius $|u|$, the only term that depends on the dimension is $E_U:=U^{-d}\|\phi\|_{L^1(B_U^d)}$. Owing to $\|\phi\|_\infty\le1$, it is uniformly bounded by the volume of the unit ball \(B_1^d\) in $\R^d$ which in turn is uniformly bounded in $d$ (in fact it is decreasing as $d\to\infty$). If $\phi\in L^1(\R^d)$, we even have $E_U\lesssim U^{-d}$. We discuss this quite surprising factor further after Corollary~\ref{cor:mildly} and, from a more abstract perspective, in Section~\ref{sec:reg}.

\vspace*{1em}

For specific decay behaviours of $\mathscr L(z)$ we can now conclude convergence rates for the diffusion matrix estimator. In contrast to the nonparametric estimation of the coefficients of a diffusion process, as studied by \citet{chorowskiTrabs2015}, the convergence rates depend enormously on the sampling distribution (resp. time-change). We start with an exponential decay of the Laplace transform of $\pi$. This is especially the case if the L\'evy process $X$ is observed at equidistant time points $\Delta j$ for some $\Delta>0$ and thus $\mathscr L(v)=e^{\Delta v}$. 

\begin{cor}\label{cor:ExpDec}
  Grant Assumptions~\ref{ass:jumps} and \ref{ass:time}(i)-(iii) and let $d=o(n/(2\log\log n))$. Suppose that $0\neq\Sigma\in\mathbb{M}$,
  where $\mathbb{M}$ is a convex subset of positive semi-definite $d\times d$ matrices and let $\alpha\in I$. If $|\mathscr{L}'(z)|\gtrsim\exp(-a|z|^\eta)$ for $z\in\mathbb C$ with $\Re(z)>0$ and for some $a,\eta>0$, we set
  \begin{align*}
  U&=\Big(\frac{\tau}{a(\|\Sigma\|_\infty+C_{\nu}+\alpha)}\log\Big(\frac n{d\log(d+1)}\Big)\Big)^{1/(2\eta)}\vee2\quad\text{and}\\
  \lambda &= C\big(\|\Sigma\|_\infty+C_{\nu}+\alpha\big)^{(s+2)/(2\eta)}\Big(\log \frac n{d\log(d+1)}\Big)^{-(s+2)/(2\eta)}
  \end{align*}
  for some $\tau<1/2$ and a constant $C>0$ which can be chosen independently of $\gamma,\Sigma$ and $\nu$. Then we have for sufficiently large $n$
  \[
  \|\tilde\Sigma_{n,\lambda}-\Sigma\|_{2}\le C\big(\|\Sigma\|_\infty+C_{\nu}+\alpha\big)^{(s+2)/(2\eta)}\sqrt{\rank(\Sigma)}\Big(\log \frac n{d\log(d+1)}\Big)^{-(s+2)/(2\eta)}
  \]
  with probability larger than $1-2(d+1)e^{-cd(\log\log n)^2}$ for some $c>0$ depending only on $w,C_\nu$ and $C_L$.
\end{cor} 
The assertion follows from Corollary~\ref{cor:FrobError} and Theorem~\ref{thm:conR} (setting $\kappa=2\sqrt {d\log(d+1)}(\log \log n)$), taking into account that due to \eqref{eq:psi} (and \eqref{REst}) we have
  \begin{align*}
  \sup_{|u|\le U}|\psi(u)| & \le\|\Sigma\|_\infty U^{2}/2+2C_{\nu}U^{-(s\wedge0)}+\alpha
    \le\frac{U^{2}}{2}\bigl(\|\Sigma\|_\infty+C_{\nu}+\alpha\bigr).
  \end{align*}
This corollary shows that the exponential decay of $\mathscr L'$ leads to a severely ill-posed problem and, as a consequence, the rate is only logarithmic. Therein, the interplay between sample size $n$ and dimension $d$, i.e. the term $\frac n{d\log (d+1)}$ has been also observed for matrix completion \cite[Cor. 2]{koltchinskii2011nuclear}. Whenever there is some $\rho\in(0,1)$ such that $d\lesssim n^\rho$ the logarithmic term in the upper bound simplifies to $(\log n)^{-(s+2)/(2\eta)}$. The slow rate implies that $\tilde\Sigma_{n,\lambda}$ is consistent in absolute error only if $\rank\Sigma=o((\log n)^{(s+2)/\eta})$. Considering the relative error ${\|\tilde\Sigma_{n,\lambda}-\Sigma\|_2}/{\sqrt{\rank(\Sigma)}}$, this restriction vanishes. Note also that in a high-dimensional principal component analysis, the spectral norm $\|\Sigma\|_\infty$ may grow in $d$, too, cf. the discussion for the factor model by \citet{fanEtAl2013}.

If $\mathscr L'$ decays only polynomially, as for instance for the gamma subordinator, cf. Examples~\ref{ex:laplace}, the estimation problem is only mildly ill-posed and the convergence rates are polynomial in $n$.

\begin{cor}\label{cor:mildly} Grant Assumptions~\ref{ass:jumps} and \ref{ass:time}(i)-(iv) and assume $d\log (d+1)=o(n(\log n)^{-2})$. For a convex subset $\mathbb M\subset \R^{d\times d}$ of positive semi-definite matrices, let $0\neq\Sigma\in\mathbb{M}$, $\rank\Sigma=k$ and $\alpha\in I$. Suppose $|\mathscr{L}(z)|\lesssim (1+|z|)^{-\eta}$ and $|\mathscr{L}'(z)|\gtrsim|z|^{-\eta-1}$ for $z\in\mathbb C$ with $\Re(z)>0$ and for some $\eta>0$ such that $s+2>(2\eta\wedge k)$. Denoting the smallest strictly positive eigenvalue of $\Sigma$ by $\lambda_{min}(\Sigma)$, we set 
  \begin{align*}
    U&= \Big(\frac n{(\log n)^2d\log (d+1)}\Big)^{1/(2s+4+4\eta-(2\eta\wedge k))}\quad \text{and}\\
    \lambda&= C\bigl(\|\Sigma\|_\infty+C_{\nu}+\alpha\bigr)^{\eta+1}\lambda_{min}(\Sigma)^{-(2\eta\wedge k)/4}\Big(\frac n{(\log n)^2d\log (d+1)}\Big)^{-(s+2)/(2s+4+4\eta-(2\eta\wedge k))},
  \end{align*}
  for a constant $C>0$ independent of $\gamma,\Sigma$ and $\nu$. Then we have for sufficiently large $n$
  \[
  \|\tilde\Sigma_{n,\lambda}-\Sigma\|_{2}\le C\frac{(\|\Sigma\|_\infty+C_{\nu}+\alpha)^{\eta+1}}{\lambda_{min}(\Sigma)^{(2\eta\wedge k)/4}}\sqrt{\rank(\Sigma)}\Big(\frac n{(\log n)^2d\log (d+1)}\Big)^{-(s+2)/(2s+4+4\eta-(2\eta\wedge k))}
  \]
  with probability larger than $1-2(d+1)e^{-cd(\log n)^2}$ for some $c>0$ depending only on $w,C_\nu$ and $C_L$. 

\end{cor}

\begin{remarks}\hspace*{1em}
\begin{enumerate}
\item 
The convergence rate reflects the regularity $s+2$ and the degree of ill-posedness $2\eta=2(\eta+1)-2$ of the statistical inverse problem, where $2(\eta+1)$ is the decay rate of the characteristic function and we gain two degrees since $\langle u,\Sigma u\rangle$ grows like $|u|^2$. 
The term $-(2\eta)\wedge k$ appearing in the denominator is very surprising since the rate becomes faster as the rank of $\Sigma$ increases up to some critical threshold value $2\eta$. To see this, it remains to note that the assumption $\rank\Sigma=k$ yields
\begin{align*}
  U^{-d}\|\phi\|_{L^1(B_U^d)}
    =&U^{-d}\int_{|u|\le U}|\mathscr L(-\psi(u))|\d u\\
    \lesssim& U^{-d}\int_{|u|\le U} (1+|\psi(u)|)^{-\eta}\d u\\
    \lesssim& U^{-k}\int_{u\in\R^k:|u|\le U}(1+\lambda_{min}(\Sigma)|u|^2)^{-\eta}\d u
    \lesssim \frac{(\sqrt{\lambda_{min}(\Sigma)}U)^{-2\eta+k}\vee1}{(\sqrt{\lambda_{min}(\Sigma)}U)^{k}}.  
\end{align*}
In order to shed some light on this phenomenon, we consider a related regression problem in Section~\ref{sec:reg}.

\item It is interesting to note that for very small $\eta$ we almost attain the parametric rate. Having the example of gamma-distributed increments $T_j$ in mind, a small $\eta$ corresponds to measures $\pi$ which are highly concentrated at the origin. In that case we expect many quite small increments $X_{\mathcal T(j)}-X_{\mathcal T(j-1)}$ where the jump component has only a small effect. On the other hand, for the remaining few rather large increments, the estimator is not worse than in a purely low-frequency setting. Hence, $|\mathscr{L}(z)|\sim (1+|z|)^{-\eta}$ heuristically corresponds to an interpolation between high- and low-frequency observations, cf. also \citet{kappus2015}.
\item If the condition $s+2\ge2\eta\wedge \rank\Sigma$ is not satisfied, the linearised stochastic error appears to be smaller than the remainder of the linearsation. It that case we only obtain the presumably suboptimal rate $n^{-(s+2)/(s+2+4\eta)}$ (for fixed $d$). It is still an open problem whether this condition is only an artefact of our proofs or if there is some intrinsic reason.
\end{enumerate}
\end{remarks}

Let us now investigate whether the rates are optimal in the minimax sense. The dependence on the the rank of $\Sigma$ is the same as usual in matrix estimation problems and seems natural. The optimality of the influence of the dimension $d$ is less clear. In lower bounds for matrix completion, cf. \citet[Thm. 6]{koltchinskii2011nuclear}, a similar dependence appears except for the fact that we face non-parametric rates. In the following we prove lower bounds for the rate of convergence in the number of observations $n$ for a fixed $d$. More general lower bounds in a high-dimensional setting where $d$ may grow with $n$ are highly intricate since usually lower bounds for high-dimensional statistical problems are based on Kullback-Leibler divergences in Gaussian models - an approach that cannot be applied in our model. This problem is left open for future research.

Let us introduce the class $\mathfrak{S}(s,p,C_\nu)$ of all L\'evy measures satisfying Assumption~\ref{ass:jumps} with $s\in(-2,\infty),p>2$ and a constant $C_\nu>0$. In order to prove sharp lower bounds, we need the following stronger assumptions on the time change:
\begin{ass}\label{ass:timeLowerBound} 
  Grant (i) and (ii) from Assumption~\ref{ass:time}. For $C_L>0$ and $L\in\N$ let the Laplace transform $\mathscr{L}(z)$ satisfy
  \begin{enumerate}
    \item[(m)] for some $\eta>0$ and all $z>0$
      \begin{eqnarray*}
        |\mathscr{L}'(z) |\le C_L(1+|z|)^{-\eta-1},\quad
        |\mathscr{L}^{(l+1)}(z)/\mathscr{L}^{(l)}(z)|\le C_L(1+|z|)^{-1},\quad l=1,\dots,L,
      \end{eqnarray*}
      or
    \item[(s)] for some $a,\eta>0$ and all $z\in\C$ with $\Re(z)>0$
      \begin{eqnarray*}
         |\mathscr{L}'(z)|\le C_Le^{-a|z|^\eta},\quad
        |\mathscr{L}^{(l+1)}(z)/\mathscr{L}^{(l)}(z)|\le C_L,\quad l=1,\dots,L.
      \end{eqnarray*}  
  \end{enumerate}
\end{ass}

\begin{thm}\label{thm:lowerBound} 
  Fix $d\ge1$ and let $s\in (-2,\infty),p>2,C_\nu>0$ and $k\in\{1,\dots,d\}$.
  \begin{enumerate}
    \item Suppose $2\eta> k$ and $k\le s$. 
    Under  Assumption~\ref{ass:timeLowerBound}(m) with $L>\frac{k+p\vee(-s)}{2}\vee k,$ it holds for any $\eps>0$ that
    \begin{align*}
      \liminf_{n\to\infty}\inf_{\widehat{\Sigma}}\sup_{\substack{\rank(\Sigma)=k\\ \nu\in\mathfrak{S}(s,p,C_\nu)}}\P_{(\Sigma,\nu,\mathcal{T})}^{\otimes n}\left(\|\widehat\Sigma-\Sigma\|_{2}>\varepsilon n^{-(s+2)/(2(s+2)+4\eta-k))}\right)>0.
    \end{align*}
    Moreover, if $\eta>1$, we have under  Assumption~\ref{ass:timeLowerBound}(m) with $L>\frac{1+p\vee(-s)}{2}\vee 1$ for any $\eps>0$
    \begin{align*}
      \liminf_{n\to\infty}\inf_{\widehat{\Sigma}}\sup_{\substack{2\eta\le\rank(\Sigma)\le d\\ \nu\in\mathfrak{S}(s,p,C_\nu)}}\P_{(\Sigma,\nu,\mathcal{T})}^{\otimes n}\left(\|\widehat\Sigma-\Sigma\|_{2}>\varepsilon n^{-(s+2)/(2(s+2)+2\eta))}\right)>0.
    \end{align*}
    \item Under the Assumption~\ref{ass:timeLowerBound}(s)  with $L>\frac{1+p\vee(-s)}{2}$ it holds for any $\eps>0$ that
    \begin{align*}
      \liminf_{n\to\infty}\inf_{\widehat{\Sigma}}\sup_{\substack{0<\rank(\Sigma)\le d\\\nu\in\mathfrak{S}(s,p,C_\nu)}}\P_{(\Sigma,\nu,\mathcal{T})}^{\otimes n}\left(\|\widehat\Sigma-\Sigma\|_{2}>\varepsilon (\log n)^{-(s+2)/(2\eta)}\right)>0.
    \end{align*}
  \end{enumerate}
  Note that the infima are taken over all estimators of $\Sigma$ based on $n$ independent observations of the random variable $Y_1$ whose law we denoted by $\P_{(\Sigma,\nu,\mathcal{T})}$.
\end{thm} 
Up to a logarithmic factors in the mildly ill-posed case, the upper and lower bounds coincide and thus our convergence rates are minimax optimal for fixed $d$.

\section{Discussions and extensions}\label{seq:ext}
In this section we discuss several generalisations of the previously developed theory.

\subsection{Estimating the time-change}\label{sec:lapl_est}

In order to define the diffusion matrix estimator we have assumed that the Laplace transform of the time-change increments is known. In the context of random observation times it is reasonable to suppose that we additionally observe $\mathcal T(1),\dots, \mathcal T(m)$ for some sample size $m$. Using empirical process theory, it is not difficult to show that
\[
  \mathscr L_m(z):=\frac1m\sum_{j=1}^me^{-z(\mathcal T(j)-\mathcal T(j-1))}
\]
converges uniformly to $\mathscr L(z)$ for \(z\in\R_+\) with $\sqrt n$-rate as exploited by \citet{chorowskiTrabs2015}. This result could be generalised to appropriate arguments on the complex plane and may allow for replacing $\mathscr L$ by $\mathscr L_n$ for the estimator of the characteristic exponent \eqref{eq:PsiHat}.

More precisely, we need to invert the function \(\mathscr L_{m}.\) This is not a trivial problem, since the inverse function \(\mathscr L^{-1}_{m}\) should be well defined on a complex plane (we are going to plug in a complex-valued function \(\phi_{n}(u)\)). We use the Lagrange inversion formula to compute the derivatives of  \(\mathscr L^{-1}_{m}(z)\) at point \(z=1\) and then expand this function into Taylor series around \(z=1.\) This approximation turns out to be very stable and can even improve the quality of  estimating \(\Sigma\), see Section~\ref{sec:sim_lapl_est}. We have a formal expansion 
\begin{eqnarray*}
\mathscr L_{m}(z)=\sum_{k=1}^\infty M_{k}\frac{(-z)^k}{k !},\quad  M_{k}:=M_k(m)=\frac{1}{m}\sum_{j=1}^m (\mathcal T(j)-\mathcal T(j-1))^k.
\end{eqnarray*}
In the case of a continuous distribution of \(\mathcal T(1),\)  \(\mathscr L'_{m}(z)>0\) for any \(z\in \mathbb{R}\) with probability one and the inverse function \(\mathscr L^{-1}_{m}(z)\) exists and is unique. In some vicinity of the point \(z=1\) in the complex plan, we obtain the expansion
\begin{eqnarray}
\label{eq: Linv}
\mathscr L^{-1}_{m}(z)=\sum_{j=1}^\infty H_{j}\frac{(z-1)^j}{j !}
\end{eqnarray}
with
\begin{gather*}
H_1=\frac{1}{M_1},\quad H_{j}=\frac{1}{M^j_{1}} \sum_{k=1}^{j-1} (-1)^k j^{(k)} B_{j-1,k}(\hat{M}_1,\hat{M}_2,\ldots,\hat{M}_{j-k}), \quad j \geq 2,\\
\widehat M_k:=\frac{M_{k+1}}{(k+1)M_1}\qquad \text{and}\qquad j^{(k)}:=j(j+1)\cdot\ldots\cdot (j+k-1), 
\end{gather*}
where \(B_{j-1,k}\) are partial Bell polynomials. Formula~\eqref{eq: Linv} can then be used to obtain another estimator for the characteristic exponent which adapts to the unknown Laplace transform:
\[
  \tilde\psi_{n,m}(u)=-\mathscr L^{-1}_{m}(\phi_n(u)).
\]
We investigate the numerical performance of the this approach in Section~\ref{sec:sim_lapl_est}.

\subsection{Mixing time-change}\label{sec:mixing}
Motivated by the integrated CIR process from Example~\ref{ex:laplace}(iv), for instanced used by \citet{carrEtAl2003} to model stock price processes, we will generalise the results from the previous section to time-changes $\mathcal T$, whose increments are not i.i.d., but form a strictly stationary $\alpha$-mixing sequence. Recall that the strong mixing coefficients of the sequence $(T_j)$ are defined by 
    \[
      \alpha_T(n):=\sup_{k\ge1}\alpha(\mathcal M_k,\mathcal G_{k+n}),\quad 
      \alpha(\mathcal M_k,\mathcal G_{\ell}):=\sup_{A\in\mathcal M_k,B\in\mathcal G_{\ell}}|\mathbb P(A\cap B)-\mathbb P(A)\mathbb P(B)|
    \]
for $\mathcal M_k:=\sigma(T_j:j\le k)$ and $\mathcal G_k:=\sigma(T_j:j\ge k)$.
We replace Assumption~\ref{ass:time} by the following
\begin{ass}\label{ass:time2} 
  Let the time-change $\mathcal{T}$ satisfy: 
  \begin{enumerate} 
    \item $\E[\mathcal{T}^{p}]<\infty$ for some $p>2$. 
    \item The sequence $T_{j}=\mathcal{T}(j)-\mathcal{T}(j-1)$,
    $j\in\mathbb{N},$
    is strictly stationary with invariant measure $\pi$ and $\alpha$-mixing with 
    \begin{eqnarray*}
    \alpha_{T}(j)\le\overline{\alpha}_{0}\exp(-\overline{\alpha}_{1}j),\quad j\in\mathbb{N},
    \end{eqnarray*}
    for some positive constants $\overline{\alpha}_{0}$
    and $\overline{\alpha}_{1}.$
    \item The Laplace transform $\mathscr{L}$
    satisfies $|\mathscr{L}''(z)/\mathscr{L}'(z)|\le C_L$
    for some $C_L>0$ for all $z\in\C$ with $\Re(z)>0$. 
  \end{enumerate}
\end{ass} 

If $T_{j}$ is $\alpha$-mixing, Lemma 7.1 by \citet{belomestny2011}
shows that the sequence $(Y_{j})$ inherits the mixing-property from
the sequence $(T_{j})$. In combination with the finite moments $\E[|Y_{j}|^{p}]<\infty$
for $p>2$ we can apply a concentration inequality on the empirical
characteristic function $\phi_{n}$, see Theorem~\ref{thm:concentration}
below, which follows from the results by \citet{MerlevedeEtAl2009}. 

In the $\alpha$-mixing case, a noncommutative Bernstein inequality is not known (at least to the authors' knowledge) and thus we cannot hope for a concentration inequality for $\|\mathcal R_n\|_\infty$ analogous to Theorem~\ref{thm:conR}. Possible workarounds are either to estimate 
\begin{equation}\label{eq:EstMix}
  \|\mathcal R_n\|_\infty\le 2\int_{\mathbb{R}^{d}}\Big|\frac2{|u|^2}\Re\hat{\psi}_{n}(u)-\langle\tilde\Theta(u),\Sigma\rangle\Big|\|\tilde\Theta(u)\|_\infty w_{U}(u)\,\d u,
\end{equation}
where $\|\tilde\Theta(u)\|_\infty=1$, or to bound $\|\mathcal R_n\|_\infty\le (d+1)\|\mathcal R_n\|_{max}$ for the maximum entry norm $\|A\|_{max}=\max_{ij}|A_{ij}|$ for $A\in\R^{(d+1)\times (d+1)}$.  While the former estimate leads to suboptimal rates for polynomially decaying $\mathscr L$, we loose a factor $d$ in the latter bound which is critical in a high-dimensional setting.

Having in mind that the Laplace transform of the integrated CIR process decays exponentially, we pursue the first idea and obtain the following concentration result:
\begin{thm}\label{thm:conMix}
Grant Assumptions~\ref{ass:jumps} and \ref{ass:time2} and let $\rho>1/2$. There are constants $\underline\xi,\overline\xi>0$ such that for any $n\in\N$ and $U,\kappa>0$ satisfying
\[
  \underline\xi\sqrt{d\log n}<\kappa<\overline\xi(\log U)^{-\rho}\inf_{|u|\le U}|\mathscr L'(-\psi(u))|(\log n)^{-1/2}\sqrt n
\]
the matrix $\mathcal R_n$ from \eqref{eq:error} satisfies for some constants $c,C,D>0$ depending only on $w, C_\nu$ and $C_L$ that
\[
\P\Big(\|\mathcal R_{n}\|_\infty\geq \frac{\kappa (\log U)^\rho}{\sqrt nU^{2}\inf_{|u|\le U}|\mathscr L'(-\psi(u))|}+D U^{-s-2}\Big)\le Ce^{-c\kappa^2}+Cn^{-p/2}.
\]
In particular, we have $\P(\|\mathcal R_n\|_\infty>\lambda)\le Ce^{-c\kappa^2}+Cn^{-p/2}$ if 
\[
  \lambda\ge\frac{\kappa(\log U)^\rho}{\sqrt nU^{2}\inf_{|u|\le U}|\mathscr L'(-\psi(u))|}+D U^{-s-2}.
\]
\end{thm} 
The suboptimal term $\sqrt{\log n}$ in lower bound of $\kappa$ comes from the estimate \eqref{eq:EstMix}. The term $(\log U)^\rho$ could be omitted with a more precise estimate of the linearised stochastic error term (similar to the proof of Theorem~\ref{thm:conR}), but let us keep the proof of Theorem~\ref{thm:conMix} simple, see Section~\ref{sec:proodmixing}. For exponentially decaying Laplace transforms the resulting rate is already sharp and coincides with our results for the independent case as soon as $d\le n^\tau$ for some $\tau\in[0,1)$. 
\begin{cor}
  Grant Assumptions~\ref{ass:jumps} and \ref{ass:time2}, assume $d=o(n/\log n)$ and let $\mathbb{M}$ be a convex subset of positive semi-definite $d\times d$
  matrices. Suppose that $0\neq\Sigma\in\mathbb{M}$ and $\alpha\in I$. Assume $|\mathscr{L}'(z)|\gtrsim\exp(-a|z|^\eta)$ for $z\in\mathbb C$ with $\Re(z)>0$ and for $a,\eta>0$. Choosing
  \[
  U=\Big(\frac{\tau\log(n/(d\log n))}{a(\|\Sigma\|_\infty+C_{\nu}+\alpha)}\Big)^{1/(2\eta)}\vee2\quad\text{and}\quad
  \lambda = C_1\big(\|\Sigma\|_\infty+C_{\nu}+\alpha\big)^{(s+2)/(2\eta)}\Big(\log \frac n{d\log n}\Big)^{-(s+2)/(2\eta)}
  \]
  for some $\tau<1/2$ and a constant $C_1>0$, we have
  \[
  \|\tilde\Sigma_{n,\lambda}-\Sigma\|_{2}\le C_{1}\bigl(\|\Sigma\|_\infty+C_{\nu}+\alpha\bigr)^{(s+2)/(2\eta)}\sqrt{\rank(\Sigma)}\Big(\log \frac n{d\log n}\Big)^{-(s+2)/(2\eta)}
  \]
  with probability larger than $1-C_{2}(e^{-c\sqrt{d\log n}}+n^{-p/2})$ and for some $C_2,c>0$. The constants $C_1,C_2,c$ depend only on $w,C_L,C_\nu$.
\end{cor}

\subsection{Low rank plus sparse matrix}\label{sec:sparse}

In view of \citet{fanEtAl2011,fanEtAl2013} it might be interesting in applications to relax the low rank assumption on the diffusion matrix $\Sigma$ to case where $\Sigma=\Sigma_r+\Sigma_s$ for a low-rank matrix $\Sigma_r\in\R^{d\times d}$ and a sparse matrix $\Sigma_s\in\R^{d\times d}$. The underlying idea is that $\Sigma_r$ reflects a low dimensional factor model while $\Sigma_s$ represents a sparse error covariance matrix. Sparsity in $\Sigma_s$ means here that most entries of $\Sigma_s$ are zero or very small, see for instance \cite{bickelLevina2008} for this approach to sparsity. More precisely, let us assume that the entry-wise $\ell_q$-norm $|\Sigma_s|_q$ for some $q\in[0,2)$ is small. For the sake of clarity, we focus on extending the estimator $\hat\Sigma_{n,\lambda}$ from \eqref{eq:sigmaHat} while it can be easily seen how to modify the following for the bias corrected estimator $\tilde\Sigma_{n,\lambda}$. 

In terms of the vectorisation operator $\vec(A)=(a_{11},\dots,a_{d1},a_{21},\dots,a_{dd})^\top$ for any $A=(a_{i,j})\in\R^{d\times d}$ we rewrite
\begin{align*}
  \langle\Theta(u),\Sigma_s\rangle=\tr(\Theta(u)^\top \Sigma_s)=\vec(\Theta(u))^\top\vec(\Sigma_s)=\langle\diag(\vec(\Theta(u))),\diag(\vec(\Sigma_s))\rangle.
\end{align*}
Hence, defining the set $\mathbb D:=\{\diag(A,\diag(a)):A\in\R^{d\times d},a\in\R^{d^2}\}\subset\R^{(d+d^2)\times(d+d^2)}$ and the matrices 
\[
  \bar\Theta(u):=\diag\big(\Theta(u),\diag(\vec(\Theta(u)))\big)\in\mathbb D\quad\text{and}\quad \bar\Sigma:=\diag\big(\Sigma_r,\diag(\vec(\Sigma_s))\big)\in\mathbb D,
\]
we obtain the representation
\begin{align*}
   \langle \Theta(u),\Sigma\rangle=\langle \Theta(u),\Sigma_r\rangle+\langle\Theta(u),\Sigma_s\rangle
   =\langle \bar\Theta(u),\bar\Sigma\rangle.
\end{align*}
Motivated by this reformulation, we introduce the estimator $\bar{\Sigma}_{n,\lambda}:=\bar{\Sigma}_{r}+\bar{\Sigma}_{s}$ where
\begin{equation}
\diag\big(\bar\Sigma_r,\diag(\vec(\bar\Sigma_s))\big):=\argmin_{M\in\bar {\mathbb M}}\left\{ \int_{\mathbb{R}^{d}}\big(2|u|^{-2}\Re\hat{\psi}_{n}(u)+\langle\bar\Theta(u),M\rangle\big)^{2}w_{U}(u)\,\d u+\lambda\|M\|_{1}\right\} \label{eq:sigmabar}
\end{equation}
for a subset $\bar{\mathbb M}\subset\mathbb D$. Since $\rank(\bar\Sigma)=\rank(\Sigma_r)+|\Sigma_s|_0$, the nuclear norm penalisation in \eqref{eq:sigmabar} will enforce a low rank structure in $\bar{\Sigma}_{r}$ and a sparse structure in $\bar{\Sigma}_{s}$.
In order to carry over the oracle inequalities, we first need verify the isometry property for the modified weighted norm
\[
  \langle A,B\rangle_{\bar w}:=\int_{\R^d}\langle\bar\Theta(u),A\rangle\langle\bar\Theta(u),B\rangle w_U(u)\d u,\quad A,B\in\mathbb D.
\]
Indeed, we easily deduce from Lemma~\ref{lem:rip} and $|\Theta(u)|_2=1$:
\begin{lem}\label{lem:rip2}
  For any symmetric matrix $A\in\mathbb D$ we have
  \begin{equation*}
    \underline\varkappa_{\bar w}\|A\|_2\le\|A\|_{\bar w} \le \bar\varkappa_{\bar w}\|A\|_2,
  \end{equation*} 
  where $\underline\varkappa_{\bar w}^2:=\min_{i,j=1,2}\int_{\R^d}v_{i}^{2}v_j^2/(|v|^{4})w(v)\d v$ and $\bar\varkappa_w^2:=2\int_{\R^d}w(v)\d v$.
\end{lem}
We recover the exact structure for which Theorem~\ref{thm:oracle2} has been proven and thus conclude:
\begin{prop}
  Suppose that $\bar{\mathbb{M}}\subset\mathbb D$ is a convex subset of the symmetric matrices. Define
  \begin{equation}
    \bar{\mathcal{R}}_n:=2\int_{\mathbb{R}^{d}}\Big(\frac2{|u|^2}\Re\hat{\psi}_{n}(u)-\langle\bar\Theta(u),\bar\Sigma\rangle\Big)\bar\Theta(u) w_{U}(u)\,\d u\in\mathbb D,\label{eq:barR}
  \end{equation}
  On the event $\{\|\bar{\mathcal R}_n\|_\infty\le \lambda\}$ for some $\lambda>0$ the estimator $\bar\Sigma_{n,\lambda}=\bar{\Sigma}_{r}+\bar{\Sigma}_{s}$ given in \eqref{eq:sigmabar} satisfies
  \begin{align*}
  &\big\|\diag\big(\bar\Sigma_r,\diag(\vec(\bar\Sigma_s))\big)-\bar\Sigma\big\|_{w}^{2}
  \le\inf_{M\in\bar{\mathbb M}}\left\{\|M-\bar\Sigma\|_{w}^{2}+(\tfrac{1+\sqrt 2}{2\underline\varkappa_{\bar w}})^2\lambda^2\operatorname{rank}(M)\right\}
  \end{align*}
  with $\underline\varkappa_{\bar w}$ from Lemma~\ref{lem:rip2}.
\end{prop}

Noting $\|\bar\Sigma\|_2^2=\|\Sigma_r\|_2^2+\|\Sigma_s\|_2^2\ge\frac12\|\Sigma\|_2^2$ and using Lemma~\ref{lem:rip2}, we obtain an oracle inequality for the estimation error of $\bar\Sigma_{n,\lambda}$ in the Frobenius norm. The extension from $q=0$ to $q\in(0,2)$ can be proven as Corollary~2 by \citet{rigolletTsybakov2012}. 
\begin{thm}
   Let $\bar{\mathbb{M}}\subset\mathbb D$ is a convex subset of the symmetric matrices. On the event $\{\|\bar{\mathcal R}_n\|_\infty\le \lambda\}$ for some $\lambda>0$ and for $\bar{\mathcal R}_n$ from \eqref{eq:barR} we have for a constant $C_w$ depending only on $\underline\varkappa_{\bar w}$ and $\bar\varkappa_{\bar w}$:
  \begin{align*}
  \|\bar\Sigma_{n,\lambda}-\Sigma\|_2^2
  &\le 2\big(\|\bar{\Sigma}_{r}-\Sigma_r\|^2_{2}+\|\bar\Sigma_s-\Sigma_s\|^2_2\big)\\
  &\le C_w\inf_{M_r,M_s}\Big\{\|M_r-\Sigma_r\|_2^{2}+\|M_s-\Sigma_s\|_2^2+\lambda^2\big(\rank(M_r)+|M_s|_0\big)\Big\}
  \end{align*}
  where the infimum is taken over all $M_r,M_s$ such that $\diag(M_r,\diag(\vec(M_s)))\in\bar{\mathbb M}$. Moreover, for another constant $C'_w$ and any $q\in(0,2)$ we have
  \begin{align*}
  \|\bar\Sigma_{n,\lambda}-\Sigma\|_2^2
  &\le 2\big(\|\bar{\Sigma}_{r}-\Sigma_r\|^2_{2}+\|\bar\Sigma_s-\Sigma_s\|^2_2\big)\\
  &\le C'_w\inf_{M_r,M_s}\Big\{\|M_r-\Sigma_r\|_2^{2}+\|M_s-\Sigma_s\|_2^2+\lambda^2\rank(M_r)+\lambda^{2-q}|M_s|^q_q\Big\}.
  \end{align*}
\end{thm}

The analysis of the error term $\bar{\mathcal R}_n$ can be done as in Section~\ref{seq: mcr} noting that $\bar\Theta(u)$ has very similar properties as $\Theta(u)$, especially $\|\bar\Theta(u)\|_\infty\le\|\Theta(u)\|_\infty+|\Theta(u)|_{\infty}\le2$ and $\|\bar\Theta(u)\|_2=2\|\Theta(u)\|_2$. We omit the details.

\subsection{Incorporating positive definiteness constraints }

Consider the optimisation problem 
\begin{equation}
M_{n}:=\arginf_{M\succeq0,\, M\in\R^{d\times d}}Q_{n}(M)
\quad\text{with}\quad
Q_{n}(M)=\int_{\mathbb{R}^{d}}\Big(\frac{2}{|u|^2}\Re\hat\psi_n(u)-\langle\Theta(u),M\rangle\Big)^{2}w_{U}(u)\d u.\label{eq:optim_gen}
\end{equation}
In general, the solution of \eqref{eq:optim_gen} has to be searched
in a space of dimension $O(d^{2})$. Solving such a problem becomes
rapidly intractable for large $d$. In order to solve \eqref{eq:optim_gen}
at a reduced computational cost, we assume that $\rank(\Sigma)\le k\ll d.$
In order to handle the positive definiteness constraint, we let $\ensuremath{M=LL^{\top}}$ with
$\ensuremath{L\in\mathbb{R}^{k\times d}}$ and rewrite the problem
\eqref{eq:optim_gen} as

\begin{equation}
L_{n}:=\arginf_{L\in\mathbb{R}^{k\times d}}Q_{n}(LL^{\top}).\label{eq:optim_gen_l}
\end{equation}
In fact any local minima of \eqref{eq:optim_gen_l} leads to a local
minima of \eqref{eq:optim_gen}. Since any local minima is a global
minima for the convex minimization problem \eqref{eq:optim_gen},
any local minima of \eqref{eq:optim_gen_l} is a global minima.

\subsection{A related regression problem}\label{sec:reg}

In order to understand the dimension effect that we have observed in the convergence rates in Corollary~\ref{cor:mildly}, we will take a more abstract point of view considering a related regression-type problem in this section. Motivated from the regression formula \eqref{eq:regression}, we study the estimation of a possibly low-rank matrix $\Sigma\in\R^{d\times d},\Sigma>0$, based on the observations
\begin{equation}\label{eq:regModel}
X_{i}(u)=\langle u,\Sigma u\rangle+\Psi(u)+\epsilon_{i}(u),\quad i=1,\dots,n,u\in\R^{d},
\end{equation}
where $\Psi\colon\R^{d}\to\R$ is an unknown deterministic nuisance function satisfying $|\Psi(u)|\to0$ as $|u|\to\infty$ and $\varepsilon_{i}=(\epsilon_{i}(u))_{u\in\R^{d}}$ are centred, bounded random variables such that:
\begin{equation}
\epsilon_{1},\dots,\epsilon_{n}\mbox{ are i.i.d. with}\quad\E[\epsilon_{i}(u)\epsilon_{i}(v)]=(|u|^{\beta}|v|^{\beta}\vee1)\phi(u-v),\quad u,v\in\R^{d}, i\in\{1,\dots,n\},\label{eq:eps}
\end{equation}
for some symmetric function $\phi\in L^{1}(\R^{d})$ and $\beta\ge0$.
The covariance structure of $\epsilon_{i}(u)$ includes on the one
hand that the variance of $\epsilon_{i}(u)$ increases (polynomially)
as $|u|\to\infty$, implying some ill-posedness, and on the other
hand the random variables $\epsilon_{i}(u)$ and $\epsilon_{i}(v)$
decorrelate as the distance of $u$ and $v$ increases (supposing
that $\phi$ is a reasonable regular function).

Following our weighted Lasso approach, we define the statistic 
\[
T_{n}(u):=\frac{1}{n}\sum_{i=1}^{n}X_{i}(u),\quad u\in\R^{d},
\]
and introduce the weighted least squares estimator
\begin{equation}\label{eq:regEst}
\hat{\Sigma}_{n,\lambda}:=\argmin_{M\in\mathbb{M}}\int_{\R^{d}}\big(T_{n}(u)-\langle u,M u\rangle\big)^{2}w_{U}(u)\,\d u+\lambda\|M\|_{1}
\end{equation}
for a convex subset $\mathbb{M}\subset\{A\in\R^{d\times d}:A>0\}$,
a weight function $w_{U}(u):=w(u/U),u\in\R^{d},$ for some radial
function $w\colon\R^{d}\to\R_{+}$ with support $\supp w\subset\{\frac12\le |u|\le1\}$
and a penalisation parameter $\lambda>0$.

\vspace*{1em}

Our oracle inequalities easily carry over to the regression setting.
We define the weighted scalar product and corresponding (semi-)norm
\[
\langle A,B\rangle_{U}:=\int_{\R^{d}}\langle u,Au\rangle\langle u,Bu\rangle w_{U}(u)\d u\quad\mbox{and}\quad\|A\|_{U}^{2}:=\langle A,A\rangle_{U}
\]
as well as the error matrix
\[
\mathcal{R}_{n}:=\int_{\R^{d}}\big(T_{n}(u)-\langle u,\Sigma u\rangle\big)uu^{\top}w_{U}(u)\d u\in\mathbb R^{d\times d}.
\]
Along the lines of Proposition~\ref{prop:oracle1} we obtain
\[
\|\hat{\Sigma}_{n,\lambda}-\Sigma\|_{U}^{2}\le\inf_{M\in\mathbb{M}}\big\{\|M-\Sigma\|_{U}^{2}+2\lambda\|M\|_{1}\big\}\quad\text{on the event }\{\|\mathcal{R}_{n}\|_{\infty}\le\lambda\}.
\]
It is now crucial to compare the weighted norm $\|\cdot\|_{U}$ with
the Frobenius norm $\|\cdot\|_{2}$. Analogously to Lemma~\ref{lem:rip}, we
have for any positive definite matrix $A=QDQ^{\top}\in\R^{d\times d}$
with diagonal matrix $D$ and orthogonal matrix $Q$ that
\begin{align}
\|A\|_{U}^{2} & =\int_{\R^{d}}\langle Q^{\top}u,DQ^{\top}u\rangle^{2}w_{U}(u)\,\d u\nonumber \\
 & =U^{d+4}\int_{\R^{d}}\langle v,Dv\rangle^{2}w(v)\,\d v\ge U^{d+4}\|A\|_{2}^{2}\underbrace{\int v_{1}^{4}w(v)\d v}_{=:\varkappa^2}.\label{eq:norms}
\end{align}
Hence, compared to the Frobenius norm the weighted norm becomes stronger as $U$ and $d$ increase. In the well specified case $\Sigma\in\mathbb{M}$ we conclude
\[
\|\hat{\Sigma}_{n,\lambda}-\Sigma\|_{2}\le U^{-d/2-2}\varkappa\sqrt{2\lambda\|\Sigma\|_{1}}\quad\mbox{on the event }\{\|\mathcal{R}_{n}\|_{\infty}\le\lambda\}.
\]
Theorem~\ref{thm:oracle2} corresponds to the following stronger inequality which can be proved analogously. Because $w$ is not normalised in the sense of inequality (\ref{eq:norms}), we have to be a bit careful in the very last step in Section~\ref{sec:proofOracle} in order not to oversee a factor $U^{-d/2-2}$.

\begin{thm}
  In the regression model \eqref{eq:regModel} with $\Sigma>0$ the estimator $\hat\Sigma_{n,\lambda}$ from \eqref{eq:regEst} satisfies
  \begin{equation}
    \|\hat{\Sigma}_{n,\lambda}-\Sigma\|_{2}\le CU^{-d-4}\lambda\sqrt{\operatorname{rank}\Sigma}\quad\mbox{on the event }\{\|\mathcal{R}_{n}\|_{\infty}\le\lambda\}\label{eq:oracle}
  \end{equation}
  for some numerical constant $C>0$.   
\end{thm}

Let us now estimate the spectral norm of the error matrix
\begin{equation}
\mathcal{R}_{n}=\frac{1}{n}\sum_{i=1}^{n}\underbrace{\int_{\R^{d}}\varepsilon_{i}(u)uu^{\top}w_{U}(u)\,\d u}_{=:Z_{i}}+\int_{\R^{d}}\Psi(u)uu^{\top}w_{U}(u)\,\d u.\label{eq:R}
\end{equation}
The second term is a deterministic error term which is bounded by
\begin{align*}
\Big\|\int_{\R^{d}}\Psi(u)uu^{\top}w_{U}(u)\Big\|_{\infty}\le & \int_{\R^{d}}|\Psi(u)||u|^{2}w_{U}(u)\d u\\
\le & \sup_{|u|\ge U/2}|\Psi(u)|U^{d+2}\int_{\R^{d}}|v|^{2}w(v)\d v,
\end{align*}
using $\|uu^{\top}\|_{\infty}=|u|^{2},u\in\R^{d}$. The dimension occurs here as a consequence of how the regularity
(decay of $\Psi$) is measured in this problem. For the first term
in (\ref{eq:R}) we apply, for instance, non-commutative Bernstein
inequality noting that $Z_{1},\dots,Z_{n}$ are i.i.d., bounded and
centred. The Cauchy-Schwarz inequality, or more precisely the estimate \eqref{eq:CStrick}, yields
\begin{align*}
\|\E[Z_{1}Z_{1}]\|_{\infty}=\|\E[Z_{1}Z_{1}^{\top}]\|_{\infty} & =\Big\|\int_{\R^{d}}\int_{\R^{d}}(|u|^{\beta}|v|^{\beta}\vee1)\phi(u-v)(uu^{\top})(vv^{\top})w_{U}(u)w_{U}(v)\d u\d v\Big\|_{\infty}\\
 & \le\int_{\R^{d}}\int_{\R^{d}}|\phi(u-v)||u|^{2+\beta}|v|^{2+\beta}w_{U}(u)w_{U}(v)\d u\d v\\
 & \le\int_{\R^{d}}|u|^{4+2\beta}w_{U}^{2}(u)\int_{\R^{d}}|\phi(u-v)|\d v\d u\\
 & =\|\phi\|_{L^{1}}U^{4+2\beta+d}\int_{\R^{d}}|v|^{4+2\beta}w^{2}(v)\,\d v.
\end{align*}
Note that the dependence of the variance on $d$ and $\beta$ is the usual
non-parametric behaviour. Therefore,
\[
\|\mathcal{R}_{n}\|_{\infty}=\mathcal{O}_{P}\big(n^{-1/2}U^{2+\beta+d/2}+\sup_{|u|\ge U/2}|\Psi(u)|U^{d+2}\big).
\]
We thus choose $\lambda\sim n^{-1/2}U^{2+\beta+d/2}+\sup_{|u|\ge U/2}|\Psi(u)|U^{d+2}$
and conclude from (\ref{eq:oracle})
\begin{equation}
\|\hat{\Sigma}_{n,\lambda}-\Sigma\|_{2}\lesssim\Big(n^{-1/2}U^{\beta-2-d/2}+\sup_{|u|\ge U/2}|\Psi(u)|U^{-2}\Big)\sqrt{\operatorname{rank}\Sigma}.\label{eq:regressionRate}
\end{equation}
If $\beta>2+d/2$, then we face a bias-variance trade-off. Supposing
$|\Psi(u)|\lesssim|u|^{-\alpha}$ for some $\alpha>0$, we may choose
the optimal $U^{*}=n^{1/(2\alpha+2\beta-d)}$ and obtain the rate
\[
\|\hat{\Sigma}_{n,\lambda}-\Sigma\|_{2}\lesssim n^{-(\alpha+2)/(2\alpha+2\beta-d)}.
\]
This rate coincides with our findings in Corollary~\ref{cor:mildly}. Note that the second regime in the mildly ill-posed case comes from the auto-deconvolution structure of the L\'evy process setting which is not represented in the regression model \eqref{eq:regModel}. 

From this analysis we see that root for the surprising dimension dependence
is the factor $U^{d}$ in the isometry property 
(\ref{eq:norms}). Intuitively, it reflects that we ``observe'' more frequencies in the annulus $\{u\in\R^{d}:\frac{U}{2}\le|u|\le U\}$ as $d$ increases and, due to the function $\phi$ in \eqref{eq:eps}, we can profit from these since the observations decorrelate if the frequencies have a large distance from each other.   

\section{Simulations}
\label{eq:sim}
\subsection{Known time change}
\label{sec:tc_known}
Let us analyse a model based on a time-changed normal inverse Gaussian (NIG) L\'evy process.
The NIG L\'evy processes, characterised by their increments being NIG distributed, have been introduced in \cite{barndorff1997processes} as a model for log returns of stock prices. \citet{barndorff1997processes} considered classes of normal variance-mean mixtures and defined the NIG distribution as the case when the
mixing distribution is inverse Gaussian. Shortly after its introduction it was shown that the NIG distribution fits very well the log returns on German stock market data, making the NIG L\'evy processes  of great interest for practioneers.  

An NIG distribution has in general the four parameters \( \alpha\in \mathbb{R}_{+}, \) \( \beta\in \mathbb{R}, \) \( \delta\in \mathbb{R}_{+} \) and \( \mu\in \mathbb{R} \) with \( |\beta|<\alpha \) having  different effects on the shape of the distribution: \( \alpha \) is responsible for the tail heaviness of steepness, \( \beta \) affects symmetry, \( \delta \) scales the distribution and \( \mu \) determines its mean value. The NIG distribution is infinitely divisible with the characteristic function
\begin{eqnarray*}
\phi(u)=\exp\left\{ \delta\left( \sqrt{\alpha^{2}-\beta^{2}}-\sqrt{\alpha^{2}-(\beta+i u)^{2}}+i \mu u \right) \right\}.
\end{eqnarray*}
Therefore one can define the NIG L\'evy process \( (X_{t})_{t\geq 0} \) which
starts at zero and has independent and stationary increments such that each increment \( X_{t+\Delta}-X_{t} \) is \( \operatorname{NIG}(\alpha, \beta, \Delta\delta,\Delta\mu) \)  distributed. The NIG process has no diffusion component making it a pure jump
process with the L\'evy density
\begin{eqnarray}
    \label{NIG_NU}
    \nu(x)&=&\frac{2\alpha\delta}{\pi}\frac{\exp(\beta x)K_{1}(\alpha |x|)}{|x|}
\end{eqnarray}
where \( K_{1}(z) \) is the modified Bessel function of the third kind. Taking into account the asymptotic relations
\begin{eqnarray*}
    K_{1}(z)\asymp 2/z, \quad z\to +0 \mbox{ and }  K_{1}(z)\asymp \sqrt{\frac{\pi}{2z}} e^{-z}, \quad z\to +\infty,
\end{eqnarray*}
we conclude that   Assumption B  is fulfilled for  \( s=-1 \) and any \(p>0.\)

A Gamma process is used for the time change \( \mathcal{T} \) which is a L\'evy process such that its increments have a Gamma distribution, so that
\( \mathcal{T} \) is a pure-jump increasing L\'evy process with L\'evy density
\begin{eqnarray*}
    \nu_{\mathcal{T}}(x)=\theta x^{-1}\exp(-\eta x), \quad x\ge 0,
\end{eqnarray*}
where the parameter \( \theta  \) controls the rate of jump arrivals and the scaling parameter \( \eta  \) inversely controls the jump size.
The Laplace transform of \( \mathcal{T} \) is of the form
\begin{eqnarray*}
\mathscr{L}_{t}(z)=\E[\exp(-z\mathcal{T}(t))] =(1+z/\eta)^{-\theta t}, \quad \Re z\geq 0.
\end{eqnarray*}
It follows from the properties of the Gamma distribution that Assumption C (with (iv)) is fulfilled for the Gamma process \( \mathcal{T} \) for any \(p>0.\)
Introduce a time-changed L\'evy process 
\[ 
Y_{t}=\Sigma^{1/2} W_{\mathcal{T}(t)}+L_{\mathcal{T}(t)}, \quad t\geq 0,
\]
where \(W_t\) is a \(10\)-dimensional Wiener process, \(\Sigma\) is a \(10\times10\) positive semi-definite matrix, \( L_{t}=(L^{1}_{t},\ldots, L^{10}_{t}),\) \(t\geq0,\) is a \(10\)-dimensional L\'evy process with independent NIG components and \( \mathcal{T} \) is a Gamma process. Note that the process \( Y_{t} \) is a multi-dimensional L\'evy process since \( \mathcal{T} \) was itself a L\'evy process (subordinator). Let us be more specific and set
\begin{eqnarray*}
\Sigma=O^\top \Lambda O,
\end{eqnarray*}
where \(\Lambda:=\operatorname{diag}(1,0.5,0.1,0,\ldots,0)\in \mathbb{R}^{10}\times \mathbb{R}^{10}\) and \(O\) is a randomly generated \(10\times 10\) orthogonal matrix. 
We choose \( \Delta=1  \) and \( \operatorname{NIG}(1, -0.1, 1,-0.1) \) distributed increments of the coordinate L\'evy processes $L^{j},j=1,\dots,10$. Take also \( \theta=1 \) and \( \eta=1 \) for the parameters of the Gamma process \( \mathcal{T}. \) 

Simulating a discretised trajectory \(Y_{0},Y_{1},\ldots,Y_{n}\) of the process \( (Y_{t}) \) of the length \(n\), we estimate the covariance matrix \(\Sigma\) with $\hat\Sigma_{n,\lambda}$  from \eqref{eq:sigmaHat}. We solve the convex optimisation problem
\begin{equation}
\hat{\Sigma}_{n,\lambda}:=\argmin_{M\in\{A\in\R^{10\times10}:A=A^\top\}}\left\{ \int_{\mathbb{R}^{d}}\big(2|u|^{-2}\Re\hat{\psi}_{n}(u)+\langle\Theta(u),M\rangle\big)^{2}w_{U}(u)\,\d u+\lambda\|M\|_{1}\right\}
\label{eq:optl1}
\end{equation}
where $\hat{\psi}_{n}(u):=-\mathscr{L}^{-1}(\phi_{n}(u))$ and $\phi_{n}(u)=\frac{1}{n}\sum_{j=1}^{n}e^{i\langle u,Y_{j}\rangle}$. The integral in \eqref{eq:optl1} is approximated by a Monte Carlo algorithm using \(70\) independent draws from the uniform distribution on \([-U,U]^d.\)
In order to ensure that the estimate \(\hat{\Sigma}_{n,\lambda}\) is a positive semi-definite
matrix, we compute  the nearest positive definite matrix \(\widetilde{\Sigma}_{n,\lambda}\) which approximates  \(\hat{\Sigma}_{n,\lambda}.\) To find such an approximation we use \(\textsf{R}\) package \textsf{Matrix} (function \textsf{nearPD}).
\begin{figure}[!tp]
\centering
\includegraphics[width=0.4\linewidth]{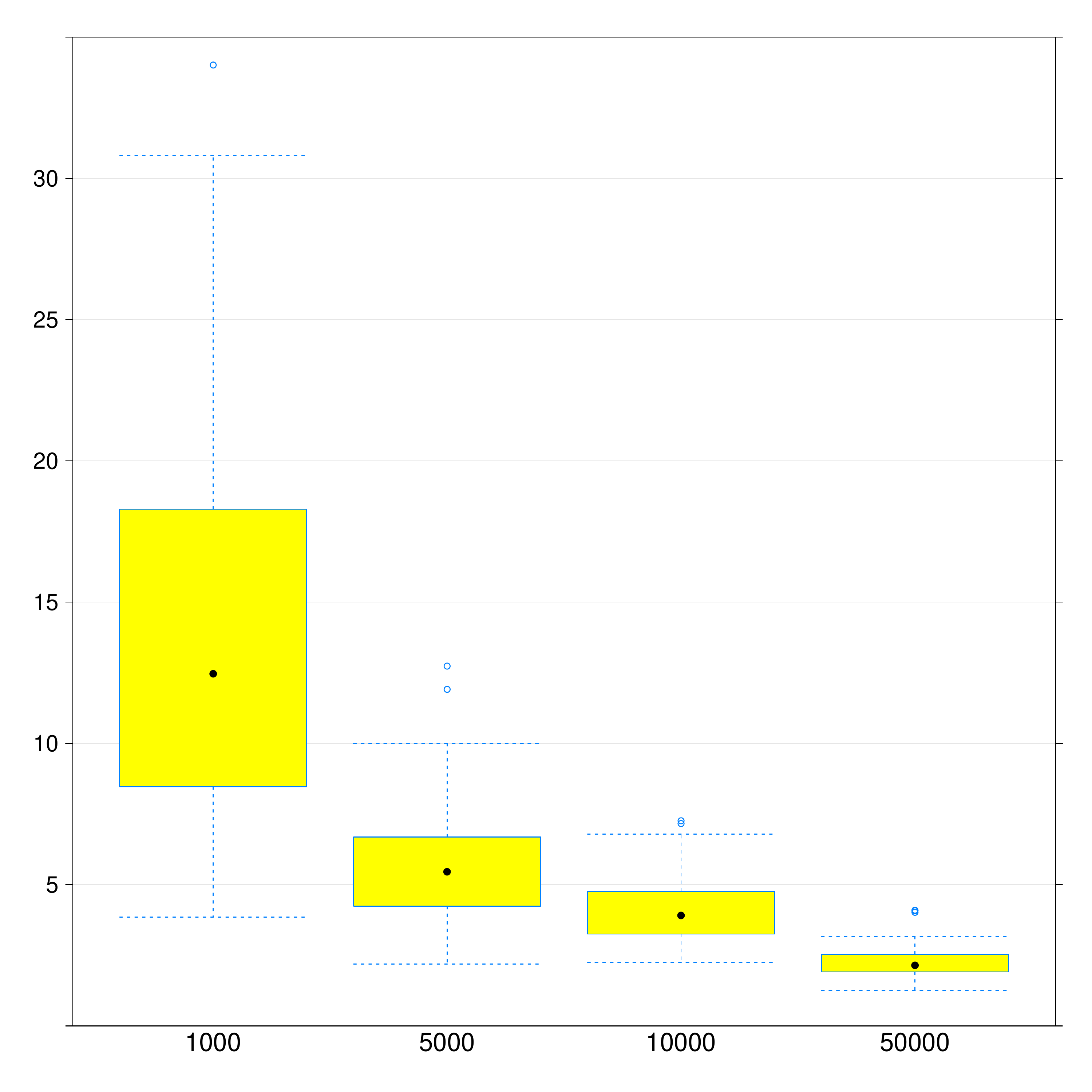}
~\includegraphics[width=0.4\linewidth]{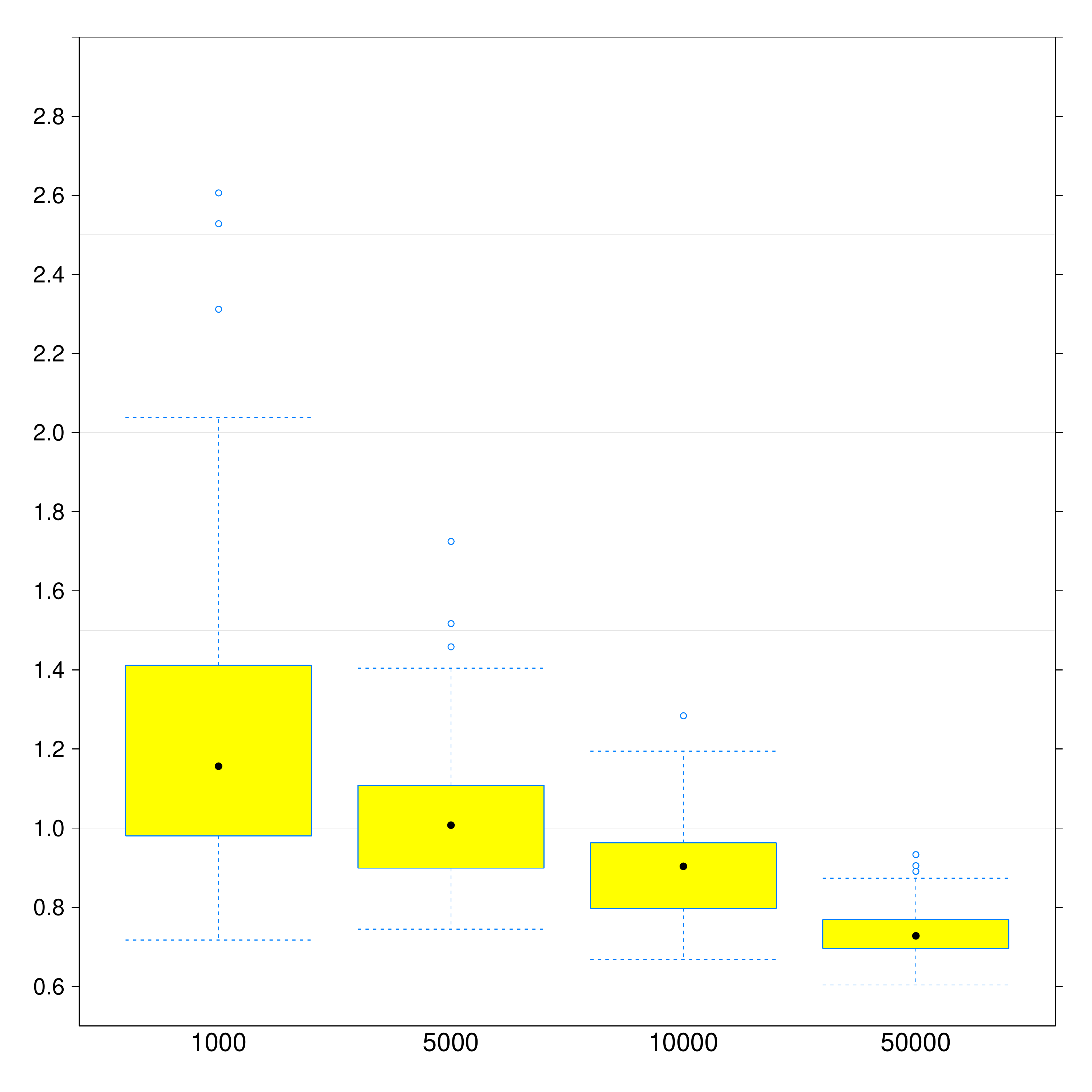}
 \caption{The relative estimation error \(\|\widetilde{\Sigma}_{n,\lambda}-\Sigma\|_{2}/\|\Sigma\|_2\) in dependence on \(n\in\{1000,5000,10000,50000\}\) for  \(\lambda\equiv 0\) (left) and \(\lambda_n=10^3/n\) (right).
\label{fig:gamma_change}}
\end{figure}
\begin{figure}[!tp]
\centering
\includegraphics[width=\linewidth]{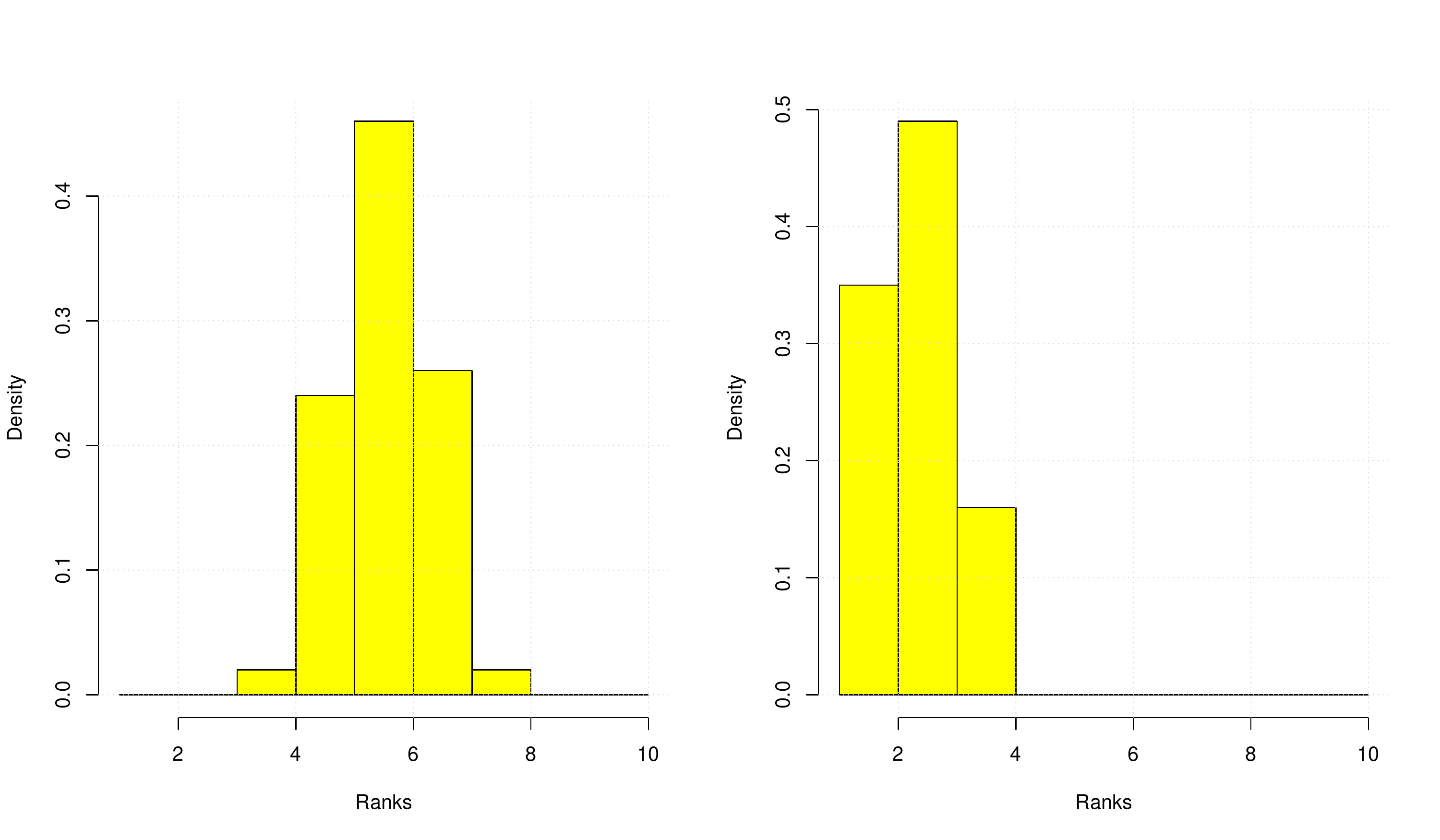}
 \caption{The histograms (over \(100\) simulation runs) of ranks of  the estimates \(\widetilde{\Sigma}_{n,\lambda}\) for  \(\lambda\equiv 0\) (left) and \(\lambda_n=0.2\) (right) in dimension \(d=10\) (the rank of \(\Sigma\) is \(3\)).
\label{fig: ranks}}
\end{figure}

We choose the tuning parameters according to the rule of  thumb \(U_n=0.7\cdot n^{1/4}.\) A data-driven selection procedure adapting to the unknown quantities $s$ and $\eta$ would be interesting, but is behond the scope of this paper. In view of the decompostion \eqref{eq:decomp} and our concentration inequalities for $\phi_n$, a model selection procedure or a Lepski method could be developed, cf. \cite{kappus2014} or, in the very similar deconvolution setting with unknown error distribution, \cite{dattnerEtAl2016}.

Figure~\ref{fig:gamma_change} shows box plots of the relative estimation error $\|\widetilde{\Sigma}_{n,\lambda}-\Sigma\|_{2}/\|\Sigma\|_2$ for sample sizes $n\in\{1000,5000,10000,50000\}$ based on 100 simulation iterations. As one can see,  the nuclear norm penalisation stabilises the estimates especially for smaller sample sizes. Without penalisation the approximated mean squared error is 5 to 10 times larger than the estimation error with penalisation choosing $\lambda=10^3/n$. 

Next we look at the ranks of estimated matrices \(\widetilde{\Sigma}_{n,\lambda}.\) The corresponding histograms (obtained over \(100\) simulation runs each consisting of \(10000\) observations) are presented in Figure~\ref{fig: ranks}, where the ranks were computed using the function \textsf{rankMatrix} from the package \textsf{Matrix} with tolerance for testing of ``practically zero'' equal to \(10^{-6}.\) As expected the ranks of the estimated matrices \(\widetilde{\Sigma}_{n,\lambda}\) are significantly lower in the case of nuclear norm penalization (\(\lambda>0\)) and concentrate around the true rank.

Let us now increase the dimension of the matrix \(\Sigma\) to \(100\) by adding zeros to  eigenvalues, i.e., we set \(\Sigma=O^\top \Lambda O\) with \(\Lambda:=\operatorname{diag}(1,0.5,0.1,0,\ldots,0)\in \mathbb{R}^{100}\times \mathbb{R}^{100}\) and randomly generated orthogonal matrix \(O\in \mathbb{R}^{100}\times \mathbb{R}^{100}.\) We take  \(100\)-dimensional time-changed L\'evy process with independent NIG components and run our estimation algorithm with \(n=20000\) and \(\lambda=5.\) The Figure~\ref{fig: ranks_100} shows the histograms of ranks based on \(100\) repetitions of the estimation procedure. The relative estimation error under an optimal choice of \(\lambda \) is of order \(0.7.\) 
\begin{figure}[!tp]
\centering
\includegraphics[width=\linewidth]{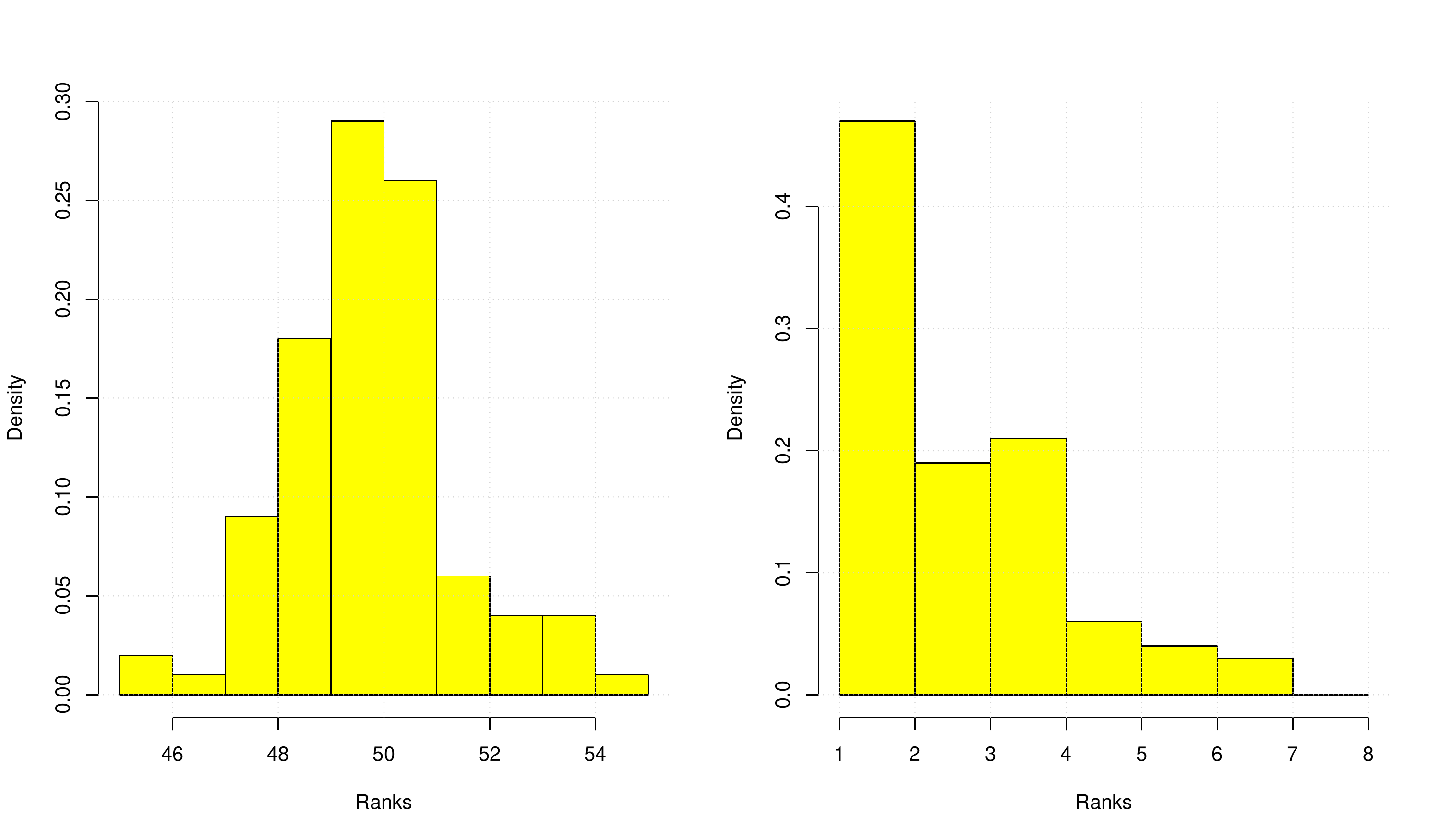}
 \caption{The histograms (over \(100\) simulation runs) of ranks of  the estimates \(\widetilde{\Sigma}_{n,\lambda}\) for  \(\lambda\equiv 0\) (left) and \(\lambda_n=10\) (right) in dimension \(d=100\) (the rank of \(\Sigma\) is \(3\)).
\label{fig: ranks_100}}
\end{figure}

\par

Let us now illustrate the importance of including the intercept \(a\) in the optimisation problem  \eqref{eq:sigmaTilde} for L\'evy processes with \(\alpha=\nu(\mathbb{R}^d)<\infty\). Consider a model of the type: 
\[ 
Y_{t}=\Sigma^{1/2} W_{\mathcal{T}(t)}+L_{\mathcal{T}(t)}, \quad t\geq 0,
\]
where \((W_t)\) is a \(10\)-dimensional Wiener process, \(\Sigma\) is a \(10\times10\) positive semi-definite matrix as before, \( L_{t}=(L^{1}_{t},\ldots, L^{10}_{t}),\) \(t\geq0,\) is a \(10\)-dimensional L\'evy process of compound Poisson type and \( \mathcal{T} \) is a Gamma process with parameters as before. In particular, we take  each L\'evy process \(L^{k}_{t},\) \(k=1,\ldots,10,\) to be of the form \(\sum^{N_t}_{i=1}\xi_i,\) where \((N_t)\) is a Poisson process with intensity \(1\) and \(\xi_1,\xi_2,\ldots\) are i.i.d standard normals. We compute and compare the estimates \(\widehat\Sigma\) and \(\widetilde\Sigma\) for different values of \(\lambda\) (\(n=10000\)). As can be seen from Figure~\ref{fig: cpp_error_intercept}, the estimate \(\widetilde\Sigma\) has much better performance than \(\widehat\Sigma.\) Moreover,    larger values of \(\lambda\) (stronger penalisation) can make the difference between \(\widehat\Sigma\) and \(\widetilde\Sigma\) less severe. 

 \begin{figure}[!tp]
\centering
\includegraphics[width=0.4\linewidth]{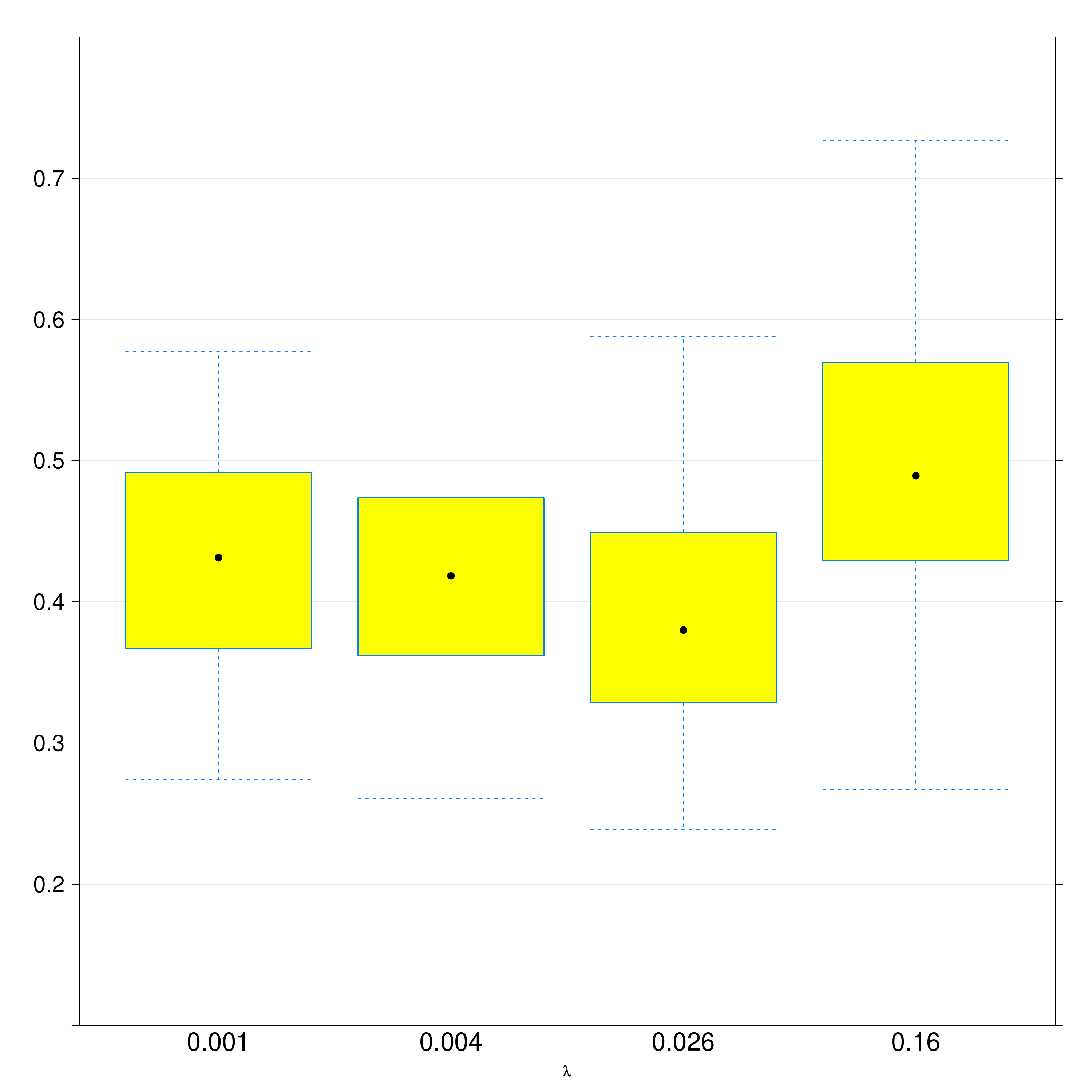}
\includegraphics[width=0.4\linewidth]{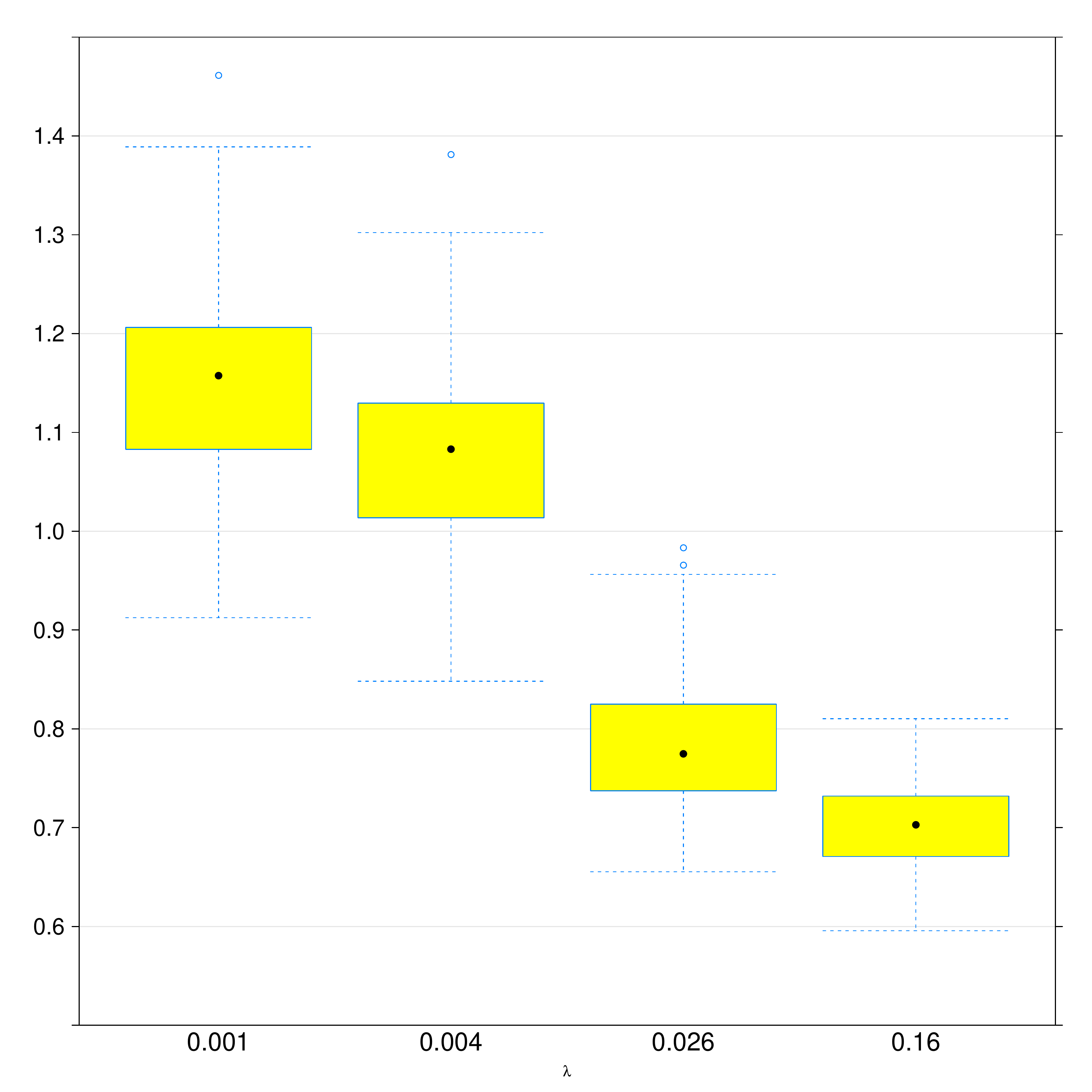}
 \caption{The relative estimation errors \(\|\widetilde{\Sigma}_{n,\lambda}-\Sigma\|_{2}/\|\Sigma\|_2\) (left) and  \(\|\widehat{\Sigma}_{n,\lambda}-\Sigma\|_{2}/\|\Sigma\|_2\)(right) in dependence on \(\lambda\) in the case \(\nu(\mathbb{R}^d)<\infty.\)
\label{fig: cpp_error_intercept}}
\end{figure}

 \subsection{Estimated time change}\label{sec:sim_lapl_est}

 Let us now illustrate the performance of our estimation algorithm in the case of unknown distribution of the time change where the Laplace tranform $\mathscr L$ is estimated according to Section~\ref{sec:lapl_est}. In practice we need to cut the expansion of $\mathscr L_m^{-1}$ from \eqref{eq: Linv}. We thus consider an approximation \(\mathscr L^{-1}_{m,J}(z)\) of the form
\begin{eqnarray}
\label{eq: Linv_app}
\mathscr L^{-1}_{m,J}(z):=\sum_{j=1}^J H_{j}\frac{(z-1)^j}{j !}
\end{eqnarray} 
for some integer \(J>1.\) Figure~\ref{fig: laplace_estimated} (left) shows a typical realisation of the difference  \(\mathscr L^{-1}_{m,J}(z)-\mathscr L^{-1}(z),\) \(z=x+iy,\) for \(m=100\) and \(J=20\) in the case of a Gamma subordinator with parameters as before. On the right hand side of Figure~\ref{fig: laplace_estimated}, one can see   box plots of the error \(\|\widehat{\Sigma}_{n,\lambda}-\Sigma\|_{2}/\|\Sigma\|_2\) when using \(\mathscr L^{-1}_{m,J}(z)\) instead of \(\mathscr L^{-1}(z)\) for different values of \(m\) and \(J=20,\) other parameters being the same as in the \(10\)-dimensional example of Section~\ref{sec:tc_known}.  As can be seen, already relatively small values of \(m\)  lead to reasonable estimates of \(\Sigma.\) Moreover by comparing Figure~\ref{fig: laplace_estimated} with Figure~\ref{fig:gamma_change}, one can even observe a slight improvement of the error rate  for \(\mathscr L^{-1}_{m,J}(z)\). This can be due to a regularising effect of the series truncation in \eqref{eq: Linv_app}.
\begin{figure}[tp]
\centering
\includegraphics[width=0.5\linewidth]{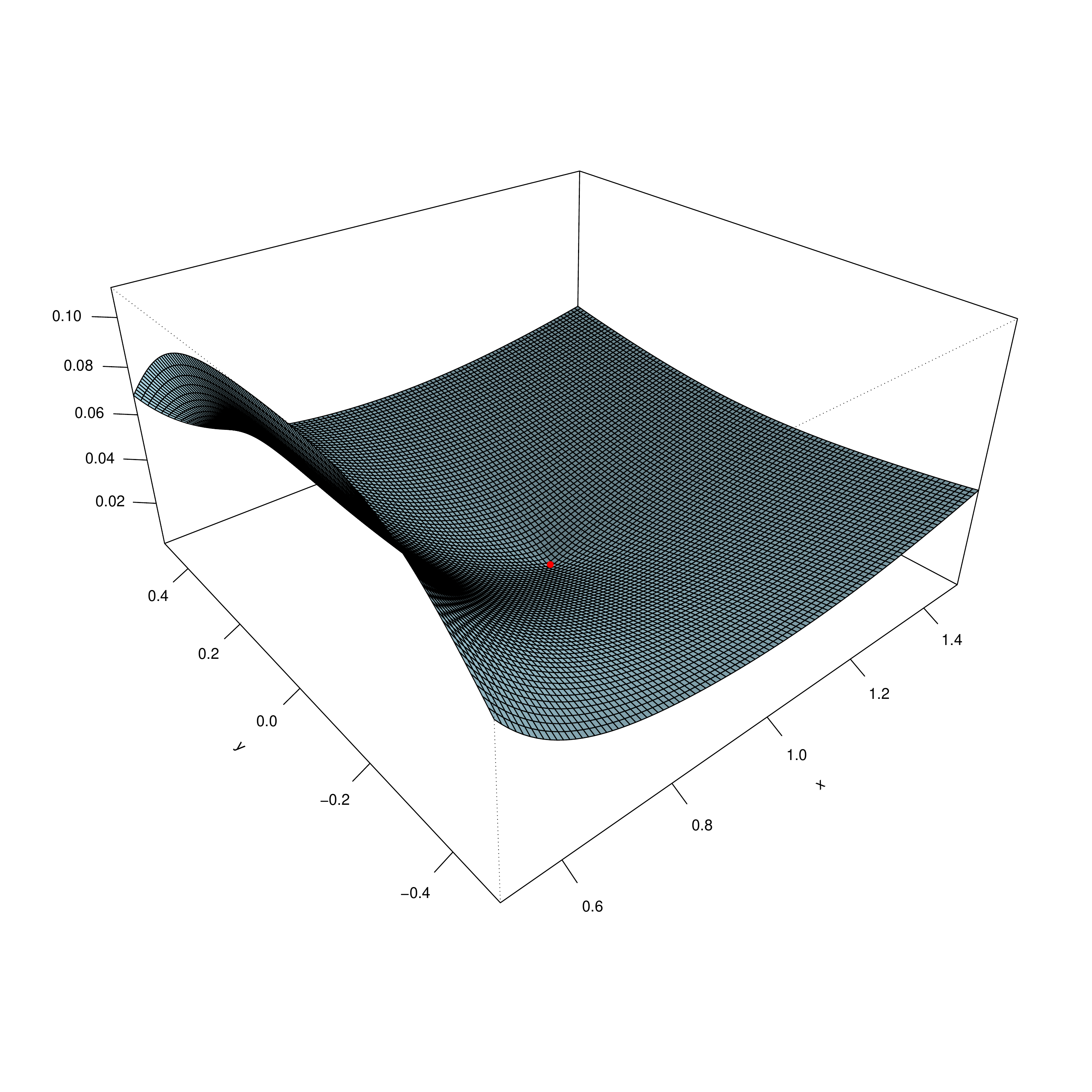}
\includegraphics[width=0.4\linewidth]{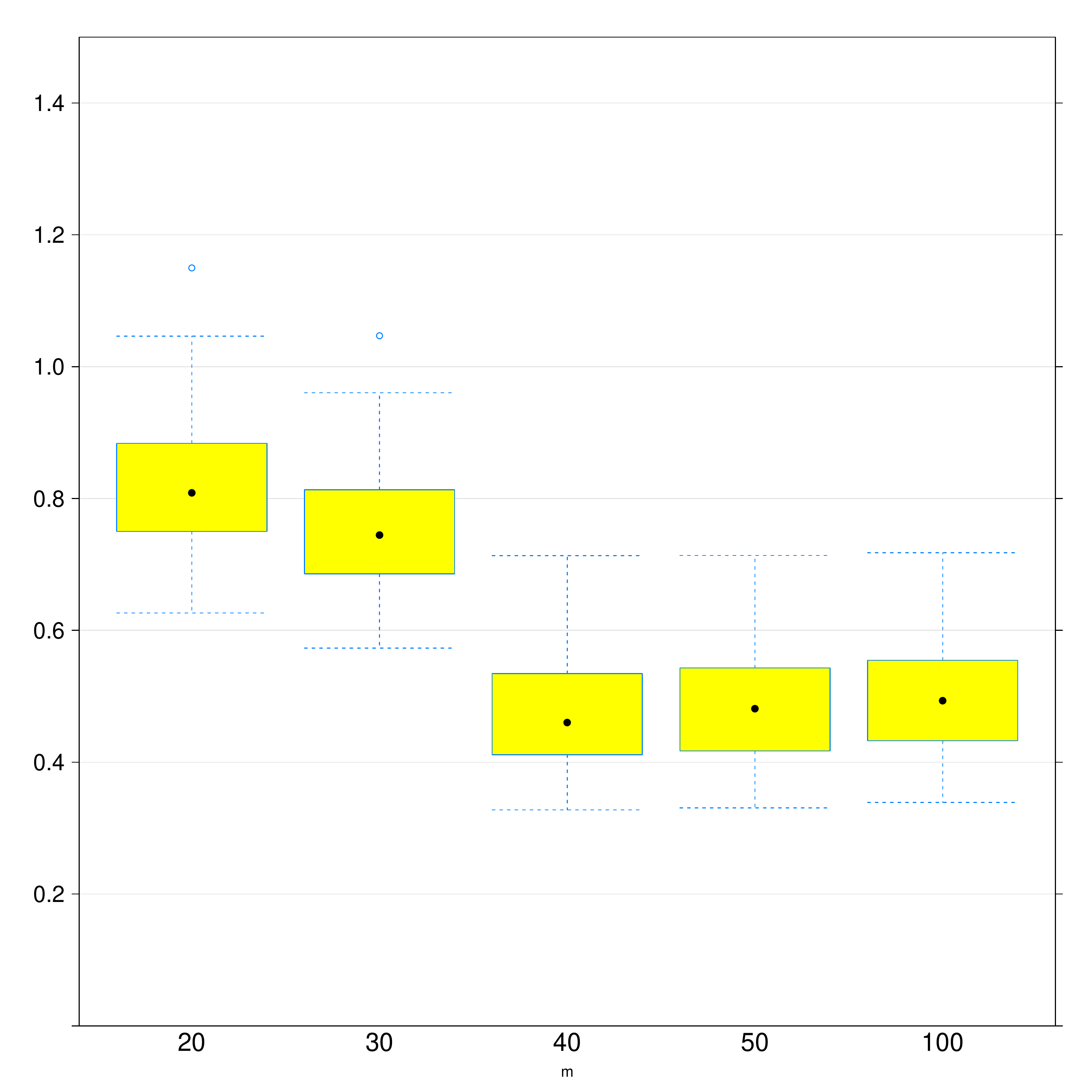}
 \caption{The difference between \(\mathscr L^{-1}_{m,J}(z)\)  and \(\mathscr L^{-1}(z)\)
 on the complex plane \(z=x+iy\) in the case of Gamma subordinator (right) and  box plots of the error \(\|\widehat{\Sigma}_{n,\lambda}-\Sigma\|_{2}/\|\Sigma\|_2\) when using \(\mathscr L^{-1}_{m,J}(z)\) instead of \(\mathscr L^{-1}(z)\) for different values of \(m\) and \(J=20.\)   
\label{fig: laplace_estimated}}
\end{figure}

\section{Proofs of the oracle inequalities}
\label{seq:proof_oracle}
\subsection{Proof of the isometry property: Lemma~\ref{lem:rip}}\label{sec:ProofRip}
  Let $A$ has rank $k$ and admit the singular value decomposition $A=QDQ^{\top}$ for the diagonal matrix $D=\diag(\lambda_{1},\ldots,\lambda_{k},0,\ldots,0)$ and an orthogonal matrix $Q$. Noting that $\langle\Theta(u),A\rangle\ge0$ for all $u\in\R^d$, we have
  \begin{align*}
  \left\Vert (A,a)\right\Vert _{w}^{2} 
   = & \int_{\mathbb{R}^{d}}\big(\langle\Theta(u),A\rangle+2U^2|u|^{-2}a\big)^2w_{U}(u)\d u \\
   \ge & \int_{\R^d}\big(\langle Q^{\top}u,DQ^{\top}u\rangle^{2}+4U^4a^2\big)\frac{w_U(u)}{|u|^{4}}\d u \\
   = & \int_{\R^d}\big((\lambda_{1}v_{1}^{2}+\ldots+\lambda_{k}v_{k}^{2})^2+4a^2\big)\frac{w(v)}{|v|^{4}}\d v \\
   \ge & \int_{\R^d}\big(\lambda_{1}^2v_{1}^{4}+\ldots+\lambda_{k}^2v_{k}^{4}+4a^2\big)\frac{w(v)}{|v|^{4}}\d v \\
   = & \varkappa_1\left\Vert A\right\Vert _{2}^{2}+\varkappa_2a^2,
  \end{align*}
  with $\varkappa_1:=\int_{\R^d}(v_{1}^{4}/|v|^{4})w(v)\d v$ and $\varkappa_2:=4\int_{\R^d}|v|^{-4}w(v)\d v\ge \varkappa_1$.
 On the other hand the Cauchy-Schwarz inequality yields
  \begin{align*}
    \left\Vert (A,a)\right\Vert _{w}^{2} 
   \le &2\int_{\mathbb{R}^{d}}\big(\langle\Theta(u),A\rangle^2+4U^4|u|^{-4}a^2\big)w_{U}(u)\d u \\
   \le & 2\|A\|_2^2\int_{\mathbb{R}^{d}}\|\Theta(u)\|_2^2w_U(u)\d u+8a^2\int_{\R^d} |v|^{-4}w(v)\d v\\
   = & 2\|w\|_{L^1}\|A\|_2^2+2\varkappa_2a^2,
  \end{align*}
  where the last equality follow from $\|\Theta(u)\|_2=1$ for any $u\neq1$ since $\Theta(u)=|u|^{-2}uu^\top$ has the spectrum $\{1,0,\dots,0\}$.
  \qed

\subsection{Proof of the first oracle inequality: Proposition~\ref{prop:oracle1}}\label{sec:ProofOracle1}
  For convenience define $\mathcal{Y}_{n}:=2|u|^{-2}\Re\hat{\psi}_{n}(u)$. It follows from the definition of $(\hat{\Sigma}_{n,\lambda},\hat\alpha_{n,\lambda})$ that for all $M\in\mathbb{M}$ and $a\ge0$
  \begin{align*}
  & \|(\hat{\Sigma}_{n,\lambda},\frac{\hat\alpha_{n,\lambda}}{U^2})\|_{w}^{2}-2\int_{\mathbb{R}^{d}}\mathcal{Y}_{n}(u)\langle\tilde\Theta(u),\diag(\hat{\Sigma}_{n,\lambda},\frac{\hat\alpha_{n,\lambda}}{U^2})\rangle w_{U}(u)\,\d u+\lambda(\|\hat{\Sigma}_{n,\lambda}\|_{1}+\frac{\hat\alpha_{n,\lambda}}{U^2})\\
  & \quad\le\|(M,\frac a{U^2})\|_{w}^{2}-2\int_{\mathbb{R}^{d}}\mathcal{Y}_{n}(u)\langle\tilde\Theta(u),\diag(M,\frac a{U^2})\rangle w_{U}(u)\d u +\lambda(\|M\|_{1}+\frac a{U^2}).
  \end{align*}
  Hence, 
  \begin{align*}
  &\|(\hat{\Sigma}_{n,\lambda},\frac{\hat\alpha_{n,\lambda}}{U^2})\|_{w}^{2}-2\langle(\Sigma,\frac\alpha{U^2}),(\hat{\Sigma}_{n,\lambda},\frac{\hat\alpha_{n,\lambda}}{U^2})\rangle_{w}+\lambda(\|\hat{\Sigma}_{n,\lambda}\|_{1}+\frac{\hat\alpha_{n,\lambda}}{U^2})\\
  &\qquad\le\|(M,\frac a{U^2})\|_{w}^{2}-2\langle(\Sigma,\frac\alpha{U^2}),(M,\frac a{U^2})\rangle_{w}+\lambda(\|M\|_{1}+\frac a{U^2})+\langle\mathcal{R}_{n},\diag(\hat{\Sigma}_{n,\lambda}-M,\frac {\hat\alpha_{n,\lambda}-a}{U^2})\rangle
  \end{align*}
  Due to the trace duality we have on the good event that $|\langle\mathcal{R}_{n},\diag(\hat{\Sigma}_{n,\lambda}-M,U^{-2}(\hat\alpha_{n,\lambda}-a))\rangle|\le\lambda\|(\hat{\Sigma}_{n,\lambda}-M\|_{1}+U^{-2}|\hat\alpha_{n,\lambda}-a|)$ and as a result 
  \begin{align*}
  &\|(\hat{\Sigma}_{n,\lambda}-\Sigma,U^{-2}(\hat\alpha_{n,\lambda}-\alpha))\|_{w}^{2} \\
   &\quad \le  \|(M-\Sigma,U^{-2}(a-\alpha))\|_{w}^{2}+\lambda\big(\|\hat{\Sigma}_{n,\lambda}-M\|_{1}+\|M\|_{1}-\|\hat{\Sigma}_{n,\lambda}\|_{1}+U^{-2}(|\hat\alpha_{n,\lambda}-a|+a-\hat\alpha_{n,\lambda})\big)\\
  &\quad \le  \|(M-\Sigma,U^{-2}(a-\alpha))\|_{w}^{2}+2\lambda(\|M\|_{1}+U^{-2}a).
  \end{align*}
  Consequently, choosing $a=\alpha$ yields
  \begin{equation*}
  \|(\hat{\Sigma}_{n,\lambda}-\Sigma,U^{-2}(\hat\alpha_{n,\lambda}-\alpha))\|_{w}^{2}\le\inf_{M\in\mathbb{M}}\left\{\|M-\Sigma\|_{w}^{2}+2\lambda(\|M\|_{1}+U^{-2}\alpha)\right\}.\tag*{\qed}
  \end{equation*}

\subsection{Proof of the second oracle inequality: Theorem~\ref{thm:oracle2}}\label{sec:proofOracle}

We first introduce some abbreviations. Define $\bar{\mathbb M}:=\{\diag(M,a):M\in\mathbb M,a\in I\}$ with elements $\bar M=\diag(M,a)\in\bar{\mathbb M}$ and the convention $\|\bar M\|_w:=\|(M,a)\|_w$. Note that convexity of $\mathbb M$ and $I$ implies convexity of $\bar{ \mathbb M}$. Moreover, we write $\bar \Sigma_{n,\lambda}:=\diag(\tilde\Sigma_{n,\lambda},U^{-2}\hat\alpha_{n,\lambda})$ and $\bar\Sigma=\diag(\Sigma,U^{-2}\alpha)$. As above we use $\mathcal{Y}_{n}:=2|u|^{-2}\Re\hat{\psi}_{n}(u)$.

The proof borrows some ideas and notation from the proof of Theorem
1 in \cite{koltchinskii2011nuclear}. First note that a necessary
condition of extremum in \eqref{eq:sigmaTilde} implies that there is
$V_{n}\in\partial\|\bar{\Sigma}_{n,\lambda}\|_{1}$, i.e. $V_{n}$
is an element of the subgradient of the nuclear norm, such that the
matrix 
\[
A=2\int_{\mathbb{R}^{d}}\tilde\Theta(u)\left[\mathcal{Y}_{n}(u)-\langle\tilde\Theta(u),\bar{\Sigma}_{n,\lambda}\rangle\right]w_{U}(u)\,\d u-\lambda V_{n}
\]
belongs to the normal cone at the point $\bar{\Sigma}_{n,\lambda}$
i.e. $\langle A,\bar{\Sigma}_{n,\lambda}-\bar M\rangle\geq0$ for all $\bar M\in\bar{\mathbb{M}}.$
Hence 
\[
2\int_{\mathbb{R}^{d}}\langle\tilde\Theta(u),\bar{\Sigma}_{n,\lambda}-\bar M\rangle\left[\mathcal{Y}_{n}(u)-\langle\tilde\Theta(u),\bar{\Sigma}_{n,\lambda}\rangle\right]w_{U}(u)\,\d u-\lambda\langle V_{n},\bar{\Sigma}_{n,\lambda}-\bar M\rangle\geq0
\]
or equivalently 
\[
2\int_{\mathbb{R}^{d}}\langle\tilde\Theta(u),\bar{\Sigma}_{n,\lambda}-\bar M\rangle\mathcal{Y}_{n}(u)w_{U}(u)\,\d u-2\langle\bar{\Sigma}_{n,\lambda},\bar{\Sigma}_{n,\lambda}-\bar M\rangle_{w}-\lambda\langle V_{n},\bar{\Sigma}_{n,\lambda}-\bar M\rangle\geq0.
\]
Furthermore 
\begin{multline*}
2\int_{\mathbb{R}^{d}}\langle\tilde\Theta(u),\bar{\Sigma}_{n,\lambda}-\bar M\rangle[\mathcal{Y}_{n}(u)-\langle\tilde\Theta(u),\bar\Sigma\rangle]w_{U}(u)\,\d u\\
-2\langle\bar{\Sigma}_{n,\lambda}-\bar\Sigma,\bar{\Sigma}_{n,\lambda}-\bar M\rangle_{w}-\lambda\langle V_{n}-V,\bar{\Sigma}_{n,\lambda}-\bar M\rangle\geq\lambda\langle V,\bar{\Sigma}_{n,\lambda}-\bar M\rangle
\end{multline*}
for any $V\in\partial\|\bar M\|_{1}.$ Fix some $\bar M\in\bar{\mathbb{M}}$ of rank
$r$ with the spectral representation 
\begin{eqnarray*}
\bar M=\sum_{j=1}^{r}\sigma_{j}u_{j}v_{j}^{\top},
\end{eqnarray*}
where $\sigma_{1}\geq\ldots\geq\sigma_{r}>0$ are singular values
of $\bar M.$ Due to the representation for the subdifferential of the
mapping $\bar M\to\|\bar M\|_{1}$ (see \citet{watson1992characterization})
\begin{eqnarray*}
\partial\|\bar M\|_{1}=\Big\{ \sum_{j=1}^{r}u_{j}v_{j}^{\top}+\Pi_{S_{1}^{\top}}\Lambda\Pi_{S_{2}^{\top}}:\|\Lambda\|_{\infty}\le1\Big\} ,
\end{eqnarray*}
where $(S_{1},S_{2})=\left(\spn(u_{1},\ldots,u_{r}),\spn(v_{1},\ldots,v_{r})\right)$
is the support of $\bar M$, $\Pi_{S}$ is a projector on the linear vector subspace $S$, we get 
\begin{eqnarray*}
V=\sum_{j=1}^{r}u_{j}v_{j}^{\top}+\Pi_{S_{1}^{\top}}\Lambda\Pi_{S_{2}^{\top}},\text{ for some }\Lambda\text{ with }\|\Lambda\|_{\infty}\le1.
\end{eqnarray*}
By the trace duality, there is $\Lambda$ such that 
\begin{eqnarray*}
\langle\Pi_{S_{1}^{\top}}\Lambda\Pi_{S_{2}^{\top}},\bar{\Sigma}_{n,\lambda}-\bar M\rangle=\langle\Pi_{S_{1}^{\top}}\Lambda\Pi_{S_{2}^{\top}},\bar{\Sigma}_{n,\lambda}\rangle=\langle\Lambda,\Pi_{S_{1}^{\top}}\bar{\Sigma}_{n,\lambda}\Pi_{S_{2}^{\top}}\rangle=\|\Pi_{S_{1}^{\top}}\bar{\Sigma}_{n,\lambda}\Pi_{S_{2}^{\top}}\|_{1}
\end{eqnarray*}
and for this $\Lambda$
\begin{align*}
&2\int_{\mathbb{R}^{d}}\langle\tilde\Theta(u),\bar{\Sigma}_{n,\lambda}-\bar M\rangle[\mathcal{Y}_{n}(u)-\langle\tilde\Theta(u),\bar\Sigma\rangle]w_{U}(u)\d u
-2\langle\bar{\Sigma}_{n,\lambda}-\bar\Sigma,\bar{\Sigma}_{n,\lambda}-\bar M\rangle_{w}\\
&\qquad\ge\lambda\Big\langle \sum_{j=1}^{r}u_{j}v_{j}^{\top},\bar{\Sigma}_{n,\lambda}-\bar M\Big\rangle +\lambda\|\Pi_{S_{1}^{\top}}\bar{\Sigma}_{n,\lambda}\Pi_{S_{2}^{\top}}\|_{1}
\end{align*}
since $\langle V_{n}-V,\bar{\Sigma}_{n,\lambda}-\bar M\rangle=\langle V_{n},\bar{\Sigma}_{n,\lambda}-\bar M\rangle+\langle V,\bar M-\bar{\Sigma}_{n,\lambda}\rangle\geq0$ ($V\in\partial\|\bar M\|_{1}$ and $V_{n}\in\partial\|\bar{\Sigma}_{n,\lambda}\|_{1}$). 
Using the identities 
\[
\Big\Vert \sum_{j=1}^{r}u_{j}v_{j}^{\top}\Big\Vert _{\infty}=1,\quad\Big\langle \sum_{j=1}^{r}u_{j}v_{j}^{\top},\bar{\Sigma}_{n,\lambda}-\bar M\Big\rangle =\Big\langle \sum_{j=1}^{r}u_{j}v_{j}^{\top},\Pi_{S_{1}}(\bar{\Sigma}_{n,\lambda}-\bar M)\Pi_{S_{2}}\Big\rangle ,
\]
we deduce 
\begin{align*}
&\|\bar{\Sigma}_{n,\lambda}-\bar \Sigma\|_{w}^{2}+\|\bar{\Sigma}_{n,\lambda}-\bar M\|_{w}^{2}+\lambda\|\Pi_{S_{1}^{\top}}\bar{\Sigma}_{n,\lambda}\Pi_{S_{2}^{\top}}\|_{1}\\
\le&
\|\bar M-\bar\Sigma\|_{w}^{2}+\lambda\|\Pi_{S_{1}}(\bar{\Sigma}_{n,\lambda}-\bar M)\Pi_{S_{2}}\|_{1}-\langle\mathcal R_n,\bar{\Sigma}_{n,\lambda}-\bar M\rangle.
\end{align*}
On the good event the trace duality yields
\begin{align*}
\big|\langle\mathcal R_n,\bar{\Sigma}_{n,\lambda}-\bar M\rangle\big|&\le\big|\langle\mathcal R_n,\pi_{\bar M}(\bar{\Sigma}_{n,\lambda}-\bar M)\rangle\big|+\big|\langle\mathcal R_n,\Pi_{S_{1}^{\top}}\bar{\Sigma}_{n,\lambda}\Pi_{S_{2}^{\top}}\rangle\big|\\
&\le \lambda\|\pi_{\bar M}(\bar{\Sigma}_{n,\lambda}-\bar M)\|_1+\lambda\|\Pi_{S_{1}^{\top}}\bar{\Sigma}_{n,\lambda}\Pi_{S_{2}^{\top}}\|_1
\end{align*}
with $\pi_{\bar M}(A)=A-\Pi_{S_{1}^{\top}}A\Pi_{S_{2}^{\top}}.$ 
Next by the Cauchy-Schwarz inequality, 
\begin{align*}
\|\Pi_{S_{1}}(\bar{\Sigma}_{n,\lambda}-\bar M)\Pi_{S_{2}}\|_{1}
\le\sqrt{\operatorname{rank}(\bar M)}\|\Pi_{S_{1}}(\bar{\Sigma}_{n,\lambda}-\bar M)\Pi_{S_{2}}\|_{2}
\le\sqrt{\operatorname{rank}(\bar M)}\|\bar{\Sigma}_{n,\lambda}-\bar M\|_{2},
\end{align*}
as well as 
\begin{align*}
\|\pi_{\bar M}(\bar{\Sigma}_{n,\lambda}-\bar M)\|_{1} & \le \sqrt{\operatorname{rank}(\pi_{\bar M}(\bar{\Sigma}_{n,\lambda}-\bar M))}\|\pi_{\bar M}(\bar{\Sigma}_{n,\lambda}-\bar M)\|_{2}\\
 & \le \sqrt{2\operatorname{rank}(\bar M)}\|\bar{\Sigma}_{n,\lambda}-\bar M\|_{2},
\end{align*}
since $\pi_{\bar M}(A)=\Pi_{S_{1}^{\top}}A\Pi_{S_{2}}+\Pi_{S_{1}}A$. Combining the above inequalities, we get 
\begin{align*}
&\|\bar{\Sigma}_{n,\lambda}-\bar\Sigma\|_{w}^{2}+\|\bar{\Sigma}_{n,\lambda}-\bar M\|_{w}^{2}+\lambda\|\Pi_{S_{1}^{\top}}\bar{\Sigma}_{n,\lambda}\Pi_{S_{2}^{\top}}\|_{1}\\
\le&\|\bar M-\bar\Sigma\|_{w}^{2}+\lambda\sqrt{\operatorname{rank}(\bar M)}\|\bar{\Sigma}_{n,\lambda}-\bar M\|_{2}+
\lambda\sqrt{2\operatorname{rank}(\bar M)}\|\bar{\Sigma}_{n,\lambda}-\bar M\|_{2}+\lambda\|\Pi_{S_{1}^{\top}}\bar{\Sigma}_{n,\lambda}\Pi_{S_{2}^{\top}}\|_{1}
\end{align*}
which is equivalent to
\[
 \|\bar{\Sigma}_{n,\lambda}-\bar\Sigma\|_{w}^{2}+\|\bar{\Sigma}_{n,\lambda}-\bar M\|_{w}^{2}
\le\|\bar M-\bar \Sigma\|_{w}^{2}+(1+\sqrt 2)\lambda\sqrt{\operatorname{rank}(\bar M)}\|\bar{\Sigma}_{n,\lambda}-\bar M\|_{2}.
\]
Finally, we choose $\bar M=\diag(M,U^{-2}\alpha)$ and apply Lemma~\ref{lem:rip} as well as the standard estimate $ab\le(a/2)^2+b^2$ for any $a,b\in\R$ to obtain on the good event
\begin{equation*}
  \|\bar{\Sigma}_{n,\lambda}-\bar\Sigma\|_{w}^{2}\le\|M-\Sigma\|_{w}^{2}+(\tfrac{1+\sqrt 2}2)^2\underline\varkappa_w^{-1}\lambda^2(\operatorname{rank}(M)+\1_{\alpha\neq0}).\tag*{\qed}
\end{equation*}

\section{Proof of the convergence rates}
\label{seq:proof_conv}
\subsection{Proof of the upper bound: Theorem \ref{thm:conR}}

We start with an auxiliary lemma:
\begin{lem}\label{lem:moments} Let $L=\{L_{t}:t\ge0\}$ be a $d$-dimensional
  L\'evy process with characteristic triplet $(A,c,\nu)$. If $\int_{|x|\ge1}|x|^{p}\nu(\d x)$
  for $p\ge1$, then $\E[|L_{t}|^{p}]=\mathcal{O}(t^{p/n}\vee t^{p})$
  for any $t>0$ where $n\in\N$ is the smallest even natural number
  satisfying $n\ge p$. 
\end{lem} 
\begin{proof} 
  We decompose the L\'evy
  process $L=M+N$ into two independent L\'evy processes $M$ and $N$
  with characteristics $(A,c,\nu\1_{\{|x|\le1\}})$ and $(0,0,\nu\1_{\{|x|>1\}})$,
  respectively. The triangle inequality yields 
  \[
  \E[|L_{t}|^{p}]\lesssim\E[|M_{t}|^{p}]+\E[|N_{t}|^{p}].
  \]
  Let $n\ge p$ be an even integer. Due to the compactly supported jump
  measure of $M$, the $n$-th moment of $M_{t}$ is finite and can
  be written as a polynomial of degree $n$ of its first $n$ cumulants.
  Since all cumulants of $M_{t}$ are linear in $t$, we conclude with
  Jensen's inequality 
  \[
  \E[|M_{t}|^{p}]\le\E[M_{t}^{n}]^{p/n}\lesssim\big(t+t^{n})^{p/n}\lesssim t^{p/n}\vee t^{p}.
  \]
  Since $N$ is a compound Poisson process, it can be represented as
  $N_{t}=\sum_{k=1}^{P_{t}}Z_{k}$ for a Poisson process $P_{t}$ with
  intensity $\lambda=\nu(\{|x|>1\})<\infty$ and a sequence of i.i.d. random variables $(Z_{k})_{k\ge1}$
  which is independent of $P$. Hence, 
  \[
  \E[|N_{t}|^{p}]\le\E\Big[\sum_{k=1}^{P_{t}}|Z_{k}|^{p}\Big]\le\E[P_{t}]\E[|Z_{1}|^{p}]=t\lambda\E[|Z_{1}|^{p}].
  \]
  Note that $\E[|Z_{1}|^{p}]$ is finite owing to the assumption $\int_{|x|\ge1}|x|^{p}\nu(\d x)<\infty$.
\end{proof} 
\begin{remark} While the upper bound $t^{p}$ is natural,
the order $t^{p/n}$ is sharp too in the sense that for a Brownian
motion $L_{t}=W_{t}$ and $p=1$ we have $\E[|W_{t}|]=t^{1/2}\E[|W_{1}|]$.
\end{remark}

To show the upper bound, we start with bounding the approximation error term in the decomposition \eqref{eq:decomp}.
\begin{lem}\label{lem:ApproxError}
  If the jump measure $\nu$ of $Z$ satisfies Assumption~\ref{ass:jumps}(i)
  for some $s\in(-2,\infty),C_{\nu}>0$, then
  \begin{equation}
  \int_{\R^{d}}\frac{|\Re\Psi(u)+\alpha|}{|u|^2}w_{U}(u)\d u
  \le C_{s}|U_n|^{-s-2},\label{REst}
  \end{equation}
  where $C_s:=2C_{\nu}\int_{\{1/4\le|v|\le1/2\}}|v|^{-s-2}w(v)\d v.$
\end{lem} 
\begin{proof} 
  Let us start with the case $s\in(-2,0)$ and $\alpha=0$. For all $u\in\mathbb{R}^{d}\setminus\{0\}$ we have 
  \begin{align*}
  \left|\Psi(u)\right|  \le & \int_{\R^{d}}\big|e^{i\langle x,u\rangle}-1-i\langle x,u\rangle\1_{\{|x|\le1\}}(x)\big|\nu(\d x)
  \le  \int_{\mathbb{R}^{d}}\Big(\frac{\langle x,u\rangle^{2}}{2}\wedge2\Big)\nu(\d x)\\
  \le & 2^{1+s}\int_{\mathbb{R}^{d}}\left|\left\langle x,u\right\rangle \right|^{|s|}\nu(\d x)
   \le  2^{1+s}C_{\nu}|u|^{-s}.
  \end{align*}
  Hence, we obtain
  \begin{align*}
  \int_{\R^d}\frac{|\Re\Psi(u)|}{|u|^2}w_{U}(u)\,\d u 
  & \le 2^{1+s}C_\nu\int_{\R^d}|u|^{-s-2}|w_U(u)\,\d u\\
  & \le U^{-s-2}2^{1+s}C_{\nu}\int_{\{1/4\le|v|\le1/2\}}|v|^{-s-2}w(v)\d v.
  \end{align*}
  In the case $s\ge0$ and $\alpha=\nu(\R^d)<\infty$, we have $\Re\Psi(u)+\alpha=\Re(\F\nu)(u),u\in\R^d,$ such that
  \begin{align*}
    \int_{\R^d}\frac{|\Re\Psi(u)+\alpha|}{|u|^2}w_{U}(u)\,\d u
    &\le C_\nu\int_{\R^d}|u|^{-2}(1+|u|^2)^{-s/2}w_U(u)\,\d u\\
    &\le U^{-s-2} C_\nu\int_{1/4\le|u|\le1/2}|v|^{-s-2}w_U(u)\,\d u.\tag*{\qedhere}
  \end{align*}

\end{proof} 

In order the bound the stochastic error term in \eqref{eq:decomp}, we apply the following linearisation lemma. We denote throughout
$\|f\|_{U}:=\sup_{|u|\le U}|f(u)|$. 
\begin{lem}\label{lem:linearisation}
  Grant Assumption~\ref{ass:time}(iii) with $C_L>0$. For all $n\in\N$ and $U>0$ we have on the event
  \[
  \mathcal{H}_{n,U}:=\Big\{\|\phi_{n}-\phi\|_{U}\le\tfrac{2}{C_L}\inf_{|u|\le U}\mathscr{L}'(-\psi(u))\Big\}
  \]
  that for all $u\in\R^d$ with $|u|\le U$ it holds 
  \[
  \Big|\mathscr{L}^{-1}(\phi_{n}(u))-\mathscr{L}^{-1}(\phi_{n}(u))-\frac{\phi_{n}(u)-\phi(u)}{\mathscr{L}'(-\psi(u))}\Big|
  \le4C_L\frac{|\phi_{n}(u)-\phi(u)|^2}{|\mathscr{L}'(-\psi(u))|^2|\psi(u)|^q},
  \]
  where $q=1$ if Assumption~\ref{ass:time}(iv) is satisfied and $q=0$ otherwise.
\end{lem} 
\begin{proof} 
  First note that 
  \[
  (\mathscr{L}^{-1})'(z)=\frac{1}{\mathscr{L}'(\mathscr{L}^{-1}(z))}\quad\text{and}\quad(\mathscr{L}^{-1})''(z)=-\frac{\mathscr{L}''(\mathscr{L}^{-1}(z))}{\mathscr{L}'(\mathscr{L}^{-1}(z))^{3}}
  \]
  and in particular $(\mathscr{L}^{-1})'(\phi(u))=1/\mathscr{L}'(-\psi(u))$.
  Hence, the Taylor formula yields 
  \[
  \mathscr{L}^{-1}(\phi_{n}(u))-\mathscr{L}^{-1}(\phi_{n}(u))=\frac{\phi_{n}(u)-\phi(u)}{\mathscr{L}'(-\psi(u))}+R(u)
  \]
  with 
  \begin{align}
  |R(u)| & \le|\phi_{n}(u)-\phi(u)|^{2}|(\mathscr{L}^{-1})''(\phi(u)+\xi_{1}(\phi_{n}(u)-\phi(u)))|\notag\\
  & \le\frac{C_L|\phi_{n}(u)-\phi(u)|^{2}}{|\mathscr{L}'(\mathscr{L}^{-1}(\phi(u)+\xi_{1}(\phi_{n}(u)-\phi(u)))|^{2}}.\label{eq:estRem1}
  \end{align}
  for some intermediate point $\xi_{1}\in[0,1]$ depending on $u$. For another intermediate
  point $\xi_{2}\in[0,1]$ we estimate 
  \begin{align}
  & \big|\mathscr{L}'\big(\mathscr{L}^{-1}(\phi(u)+\xi_{1}(\phi_{n}(u)-\phi(u))\big)-\mathscr{L}'(-\psi(u))\big|\notag\\
  & \quad\le|\phi_{n}(u)-\phi(u)|\big|(\mathscr{L}'\circ\mathscr{L}^{-1})'(\phi(u)+\xi_{2}(\phi_{n}(u)-\phi(u)))\big|\notag\\
  & \quad=|\phi_{n}(u)-\phi(u)|\Big|\Big(\frac{\mathscr{L}''_{\Delta}\circ\mathscr{L}^{-1}}{\mathscr{L}'\circ\mathscr{L}^{-1}}\Big)(\phi(u)+\xi_{2}(\phi_{n}(u)-\phi(u)))\Big|\notag\\
  & \quad\le C_L|\phi_{n}(u)-\phi(u)|.\label{eq:estRem2}
  \end{align}
  Therefore, we have on the event $\mathcal{H}_n$ for any $u$ in the support of $w_U$
  \begin{align*}
  |R(u)| & \le C_L|\phi_{n}(u)-\phi(u)|^{2}\big(|\mathscr{L}'(-\psi(u))|-C_L|\phi_{n}(u)-\phi(u)|\big)^{-2}\le4C_L\frac{|\phi_{n}(u)-\phi(u)|^{2}}{|\mathscr{L}'(-\psi(u))|^{2}}.
  \end{align*}
  Under Assumption~\ref{ass:time}(iv) we can obtain a sharper estimate. More precisely, \eqref{eq:estRem1} and \eqref{eq:estRem2} imply together with the faster decay that
  \begin{align*}
    |R(u)| & \le\frac{4C_L|\phi_{n}(u)-\phi(u)|^{2}}{|\mathscr{L}'(-\psi(u))|^{2}(1+|\mathscr{L}^{-1}(\phi(u)+\xi_{1}(\phi_{n}(u)-\phi(u)))|)}.
  \end{align*}
  Using $\phi=\mathscr L(-\psi)$ and again \eqref{eq:estRem2}, we have on $\mathcal H_n$
  \begin{align*}
    &|\psi(u)+\mathscr{L}^{-1}(\phi(u)+\xi_{1}(\phi_{n}(u)-\phi(u)))|\\
    &\qquad\le |\phi_n(u)-\phi(u)||(\mathscr L^{-1})'(\phi(u)+\xi_{2}(\phi_{n}(u)-\phi(u)))|\\
    &\qquad\le \frac{|\phi_n(u)-\phi(u)|}{|\mathscr L'(\mathscr L^{-1}(\phi(u)+\xi_{2}(\phi_{n}(u)-\phi(u))))|}
    \le \frac{2|\phi_n(u)-\phi(u)|}{|\mathscr L'(-\psi(u))|}
    \le\frac4{C_L}.
  \end{align*}
  We conclude
  \begin{align*}
      |R(u)| & \le\frac{4C_L|\phi_{n}(u)-\phi(u)|^{2}}{|\mathscr{L}'(-\psi(u))|^{2}|\psi(u)|}.\tag*{\qedhere}
  \end{align*}  
\end{proof} 

Denoting the linearised stochastic error term by
\[
  L_n:=\int_{\R^{d}}|u|^{-2}\Re\Big(\frac{\phi_{n}(u)-\phi(u)}{\mathscr{L}'(-\psi(u))}\Big)\Theta(u) w_{U}(u)\d u\in\R^{d\times d},
\]
we obtain the following concentration inequality
\begin{lem}\label{lem:ConcLin}
  Define $\xi_U:=U^2\inf_{|u|\le U}\mathscr{L}'(-\psi(u))$. There is some $c>0$ depending only on $w$ such that for any $n\ge1$ and any $\kappa\in(0,(nU^d/\|\phi\|_{L^1(B_U^d)})^{1/2})$ 
  \[
    \mathbb P\Big(\|L_n\|_\infty\ge \frac{\kappa}{\sqrt n\xi_U}\sqrt{U^{-d}\|\phi\|_{L^1(B_U^d)}}\Big)\le 2de^{-c\kappa^2}.
  \]
\end{lem}
\begin{proof}
  We write
  \[
    L_n=\frac{1}{n}\sum_{j=1}^n C_j\quad\text{with}\quad C_j:=\int_{\R^{d}}\Re\Big(\frac{e^{i\langle u,Y_j\rangle}-\phi(u)}{|u|^2\mathscr{L}'(-\psi(u))}\Big)\Theta(u)w_{U}(u)\d u,
  \]
  where $C_j$ are independent, centred and symmetric matrices in $\R^{d\times d}$. In order to apply the noncommutative Bernstein inequality in \cite[Thm. 4]{recht2011}, we need to bound $\|C_j\|_\infty$ and $\|\E[C_j^2]\|_\infty$. Since $\|\Theta(u)\|_\infty=1$, we have
  \[
    \|C_j\|_\infty\le\int_{\R^{d}}\frac{2}{|u|^2|\mathscr{L}'(-\psi(u))|}w_{U}(u)\d u\le2\xi^{-1}_U\int_{\R^d}|v|^{-2}w(v)\d v. 
  \]
  Using that $\Re (z_1)\Re( z_2)=\frac12(\Re(z_1z_2)+\Re(z_1\bar z_2))$ for $z_1,z_2\in\C$ and symmetry in $v$, the Variance of $C_j$ is bounded as follows:
  \begin{align*}
  \E[C_j^2]
  = &\int_{\R^{d}}\int_{\R^{d}}\E\Big[\Re\Big(\frac{e^{i\langle u,Y_{1}\rangle}-\phi(u)}{\mathscr{L}'(-\psi(u))}\Big)
  \Re\Big(\frac{e^{i\langle v,Y_{1}\rangle}-\phi(v)}{\mathscr{L}'(-\psi(v))}\Big)\Big]\frac{\Theta(u)}{|u|^2}\frac{\Theta(v)}{|v|^2} w_{U}(u)w_{U}(v)\d u\d v\\
  = & \int_{\R^{d}}\int_{\R^{d}}\Re\Big(\frac{\phi(u+v)-\phi(u)\phi(v)}{\mathscr{L}'(-\psi(u))\mathscr{L}'(-\psi(v))}\Big)\frac{\Theta(u)}{|u|^2}\frac{\Theta(v)}{|v|^2} w_{U}(u) w_{U}(v)\d u\d v.
  \end{align*}
  To estimate $\|\E[C_j^2]\|_\infty$ we bound the spectral norm of the integral by the integral over the spectral norms (Minkowski inequality). Moreover, we use that for any functions $f\colon\R^d\to\C$ and $g\colon\R^d\times\R^d\to\C$ with $|g(u,v)|=|g(v,u)|$ the Cauchy-Schwarz inequality and Fubini's theorem yield
  \begin{align}
    \int_{\R^d}\int_{\R^d}|f(u)f(v)g(u,v)|\d u\d v
    &\le\big\|f(u)\sqrt{|g(u,v)|}\big\|_{L^2(\R^{2d})}\big\|f(v)\sqrt{|g(u,v)|}\big\|_{L^2(\R^{2d})}\notag\\
    &=\int_{\R^d} |f(u)|^2\int_{\R^d}|g(u,v)|\d v\d u.\label{eq:CStrick}
  \end{align}
  Taking into account the compact support of $w_U$ and applying the previous estimate to the functions $f(u)=w_{U}(u)/(|u|^{2}\mathscr{L}'(-\psi(u)))$ and $g(u,v)=(\phi(u+v)-\phi(u)\phi(v))\1_{\{|u|\le U/2\}}\1_{\{|v|\le U/2\}}$, we obtain
  \begin{align*}
  \|\E[C_j^2]\|_\infty\le & \int_{\R^{d}}\frac{w_{U}(u)^{2}}{|u|^{4}|\mathscr{L}'(-\psi(u))|^{2}}\int_{|v|\le U/2}\big(|\phi(u+v)|+|\phi(u)\phi(-v)|\big)\d v\,\d u\\
  \le&2\|\phi\|_{L^1(B_U^d)}\int_{\R^{d}}\frac{w_{U}(u)^{2}}{|u|^{4}\mathscr{L}'(-\psi(u))^{2}}\d u.
  \end{align*}
  Using $w_{U}(u)=U^{-d}w(u/U)$, we conclude
  \begin{align*}
  \|\E[C_j^2]\|_\infty\le & 2\|\phi\|_{L^1(B_U^d)}\Big(\inf_{|u|\le U}\mathscr{L}'(-\psi(u))\Big)^{-2}\int_{\R^{d}}|u|^{-4}w_{U}(u)^{2}\d u\\
  \le & 2\xi_U^{-2}U^{-d}\|\phi\|_{L^1(B_U^d)}\int_{\R^d}|v|^{-4}w(v)^2\d v.
  \end{align*}
  Consequently, Theorem~4 from \cite{recht2011} yields
  \begin{align*}
   &\mathbb P\Big(\|L_n\|_\infty\ge \frac{\kappa}{\sqrt n\xi_U}\frac{\|\phi\|^{1/2}_{L^1(B_U^d)}}{U^{d/2}}\Big)\le 2d\exp\Big(-\frac{c\kappa^2}{1+\kappa(\|\phi\|_{L^1(B_U^d)}/(nU^d))^{1/2}}\Big)
  \end{align*}
  for some constant $c>0$ depending only on $w$.
\end{proof}

\begin{proof}[Proof of Theorem~\ref{thm:conR}] 
We write again $q=1$ if Assumption~\ref{ass:time}(iv) is satisfied and $q=0$ otherwise. Applying Lemmas \ref{lem:ApproxError} and \ref{lem:linearisation} we deduce from \eqref{eq:decomp} on the event $\mathcal H_{n,U}$, defined the linearisation lemma, 
\begin{align*}
  \|\mathcal R_n\|_\infty
  \le&4\Big\|\int_{\R^{d}}\frac{1}{|u|^2}\Re\big(\mathscr{L}^{-1}(\phi_{n}(u))-\mathscr{L}^{-1}(\phi(u))\big)\Theta(u) w_{U}(u)\d u\Big\|_\infty\\
  &\qquad+4\int_{\R^{d}}\frac{|\Re\Psi(u)+\alpha|}{|u|^2}w_{U}(u)\d u\\
  \le& 4 \|L_n\|_\infty+ 16C_L\int_{\R^d}\frac{|\phi_n(u)-\phi(u)|^2}{|u|^{2}|\psi(u)|^q|\mathscr L'(-\psi(u))|^2}w_{U}(u)\d u+DC_\nu U^{-s-2}\\
  \le& 4 \|L_n\|_\infty+ 16C_L\frac{\|\phi_n-\phi\|_{U}^2}{\inf_{|u|\le U}|\psi(u)|^q|\mathscr L'(-\psi(u))|^2}\int_{\R^d}|u|^{-2}w_{U}(u)\d u+DC_\nu U^{-s-2}\\
  \le&4 \|L_n\|_\infty+DC_L\frac{\|\phi_n-\phi\|_{U}^2}{U^{2}\inf_{|u|\le U}|\psi(u)|^q|\mathscr L'(-\psi(u))|^2}+DC_\nu U^{-s-2}
\end{align*}
for some constant $D>0$ depending only on $w$ and $s$. Writing again $\xi_U:=U^2\inf_{|u|\le U}|\mathscr{L}'(-\psi(u))|$ and defining $\zeta_U:=U^{2}\inf_{|u|\le U}|\psi(u)|^q|\mathscr L'(-\psi(u))|^2$, we obtain
\begin{align}
  &\P\Big(\|\mathcal R_n\|_\infty\ge \frac{\kappa}{\sqrt n\xi_U}\sqrt{U^{-d}\|\phi\|_{L^1(B_U^d)}}+DC_\nu U^{-s-2}\Big)\notag\\
  \le&\P\Big(\|L_n\|_\infty\ge\frac{\kappa}{8}\frac{\|\phi\|_{L^1(B_U^d)}^{1/2}}{\sqrt nU^{d/2}\xi_U}\Big)
  +\P\Big(\|\phi_n-\phi\|_{U}^2\ge\frac{\kappa}{2DC_L}\frac{\|\phi\|_{L^1(B_U^d)}^{1/2}\zeta_U}{U^{d/2}\sqrt n\xi_U}\Big)+\P(\mathcal H_{n,U}^c).\label{eq:remProb}
\end{align}
The first probability is bounded by Lemma~\ref{lem:ConcLin}. Defining  
\[
  \delta_n:=\frac{\sqrt n\|\phi\|_{L^1(B_U^d)}^{1/2}\zeta_U}{U^{d/2}\xi_U}
  =\frac{\sqrt n\|\phi\|_{L^1(B_U^d)}^{1/2}\inf_{|u|\le U}|\psi(u)|^q|\mathscr L'(-\psi(u))|^2}{U^{d/2}\inf_{|u|\le U}|\mathscr{L}'(-\psi(u))|}
\]
and using that $\kappa\le\delta_n$ by assumption, we can bound the second probability in \eqref{eq:remProb} by Theorem~\ref{thm:concPhi}:
\[
  \P\Big(\big(\sqrt n\|\phi_n-\phi\|_{U}\big)^2\ge\frac{\kappa\delta_n}{2DC_L}\Big)
  \le\P\Big(\big(\sqrt n\|\phi_n-\phi\|_{U}\big)^2\ge\frac{\kappa^2}{2DC_L}\Big)
  \le 2 e^{-c\kappa^2},
\]
for some numerical constant $c>0$, provided $\kappa\ge \sqrt {d\log(d+1)}(\log U)^{\rho}$ for  some $\rho>1/2$ and $\kappa\le \sqrt n$. The probability of the complement of $\mathcal{H}_{n,U}$ can be similarly estimated by 
\begin{align*}
\P(\mathcal{H}_{n}^{c}) 
\le & \P\left(\sqrt{n}\left\Vert \phi_{n}-\phi\right\Vert _{U}\ge\tfrac C2\sqrt nU^{-2}\xi_U\right)
\end{align*}
which yields the claimed bound owing to Theorem~\ref{thm:concPhi} and $\kappa^2\le n\xi_U^2U^{-4}$.
\end{proof}

\subsection{Proof of the lower bounds: Theorem~\ref{thm:lowerBound}}

We follow the standard strategy to prove lower bounds adapting some ideas by \citet[Chap. 1]{belomestny2015estimation}. We start with the proof of (i) which is divided into several steps.

\emph{Step 1:} We need to construct two alternatives of L\'evy triplets. 
Let $K\colon\mathbb{R}^k\to\mathbb{R}$ be a kernel given via its Fourier transform 
\[
\mathcal{F}K(u)=\begin{cases}
1, & |u|\le1,\\
\exp\left(-\frac{e^{-1/(|u|-1)}}{2-|u|}\right), & 1<|u|<2,\\
0, & |u|\geq2,
\end{cases}\quad u\in\R^k.
\]
Since $\mathcal F K$ is real and even, $K$ is indeed a real valued function. For each $n$ we define two jump measures $\nu_0$ and $\nu_n$ on $\R^d$ via their Lebesgue density, likewise denoted by $\nu_0$ and $\nu_n$, respectively. Slightly abusing notation we define the densities on $\R^k$ and set the remaining $d-k$ coordinates equal to zero. Denoting the Laplace operator by $\Delta:=\sum_{j=1}^k\partial_j^2$, we set
\begin{align*}
  \nu_0(x):=&\Big(1+\sum_{j=1}^k|x_j|^{2L}\Big)^{-1},\quad x=(x_1,\dots,x_k)^\top\in\R^k,\\
  \nu_n(x):=&\nu_0(x)+a\delta_n^{s-k}(\Delta K)(x/\delta_n),\quad x\in\R^k,
\end{align*}
for $L\in\N$ such that $2L>k+p\vee(-s)$ and $L>k$, some positive sequence $\delta_n\to0$, to be chosen later and a sufficiently small constant $a>0$. Since for any $l\in\mathbb{N}$ it holds $|x|^{l}|\Delta K(x)|\le C_{l}$ uniformly and due to the assumption $k\le s$, $\nu_0$ and $\nu_n$ are non-negative finite measures. In particular, they are L\'evy measures.

By construction $\nu_0\in\mathfrak S(s,p,C_\nu)$ for any $s>-2,p>0$ and some $C_\nu>0$ (by rescaling $C_\nu$ can be arbitrary). To verify that $\nu_n\in\mathfrak S(s,p,C_\nu)$ holds for some sufficiently small $a>0$ and for all $n\in\N$, we first note that \(\int_{\R^k}|x|^{-s}\delta_{n}^{s-k}|\Delta K|(x/\delta_{n})\d x=\int_{\R^k}|y|^{-s}|\Delta K|(y)\d y\) for $s\in(-2,0]$. In the case $s>0$ we use
\begin{align*}
  (1+|u|^2)^{s/2}\big|\F\big[\delta_n^{s-k}\Delta K(x/\delta_n)\big](u)\big|
  &=\delta_n^{s+2}(1+|u|^2)^{s/2}|u|^2\big|\F K(\delta_n u)\big|\\
  &\le\delta_n^{s+2}(1+\delta_n^{-2})^{s/2}\delta_n^{-2}\lesssim 1,
\end{align*}
owing to the compact support of $\F K$.

Now define the rank $k$ diagonal matrix $\Sigma_0=\operatorname{diag}(1,\dots,1,0,\dots,0)$ (i.e., $k$ ones followed by $d-k$ zeros) and its perturbation $\Sigma_n:=(1+2a\delta_n^{2+s})\Sigma_0$. Finally define 
\[
  Y_{t}^{(0)}=\Sigma_{0}\, W_{t}+Z^{(0)}_{t}\quad\text{and}\quad 
  Y_{t}^{(n)}=\Sigma_{n}\, W_{t}+Z_{t}^{(n)}
\]
with a Brownian motion $W_t$ and with $Z_{t}^{(0)}$ and $Z_{t}^{(n)}$ being compound Poisson processes independent of $W_t$, with jump measures $\nu_{0}$ and $\nu_{n}$, respectively. 

\emph{Step 2:} We now bound the $\chi^2$ distance of the observation laws $\mathbb P_0^{\otimes n}:=\mathbb P^{\otimes n}_{(\Sigma_0,\nu_0,\mathcal T)}$ and ${\mathbb P}_n^{\otimes n}:=\mathbb P^{\otimes n}_{(\Sigma_n,\nu_n,\mathcal T)}$. First we observe that both laws are equal on the last $d-k$ coordinates, namely being a Dirac measure in zero. Owing to the diffusion component, the marginals $\mathbb P_0$ and ${\mathbb P}_n$ admit Lebesgue densities on $\R^k$ denoted by $f_0$ and $f_n$, respectively (cf. \cite[Thm. 27.7]{sato1999levy}). Since the observations are i.i.d., $\chi^2({\mathbb P}_n^{\otimes n},{\mathbb P}_0^{\otimes n})$ is uniformly bounded in $n$, if
\begin{align*}
  n\chi^2(\mathbb P_n,{\mathbb P}_0)
  = n\int_{\R^k}\frac{|f_n(x)-f_0(x)|^2}{f_0(x)}\d x<c
\end{align*}
for some constant $c>0$. The density $f_0$ is given by $f_0(x)=\int_{\R^+}p_t(x)\pi(\d t)$, where $p_t$ denotes the density of $Y_t^{(0)}$. Since $Y^{0}$ is of compound Poisson type, its marginal density is given by the convolution exponential
\begin{align*}
  p_t(x)&=\mu_{0,t\Sigma_0}\ast\Big(e^{-t\nu_0(\R^k)}\sum_{j=0}^\infty \frac{t^j\nu_0^{\ast j}}{j!}\Big)(x)
  \ge te^{-t\nu_0(\R^k)}(\mu_{0,t\Sigma_0}\ast\nu_0)(x)
\end{align*}
with the density $\mu_{0,t\Sigma_0}$ of the $\mathcal N(0,t\Sigma_0)$-distribution. Using that there is some interval $[r,s]\subset (0,\infty)$ with $\pi([r,s])>0$ and that $\nu_0$ is independent of $t$, we obtain
\[
  f_0(x)\gtrsim \big( \mu_{0,r\Sigma_0}\ast\nu_0\big)(x)
    \gtrsim \Big(1+\sum_{j=1}^k|x_j|^{2L}\Big)^{-1},\quad x=(x_1,\dots,x_k)^\top\in\R^k.
\]
By Plancherel's identity we thus have
\begin{align*}
  \chi^2(\mathbb P_n,{\mathbb P}_0)
  &\le\int_{\R^k}\Big(1+\sum_{j=1}^k|x_j|^{2L}\Big)\big|f_n(x)-f_0(x)\big|^2\d x\\
  &\lesssim \|\phi_n-\phi_0\|_{L^2(\R^k)}^2+\sum_{j=1}^k\big\|\partial_j^L(\phi_n-\phi_0)\big\|_{L^2(\R^k)}^2,
\end{align*}
where $\phi_0$ and $\phi_n$ denote the characteristic functions of $\mathbb P_0$ and $\mathbb P_n$, respectively.

\emph{Step 3:} We have to estimate the distance of the characteristic functions. Let us denote the characteristic exponents of the L\'evy processes $Y_t^{(0)}$ and $Y_t^{(n)}$ (restricted on the first $k$ coordinates) by $\psi_0$ and $\psi_n$, respectively. Then,
\begin{align*}
  \psi_m(u)=&-\frac{1}{2}\langle u,\Sigma_mu\rangle+\int_{\R^k}\big(e^{i\langle u,x\rangle}-1-i\langle u,x\rangle\big)\nu_m(x)\d x,\quad m\in\{0,n\}
\end{align*}
Note that $\psi_m$ is real valued because $\nu_m$ is even. Using Taylor's formula, we obtain
\begin{align*}
  \phi_n(u)-\phi_0(u)
  &=\mathscr L(-\psi_n(u))-\mathscr L(-\psi_0(u))\\
  &=-\big(\psi_n(u)-\psi_0(u)\big)\int_0^1\mathscr L'\big(-\psi_0(u)-t(\psi_n(u)-\psi_0(u))\big)\d t.
\end{align*}
Defining $\Psi_{n,t}(u):=-\psi_0(u)-t(\psi_n(u)-\psi_0(u))$, we thus have 
\begin{align*}
  \partial_j^L(\phi_n-\phi_0)(u)=\sum_{r=0}^L\partial_j^r\big(\psi_n(u)-\psi_0(u)\big)\int_0^1\partial_j^{L-r}\mathscr L'(\Psi_{n,t}(u))\d t,
\end{align*}
where the partial derivatives of the composition $\mathscr L'\circ\Psi_{n,t}$ can be computed with Fa\`{a} di Bruno's formula. Since $\int(e^{i\langle u,x\rangle}-1-i\langle u,x\rangle)\Delta K(x/\delta_n)\d x=-\delta_n^{2+k}|u|^2\mathcal{F}K(\delta_n u)$, we have
\[
\psi_{n}(u)-\psi_{0}(u)=a\delta_{n}^{2+s}|u|^{2}(1-\mathcal{F}K(\delta_{n}u))
\]
and in particular $\Psi_{n,t}(u)=-\psi_0(u)(1+o(1))$ uniformly over $u$ and $t$ for $\delta_n\to0$.
Taking into account the properties 
\begin{align*}
&\mathcal{F}K(u)=0 \text{ for } |u|>2/\delta_n, 
\\
&\partial_j^r\mathcal{F}K(u)=0 \text{ for } |u|\le 1/\delta_n \text{ and } |u|>2/\delta_n, \quad r=1,\dots,L ,
\\
&|\partial_j^r\mathcal{F}K(u)|\le C \text{ for } 1/\delta_n<|u|\le 2/\delta_n, \quad r=0,1,\dots,L, 
\end{align*}
for all $j=1,\dots,k$ and some \(C>0\), we see that $\psi_n(u)-\psi_0(u)$ is zero for $|u|<1/\delta_n$ and that $|\partial_j\psi_m(u)|\lesssim 1+|u_j|,|\partial_j^r\psi_m(u)|\lesssim 1$ for $r=2,\dots,L$ and $m\in\{0,n\}$.
We conclude
\begin{align*}
  &\big\|\partial_j^L(\phi_n-\tilde\phi_n)\big\|_{L^2(\R^k)}^2
  \lesssim a\delta_n^{2s+4}\int_{|u|>1/\delta_n}|u|^4
  \sum_{r=0}^{L}\big|\mathscr L^{(1+r)}\big(-\psi_{n}(u)(1+o(1))\big)\big|^2|u|^{2r}\d u.
\end{align*}
Due to monotonicity of $\mathscr{L}'(-x)$ for $x>0$ and $\mathscr{L}^{(r+1)}(x)/\mathscr{L}^{(r)}(x)=O(1/|x|)$ for $|x|\to\infty$, $r=1,\dots,L$, the previous estimate and Step 2 yield as $n\to \infty$ 
\begin{align*}
n\chi^2(\mathbb P_n,{\mathbb P}_0)
& \lesssim an\delta_{n}^{2(s+2)}\int_{|u|>1/\delta_{n}}\left|\mathscr{L}'(-\psi_{0}(u))\right|^{2} |u|^4\d u
  \lesssim an\delta_{n}^{2s+4+4\eta-k}
\end{align*} 
if $4\eta>k$. Hence, $\chi^2(\mathbb P_n^{\otimes n},\mathbb P_0^{\otimes n})$ remains bounded for $\delta_n=n^{-1(2s+4+4\eta-k)}$ and with some sufficiently small $a>0$.

\emph{Step 4:} Noting that
\[
\big\|\Sigma_{0}-\tilde\Sigma_{n}\big\|_2=2\sqrt ka\delta_{n}^{s+2},
\]
the first lower bound in Theorem~\ref{thm:lowerBound} follows from Theorem 2.2 in \cite{tsybakov2009}. 

\emph{Step 5:}
For the second case, i.e., $k>2\eta$, we modify our construction as follows: We use the jump measures $\nu_0$ and $\nu_n$ from before, but only on the first derivative. We use the same rank $k$ diffusion matrix $\Sigma_0$ with $k$ ones on the diagonal, but choose the alternative as $\Sigma_n=\diag(1+2a\delta_n^{2+s},1,\dots,1,0,\dots,0)$ where the last $d-k$ entries are zero. Since the corresponding laws $\mathbb P_0$ and $\mathbb P_n$ are product measures which differ only on the first coordinate, the calculations from Step~2 and 3 yield
\begin{align*}
  n\chi^2(\mathbb P_n,\mathbb P_0)
  &\lesssim an\delta_n^{2(s+2)}\int_{\R\setminus[-1/\delta_n,1/\delta_n]}|\mathscr L'(-\psi_0(u))|^2|u|^4\d u\\
  &\le an\delta_n^{2(s+2)}\sup_{|u|>1/\delta_n}\big\{|\mathscr L'(-\psi_0(u))||u|^2\big\}\int_{\R}|\mathscr L'(-\psi_0(u))||u|^2\d u\\
  &\lesssim an\delta_n^{2(s+2)+2\eta}\int_\R(1+|u|^2)^{-\eta-1}|u|^2\d u,
\end{align*}
where the integral is finite by the assumption $\eta>1/2$. Hence, we have shown the second lower bound Theorem~\ref{thm:lowerBound}(i). 

\vspace*{1em}
The result in (ii) can be deduced analogously, choosing $k=1$. In Step~3 we obtain under the corresponding assumption on $\mathscr L$ that with some constant $c>0$
\begin{align*}
n\chi^2(\mathbb P_n,{\mathbb P}_0)
& \lesssim an\delta_{n}^{2(s+2)}\int_{|u|>1/\delta_{n}}\left|\mathscr{L}'(-\psi_{0}(u))\right|^{2} |u|^4(1+|u|^{2L})\d u\\
 & \lesssim an\delta_{n}^{2s+4}e^{-c\delta_n^{-2\eta}}
\end{align*} 
which remains bounded if $\delta_n\sim (\log n)^{-1/(2\eta)}$.
\qed

\subsection{The mixing case: Proof of Theorem \ref{thm:conMix}}\label{sec:proodmixing}
We denote again $\xi_U:=U^2\inf_{|u|\le U}|\mathscr L'(-\psi(u))|$ and recall the event $\mathcal H_{n,U}$ defined in Lemma~\ref{lem:linearisation}. Applying \eqref{eq:EstMix} and Lemmas~\ref{lem:ApproxError} and \ref{lem:linearisation}, we obtain on $\mathcal H_{n,U}$
\begin{align*}
  \|\mathcal R_n\|_\infty
  \le& 4\int_{\R^d}|u|^{-2}\big|\Re\big(\mathscr L^{-1}(\phi_n(u))-\mathscr L^{-1}(\phi(u))\big)\big|w_U(u)\d u+4\int_{\R^d}|u|^{-2}|\Psi(u)+\alpha|w_U(u)\d u\\
  \le& 4\int_{\R^d}\frac{|\phi_n(u)-\phi(u)|}{|u|^2|\mathscr L'(-\psi(u))|}w_U(u)\d u
  +16C_L\int_{\R^d}\frac{|\phi_n(u)-\phi(u)|^2}{|u|^2|\mathscr L'(-\psi(u))|^2}w_U(u)\d u\\
  &\qquad+4\int_{\R^d}|u|^{-2}|\Psi(u)+\alpha|w_U(u)\d u\\
  \le&C\Big( \xi_U^{-1}\|\phi_n-\phi\|_U+U^2\xi_U^{-2}\|\phi_n-\phi\|_U^2 +U^{-s-2}\Big)
\end{align*}
where the constant $C>0$ depends only on $w,C_L$ and $C_\nu$. Therefore,
\begin{align*}
  &\P\Big(\|\mathcal R_n\|_\infty\ge \frac{\kappa(\log U)^{\rho}}{\sqrt n\xi_U}+CU^{-s-2}\Big)\\
  \le& \P\Big(\sqrt n(\log U)^{-\rho}\|\phi_n-\phi\|_U\ge \frac{\kappa}{2C}\Big)
    +\P\Big(\big(\sqrt n(\log U)^{-\rho}\|\phi_n-\phi\|_U\big)^2\ge \frac{\sqrt n\kappa\xi_U}{2CU^2(\log U)^{\rho}\xi_U}\Big)+\P(\mathcal H_{n,U}^c).
\end{align*}
If $\kappa\le\sqrt n\xi_U(\log U)^{-\rho} U^{-2}$, we have 
\[
  \P(\mathcal H_{n,U}^c)\le\P\Big(\sqrt n(\log U)^{-\rho}\|\phi_n-\phi\|_U\ge \frac{2\kappa}{C_L}\Big)  
\]
Theorem~\ref{thm:concentration} yields
\[
  \P\Big(\|\mathcal R_n\|_\infty\ge \frac{\kappa(\log U)^{\rho}}{\sqrt n\xi_U}+CU^{-s-2}\Big)\lesssim e^{-c\kappa^2}+n^{-p/2}
\]
some $c>0$ and for any $n\in\N$ and $\kappa\in(\underline \xi\sqrt{d\log n},\overline\xi\sqrt n/\log^2n)$.\qed

\appendix

\section{Multivariate uniform bounds for the empirical characteristic function}\label{appendix}

\subsection{I.i.d. sequences}

Let us recall the usual multi-index notation. For a multi-index $\beta=(\beta_{1},\dots,\beta_{d})\in\N^{d}$,
a vector $x=(x_{1},\dots,x_{d})\in\R^{d}$ and a function $f\colon\R^{d}\to\R$
we write 
\begin{gather*}
|\beta|:=\beta_{1}+\dots+\beta_{d},\quad x^{\beta}:=x_{1}^{\beta_{1}}\cdots x_{d}^{\beta_{d}},\quad|x|^{\beta}:=|x_{1}|^{\beta_{1}}\cdots|x_{d}|^{\beta_{d}},\\
\partial^{\beta}f:=\partial_{1}^{\beta_{1}}\cdots\partial_{d}^{\beta_{d}}f.
\end{gather*}
We need a multivariate (straight
forward) generalisation of Theorem~4.1 by \citet{neumannReiss2009}.
For a sequence of independent random vectors $(Y_{j})_{j\ge1}\subset\R^{d}$
we define the empirical process corresponding to the empirical characteristic
function by 
\[
C_{n}(u):=\sqrt{n}(\phi_{n}(u)-\phi(u))=n^{-1/2}\sum_{j=1}^{n}\big(e^{i\langle u,Y_{j}\rangle}-\E[e^{i\langle u,Y_{1}\rangle}]\big),\quad u\in\R^{d},n\ge1.
\]

\begin{prop} \label{prop:uniformBound}Let $\beta\in\N^{d}$ be a
multi-index and let $Y_{1},\dots,Y_{n}\in\R^{d}$ be iid. $d$-dimensional
random vectors satisfying $\E[|Y_{1}|^{2\beta}|Y_{1}|^{\gamma}]<\infty$
for some $\gamma>0$. Using the weight function $w(u)=(\log(e+|u|)^{-1/2-\delta},u\in\R^{d}$,
for some $\delta>0$, there is a constant $C>0$ such that 
\[
\sup_{n\ge1}\E[\|w(u)\partial^{\beta}C_{n}(u)\|_{\infty}]\le C\sqrt{d}(\sqrt{\log d}+1).
\]
\end{prop} \begin{proof} The proof relies on a bracketing entropy
argument and we first recall some definitions. For two functions $l,u\colon\R^{d}\to\R$
a bracket is given by $[l,u]:=\{f\colon\R^{d}\to\R|l\le f\le u\}$.
For a set of functions $G$ the $L^{2}$-bracketing number $N_{[]}(\epsilon,G)$
denotes the minimal number of brackets $[l_{k},u_{k}]$ satisfying
$\E[(u_{k}(Y_{1})-l_{k}(Y_{1}))^{2}]\le\eps^{2}$ which are necessary
to cover $G$. The bracketing integral is given by 
\[
J_{[]}(\delta,G):=\int_{0}^{\delta}\sqrt{N_{[]}(\eps,G)}\d\eps.
\]
A function $F\colon\R^{d}\to\R$ is called envelop function of $G$
if $|f|\le F$ for any $f\in G$.\par Decomposing $C_{n}$ into the
real and the imaginary part, we consider the set $G_{\beta}:=\{g_{u}:u\in\R^{d}\}\cup\{h_{u}:u\in\R^{d}\}$
where 
\[
g_{u}\colon\R^{d}\to\R,y\mapsto w(u)\frac{\partial^{\beta}}{\partial u^{\beta}}\cos(\langle u,y\rangle),\quad h_{u}\colon\R^{d}\to\R,y\mapsto w(u)\frac{\partial^{\beta}}{\partial u^{\beta}}\sin(\langle u,y\rangle).
\]
Noting that $G_{\beta}$ has the envelop function $F(y)=|y|^{\beta}$,
Lemma~19.35 in \citet{vanderVaart1998} yields 
\[
\E[\|w(u)\partial^{\beta}C_{n}(u)\|_{\infty}]\lesssim J_{[]}(\E[F(Y_{1})^{2}],G_{\beta}).
\]
Since the real and the imaginary part can be treated analogously,
we concentrate in the following on $\{g_{u}:u\in\R^{d}\}$. Owing
to $|g_{u}(y)|\le w(u)|y|^{\beta}$, we have $\{g_{u}:|u|>B\}\subset[g_{0}^{-},g_{0}^{+}]$
for $g_{0}^{\pm}(y):=\pm\eps|y|^{\beta}$ and 
\[
B:=B(\eps):=\inf\{b>0:\sup_{|u|\ge b}w(u)\le\eps\}
\]
To cover $\{g_{u}:|u|\le B\}$, we define for some grid $(u_{j})_{j\ge1}\subset\R^{d}$
and $j\ge1$ 
\[
g_{j}^{\pm}(y):=\big(w(u_{j})\frac{\partial^{\beta}}{\partial u^{\beta}}\cos(\langle u_j,y\rangle)\pm\eps|y|^{\beta}\big)\1_{\{|y|\le M\}}\pm|y|^{\beta}\1_{\{|y|>M\}}
\]
with 
\[
M:=M(\eps):=\inf\big\{ m>0:\E[|Y_{1}|^{2\beta}\1_{\{|Y_{1}|>m\}}]\le\eps^{2}\big\}.
\]
We have $\E[|g_{j}^{+}(Y_{1})-g_{j}^{-}(Y_{1})|^{2}]\le4\eps^{2}(\E[|Y_{1}|^{2\beta}]+1)$
for $j\ge0$. Denoting the Lipschitz constant of $w$ by $L$, it
holds 
\begin{align*}
\big|w(u)\tfrac{\partial^{\beta}}{\partial u^{\beta}}\cos(\langle u,y\rangle)-w(u_{j})\tfrac{\partial^{\beta}}{\partial u^{\beta}}\cos(\langle u_{j},y\rangle)\big| & \le|y|^{\beta}\big(L+|y|\big)|u-u_{j}|.
\end{align*}
Therefore, $g_{u}\in[g_{j}^{-},g_{j}^{+}]$ if $(L+M)|u-u_{j}|\le\eps$.
Since any (Euklidean) ball in $\R^{d}$ with radius $B$ can be covered
with fewer than $(B/\tilde{\eps})^{d}$ cubes with edge length $2\tilde{\eps}$
and each of these cubes can be covered with a ball of radius $\sqrt{d}\tilde{\eps}$
(use $|\bull|\le\sqrt{d}\|\bull\|_{\ell^{\infty}}$), we choose $\tilde{\eps}=\eps d^{-1/2}/(L+M)$
to see that 
\[
N_{[]}(\eps,G_{\beta})\le2\Big(\frac{\sqrt{d}B(L+M)}{\eps}\Big)^{d}+2.
\]
By the choice of $w$ it holds $B\le\exp(\eps^{-1/(1/2+\delta)})$
and Markov's inequality yields $M\le(\eps^{-2}\E[|Y_{1}|^{2\beta}\|Y_{1}\|^{\gamma}])^{1/\gamma}$.
The bracketing entropy is thus bounded by 
\[
\log N_{[]}(\eps,G_{\beta})\lesssim d(\log d+\eps^{-1/(1/2+\delta)}+\log(\eps^{-2/\gamma-1}))\lesssim d(\log d+\eps^{-2/(1+2\delta)})
\]
and the entropy integral can be estimated by 
\[
J_{[]}(\E[F(Y_{1})^{2}],G_{\beta})\lesssim\sqrt{d}\Big(\sqrt{\log d}+\int_{0}^{\E[|Y_{1}|^{2\beta}]}\eps^{-1/(1+2\delta)}\d\eps\lesssim\sqrt{d}(\sqrt{\log d}+1).\tag*{\qedhere}
\]
\end{proof} 
Applying Talagrand's inequality, we conclude the following concentration result, see also Proposition 3.3 in \citet[Chap. 1]{belomestny2015estimation}.
\begin{thm}\label{thm:concPhi}
  Let $Y_{1},\dots,Y_{n}\in\R^{d}$ be i.i.d. $d$-dimensional
  random vectors satisfying $\E[|Y_{1}|^{\gamma}]<\infty$ for some
  $\gamma>0$. For any $\delta>0$ there is some numerical constant $c>0$ independent
  of $d,n,U$ such that 
  \[
  \P\big(\sup_{|u|\le U}|C_{n}(u)|\ge \kappa\big)\le2e^{-c\kappa^2},
  \]
  for any $\kappa\in[\sqrt d(\sqrt{\log d}+1)(\log U)^{1/2+\delta},\sqrt{n}]$.
\end{thm}
\begin{proof} 
  We will apply Talagrand's inequality in Bousquet's
  version (cf. \citet{massart2007}, (5.50)). Let $T\subset[-U,U]^{d}$
  be a countable index set. Noting that $Z_{j,u}:=n^{-1/2}(e^{i\langle u,Y_{j}\rangle}-\E[e^{i\langle u,Y_{1}\rangle}])$
  are centred and i.i.d. random variables satisfying $|Z_{k,u}|\le2n^{-1/2}$,
  for all $u\in T,k=1,\dots,n$, as well as $\sup_{u\in T}\Var(\sum_{k=1}^{n}Z_{k,u})\le1$,
  we have for all $\kappa>0$ 
  \[
  P\Big(\sup_{u\in T}\Big|\sum_{k=1}^{n}Z_{k,u}\Big|\ge4\E\Big[\sup_{u\in T}\Big|\sum_{k=1}^{n}Z_{k,u}\Big|\Big]+\sqrt{2\kappa}+\frac{4}{3}n^{-1/2}\kappa\Big)\le2e^{-\kappa}.
  \]
  Proposition~\ref{prop:uniformBound} yields $\E[\sup_{u\in T}|\sum_{k=1}^{n}Z_{k,u}|]\le C(\log U)^{1/2+\delta}\sqrt{d}(\sqrt{\log d}+1)$
  for some $\delta,C>0$. Choosing $T=\mathbb{Q^{d}}\cap[-U,U]^{d}$,
  continuity of $u\mapsto Z_{j,u}$ yields 
  \[
  \P\big(\sup_{|u|\le U}|C_{n}(u)|\ge C\sqrt{d\log d}(\log U){}^{1/2+\delta}+\sqrt{2\kappa}+\frac{4}{3}n^{-1/2}\kappa\Big)\le2e^{-\kappa}.\tag*{\qedhere}
  \]
\end{proof}

\subsection{Mixing sequences}

If the sequence is not i.i.d., but only $\alpha$-mixing, there is no Talagrand-type inequality to work with. At least \citet{MerlevedeEtAl2009} have proven to following Bernstein-type concentration result. The bound of the constant $v^{2}$ has been derived by \citet{belomestny2011}.
\begin{prop}[\citet{MerlevedeEtAl2009}] \label{EIB} Let $(X_{k},k\geq1)$ be a strongly mixing sequence of centred real-valued random variables
on the probability space $(\Omega,\mathcal{F},P)$ with mixing coefficients satisfying 
\begin{eqnarray}
\alpha(k)\le\bar{\alpha}_{0}\exp(-\bar{\alpha}_{1}k),\quad k\geq1,\quad\bar{\alpha}_{0},\bar{\alpha}_{1}>0.\label{ALPHA_EXP_DECAY}
\end{eqnarray}
If $\sup_{k\geq1}|X_{k}|\le M$ a.s., then there is a positive constant
 depending on $\bar{\alpha}_{0},$ and $\bar{\alpha}_{1}$ such that
\[
\P\Big(\sum_{k=1}^{n}X_{k}\geq\zeta\Big)\le\exp\left[-\frac{C\zeta^{2}}{nv^{2}+M^{2}+M\zeta\log^{2}(n)}\right].
\]
for all $\zeta>0$ and $n\geq4,$ where 
\begin{eqnarray*}
v^{2}:=\sup_{k}\Big(\E[X_{k}]^{2}+2\sum_{j\geq k}\Cov(X_{k},X_{j})\Big).
\end{eqnarray*}
Morover, there is a constant $C'>0$ such that 
\begin{equation}\label{eq:boundV2}
v^{2}\le\sup_{k}\E[X_{k}]^2+C'\sup_{k}\E\left[X_{k}^{2}\log^{2(1+\varepsilon)}\left(|X_{k}|^{2}\right)\right],
\end{equation}
provided the expectations on the right-hand side are finite. \end{prop}

Let $Z_{j}$, $j=1,\ldots,n,$ be a sequence of random vectors in $\mathbb{R}^{d}$ with corresponding empirical characteristic function
\begin{align*}
\phi_{n}(u)=\frac{1}{n}\sum_{j=1}^{n}\exp(i\langle u,Z_{j}\rangle),\quad u\in\R^d.
\end{align*}
\begin{thm}\label{thm:concentration} 
  Suppose that the following assumptions hold: \begin{enumerate} 

  \item[(AZ1)] The sequence $Z_{j}, j=1,\ldots,n$, is strictly stationary and $\alpha$-mixing with mixing coefficients $(\alpha_{Z}(k))_{k\in\mathbb{N}}$ satisfying 
  \begin{eqnarray*}
  \alpha_{Z}(k)\le\bar{\alpha}_{0}\exp(-\bar{\alpha}_{1}k),\quad k\in\mathbb{N},
  \end{eqnarray*}
  for some $\bar{\alpha}_{0}>0$ and $\bar{\alpha}_{1}>0.$
  
  \item[(AZ2)] It holds $\E[|Z_{j}|^{p}]<\infty$ for some $p>2$.
  \end{enumerate} 
  For arbitrary $\delta>0$ let the weighting function $w\colon\R^{d}\to\R_{+}$ be given by 
  \begin{equation}
  w(u)=\log^{-(1+\delta)/2}(e+|u|),\quad u\in\mathbb{R}.\label{w}
  \end{equation}
  Then there are $\underline{\xi},\overline{\xi}>0$ depending only
  on the characteristics of $Z$ and $\delta$, such that for any $n\in\N$
  and for all $\xi\in(\underline{\xi}\sqrt{d\log n},\overline{\xi}\sqrt{n}/\log^{2}n)$
  the inequality 
  \begin{align}
  \P\Big(\sqrt{n}\sup_{u\in\mathbb{R}^{d}}w(u)|\phi_{n}(u)-\E[\phi_{n}(u)]|>\xi\Big) & \le C(e^{-c\xi^{2}}+n^{-p/2})\label{MINEQ}
  \end{align}
  holds for constants $C,c>0$ independent of $\xi,n$ and $d$. 
\end{thm}

\begin{proof} We introduce the empirical process 
\begin{eqnarray*}
\mathcal{W}_{n}(u)=\frac{1}{n}\sum_{j=1}^{n}w(u)\Big(\exp(i\langle u,Z_{j}\rangle)-\E[\exp(i\langle u,Z_{j}\rangle)]\Big),\quad u\in\R^{d}.
\end{eqnarray*}
Consider the sequence $A_{k}=e^{k},\, k\in\mathbb{N}$. As discussed in the proof of Proposition~\ref{prop:uniformBound},
we can cover each ball $\{U\in\R^{d}:|u|<A_{k}\}$ with $M_{k}=\left(d^{1/2}A_{k})/\gamma\right)^{d}$ small balls with radius $\gamma>0$ and some centres $u_{k,1},\ldots,u_{k,M_{k}}$.

Since $|\mathcal{W}_{n}(u)|\le2w(u)\downarrow0$ as $|u|\to\infty$,
there is for any $\lambda>0$ some finite integer $K=K(\lambda)$
such that $\sup_{|u|>A_{k}}|\mathcal{W}_{n}(u)|<\lambda$. For $|u|\le A_{k}$
we use the bound 
\begin{align*}
\max_{k=1,\ldots,K}\sup_{A_{k-1}<|u|\le A_{k}}|\mathcal{W}_{n}(u)|\le & \max_{k=1,\ldots,K}\max_{|u_{k,m}|>A_{k-1}}|\mathcal{W}_{n}(u_{k,m})|\\
 & +\max_{k=1,\ldots,K}\max_{1\le m\le M_{k}}\sup_{u:|u-u_{k,m}|\le\gamma}|\mathcal{W}_{n}(u)-\mathcal{W}_{n}(u_{k,m})|
\end{align*}
to obtain 
\begin{align}
&\P\Big(\sup_{u\in\R^{d}}|\mathcal{W}_{n}(u)|>\lambda\Big)\notag\\
&\qquad\le\sum_{k=1}^{K}\sum_{|u_{k,m}|>A_{k-1}}\P(|\mathcal{W}_{n}(u_{k,m})|>\lambda/2)
+\P\Big(\sup_{|u-v|<\gamma}|\mathcal{W}_{n}(v)-\mathcal{W}_{n}(u)|>\lambda/2\Big).\label{DEC1}
\end{align}
It holds for any $u,v\in\mathbb{R}^{d}$
\begin{align}
  |\mathcal{W}_{n}(v)-\mathcal{W}_{n}(u)|\le & 2|w(v)-w(u)|+\frac{1}{n}\sum_{j=1}^{n}\left|\exp(i\langle v,Z_{j}\rangle)-\exp(i\langle u,Z_{j}\rangle)\right|\nonumber \\
  & +\frac{1}{n}\sum_{j=1}^{n}\left|\E\left[\exp(i\langle v,Z_{j}\rangle)-\exp(i\langle u,Z_{j}\rangle)\right]\right|\nonumber \\
  \le & |u-v|\Big(2L_{w}+\frac{1}{n}\sum_{j=1}^{n}|Z_{j}|+\frac{1}{n}\sum_{j=1}^{n}\E[|Z_{j}|]\Big),\label{WNDIFF}
\end{align}
where $L_{\omega}$ is the Lipschitz constant of $w$. Markov's inequality and the moment inequality by \citet{yokoyama1980}
yield 
\begin{align*}
\P\Big(\frac{1}{n}\sum_{j=1}^{n}(|Z_{j}|-\E[|Z_{j}|])>c\Big)\le & c^{-p}n^{-p}\E\Big[\Big|\sum_{j=1}^{n}(|Z_{j}|-\E[|Z_{j}|])\Big|^{p}\Big]\\
\le & C_{p}(\alpha)c^{-p}n^{-p/2}
\end{align*}
for any \ensuremath{c>0}
 and where \ensuremath{C_{p}(\alpha)}
 is some constant depending on \ensuremath{p}
 and \ensuremath{\alpha=(\bar{\alpha}_{0},\bar{\alpha}_{1})}
 from Assumption (AZ1). In combination with \eqref{WNDIFF} we
obtain 
\begin{align*}
\P\Big(\sup_{|u-v|<\gamma}|\mathcal{W}_{n}(v)-\mathcal{W}_{n}(u)|>\lambda/2\Big)\le & \P\Big(\frac{1}{n}\sum_{j=1}^{n}(|Z_{j}|-\E[|Z_{j}|])>\frac{\lambda}{2\gamma}-2(L_{w}+\E[|Z_{1}|])\Big)\notag\\
\le & C_{p}(\alpha)n^{-p/2}\Big(\frac{\lambda}{2\gamma}-2(L_{w}+\E[|Z_{1}|])\Big)^{-p}.
\end{align*}
Setting $\gamma=\lambda/(6(L_{\omega}+\E[|Z_{1}|]))$, we conclude 
\begin{equation*}
P\Big(\sup_{|u-v|<\gamma}|\mathcal{W}_{n}(v)-\mathcal{W}_{n}(u)|>\lambda/2\Big)\le B_{1}n^{-p/2}\label{LINEQ}
\end{equation*}
with some constant $B_{1}$ depending neither on $\lambda$ nor $n$.
 
We turn now to the first term on the right-hand side of \eqref{DEC1}.
If $|u_{k,m}|>A_{k-1}$, then it follows from Proposition~\ref{EIB} 
\begin{align*}
 & \P\big(\left|\Re(\mathcal{W}_{n}(u_{k,m}))\right|>\lambda/4\big)
 \le2\exp\Big(-\frac{B_{2}\lambda^{2}n}{w^{2}(A_{k-1})\log^{2(1+\delta)}(w(A_{k-1}))+\lambda\log^{2}(n)w(A_{k-1})}\Big),
\end{align*}
with some constant $B_{2}>0$ depending only on the characteristics of the process $Z$ and $\delta>0$. The same bound holds true for $\Im(\mathcal{W}_{n}(u_{k,m}))$. Choosing $\lambda=\xi n^{-1/2}$ for any $0<\xi\lesssim\sqrt{n}/\log^{2}(n)$
and taking into account the choice of $\gamma$ from above, we get
\begin{align*}
 \sum_{|u_{k,m}|>A_{k-1}}\P(|\mathcal{W}_{n}(u_{k,m})|>\lambda/2)
 &\le4\Big(\frac{d^{1/2}A_{k}}{\gamma}\Big)^{d}\exp\Big(-\frac{B_{3}\lambda^{2}n}{w^{2(1+\delta)}(A_{k-1})+\lambda\log^{2}(n)w(A_{k-1})}\Big)\\
 &\lesssim d^{d/2}A_{k}^{d}n^{d/2}\xi^{-d}\exp\Big(-\frac{B\xi^{2}}{w^{2(1+\delta)}(A_{k-1})}\Big)
\end{align*}
with positive constants $B_{3},B$. Fix \ensuremath{\theta>0}
 such that \ensuremath{B\theta=2d}
 and compute 
\begin{align*}
\sum_{|u_{k,m}|>A_{k-1}}\P(\sqrt{n}|\mathcal{W}_{n}(u_{k,m})|>\xi/2)\lesssim & d^{d/2}\xi^{-d}e^{dk-\theta B(k-1)}n^{d/2}e^{-B(k-1)(\xi^{2}-\theta)}\\
\le & d^{d/2}e^{2d}\xi^{-d}e^{k(d-\theta B)}e^{-B(k-1)(\xi^{2}-\theta)+d\log(n)/2}.
\end{align*}
If $\xi^{2}>\theta$ we obtain for any $K>0$ 
\begin{align*}
\sum_{k=2}^{K}\sum_{|u_{k,m}|>A_{k-1}}\P(\sqrt n|\mathcal{W}_{n}(u_{k,m})|>\xi/2) & \lesssim(d^{1/2}e^{4})^{d}\xi^{-d}e^{-(B\xi^{2}-d\log(n)/2)}.
\end{align*}
On the interval $\xi\in(\underline{\xi}\sqrt{d\log n},\overline{\xi}\sqrt{n}/\log^{2}n)$
for appropriate $\underline{\xi},\overline{\xi}>0$, we thus get \eqref{MINEQ}. \end{proof}

 \bibliographystyle{apalike}
\bibliography{reference-2-1}

\end{document}